\documentclass[10pt]{amsart}
\usepackage{amsmath}
\usepackage{amscd,amsthm,amssymb,amsfonts}
\usepackage{mathrsfs}
\usepackage{dsfont}
\usepackage{stmaryrd}
\usepackage{euscript}
\usepackage{expdlist}
\usepackage{enumerate}

\theoremstyle{plain}
\newtheorem{thm}{Theorem}[section]

\newtheorem*{thm*}{Theorem}

\newtheorem{lm}[thm]{Lemma}
\newtheorem{cor}[thm]{Corollary}
\newtheorem*{cor*}{Corollary}
\newtheorem{prop}[thm]{Proposition}
\newtheorem*{conj*}{Conjecture}

\theoremstyle{remark}

\theoremstyle{definition}
\newtheorem*{defn*}{Definition}
\newtheorem{Remark}[thm]{Remark}
\newtheorem{I_Remark*}{Remark}
\newtheorem{defn}[thm]{Definition}

\newcommand{\nc}{\newcommand}

\newcommand{\beq}{\begin{equation}}
\newcommand{\eeq}{\end{equation}}
\newcommand{\bpmx}{\begin{pmatrix}}
\newcommand{\epmx}{\end{pmatrix}}
\newcommand{\bbmx}{\begin{bmatrix}}
\newcommand{\ebmx}{\end{bmatrix}}
\newcommand{\wh}{\widehat}

\newcommand{\beqcd}[1]{\begin{equation*}\label{#1}\tag{#1}}
\newcommand{\eeqcd}{\end{equation*}}

\numberwithin{equation}{section}

\def\parref#1{\ref{#1}}
\def\thmref#1{Theorem~\parref{#1}}

\def\propref#1{Proposition~\parref{#1}}
\def\corref#1{Corollary~\parref{#1}}

\def\lmref#1{Lemma~\parref{#1}}

\def\subsubsecref#1{\S\parref{#1}}

\def\makeop#1{\expandafter\def\csname#1\endcsname
  {\mathop{\rm #1}\nolimits}\ignorespaces}

\makeop{Hom}   \makeop{End}   \makeop{Aut}   
\makeop{Pic} \makeop{Gal}       \makeop{Div} \makeop{Lie}
\makeop{PGL}   \makeop{Corr} \makeop{PSL} \makeop{sgn} \makeop{Spf}
 \makeop{Tr} \makeop{Nr} \makeop{Fr} \makeop{disc}
\makeop{Proj} \makeop{supp} \makeop{ker}   \makeop{Im} \makeop{dom}
\makeop{coker} \makeop{Stab} \makeop{SO} \makeop{SL} \makeop{SL}
\makeop{Cl}    \makeop{cond} \makeop{Br} \makeop{inv} \makeop{rank}
\makeop{id}    \makeop{Fil} \makeop{Frac}  \makeop{GL} \makeop{SU}
\makeop{Trd}   \makeop{Sp} \makeop{Tr}    \makeop{Trd} \makeop{Res}
\makeop{ind} \makeop{depth} \makeop{Tr} \makeop{st} \makeop{Ad}
\makeop{Int} \makeop{tr}    \makeop{Sym} \makeop{can} \makeop{SO}
\makeop{torsion} \makeop{GSp} \makeop{Tor}\makeop{Ker} \makeop{rec}
\makeop{Ind} \makeop{Coker}
 \makeop{vol} \makeop{Ext} \makeop{gr} \makeop{ad}
 \makeop{Gr}\makeop{corank} \makeop{Ann}
\makeop{Hol} 
\makeop{Fitt} \makeop{Mp} \makeop{CAP}

\DeclareMathAlphabet{\mathpzc}{OT1}{pzc}{m}{it}

\def\makebb#1{\expandafter\def
  \csname bb#1\endcsname{{\mathbb{#1}}}\ignorespaces}
\def\makebf#1{\expandafter\def\csname bf#1\endcsname{{\bf
      #1}}\ignorespaces}
\def\makegr#1{\expandafter\def
  \csname gr#1\endcsname{{\mathfrak{#1}}}\ignorespaces}
\def\makescr#1{\expandafter\def
  \csname scr#1\endcsname{{\EuScript{#1}}}\ignorespaces}
\def\makecal#1{\expandafter\def\csname cal#1\endcsname{{\mathcal
      #1}}\ignorespaces}

\def\doLetters#1{#1A #1B #1C #1D #1E #1F #1G #1H #1I #1J #1K #1L #1M
                 #1N #1O #1P #1Q #1R #1S #1T #1U #1V #1W #1X #1Y #1Z}
\def\doletters#1{#1a #1b #1c #1d #1e #1f #1g #1h #1i #1j #1k #1l #1m
                 #1n #1o #1p #1q #1r #1s #1t #1u #1v #1w #1x #1y #1z}
\doLetters\makebb   \doLetters\makecal  \doLetters\makebf
\doLetters\makescr
\doletters\makebf   \doLetters\makegr   \doletters\makegr

\normalsize

\makeop{Ram} \makeop{Rep} \makeop{mass}

\makeop{Bl}

\def\cB{\EuScript B}
\def\cD{\mathcal D}
\def\cE{{\mathcal E}}

\def\cL{{\mathcal L}}
\def\cH{{\mathcal H}}

\def\cK{{\mathcal K}}  

\def\cO{\mathcal O}
\def\cS{{\mathcal S}}

\def\cW{{\mathcal W}}

\def\cV{{\mathcal V}}

\def\cT{\mathcal T}

\def\bff{\mathbf f}

\def\sI{\mathscr I}

\def\sA{\mathscr A}

\def\sR{\mathscr R}

\def\bbI{\mathbb I}

\newcommand{\Z}{\mathbf Z}
\newcommand{\Q}{\mathbf Q}
\newcommand{\R}{\mathbf R}
\newcommand{\C}{\mathbf C}
\newcommand{\A}{\mathbf A}    

\def\etale{{\'{e}tale }}

\def\ot{\otimes}

\def\hookto{\hookrightarrow}
\def\longto{\longrightarrow}
\def\ol{\overline}  \nc{\opp}{\mathrm{opp}} \nc{\ul}{\underline}

\def\XYmatrix{\xymatrix@M=8pt} 
\def\ncmd{\newcommand}
\ncmd{\xysubset}[1][r]{\ar@<-2.5pt>@{^(-}[#1]\ar@<2.5pt>@{_(-}[#1]}
\ncmd{\XYmatrixc}[1]{\vcenter{\XYmatrix{#1}}}
\ncmd{\xyto}[1][r]{\ar@{->}[#1]}
\ncmd{\xyinj}[1][r]{\ar@{^(->}[#1]}
\ncmd{\xysurj}[1][r]{\ar@{->>}[#1]}
\ncmd{\xyline}[1][r]{\ar@{-}[#1]}
\ncmd{\xydotsto}[1][r]{\ar@{.>}[#1]}
\ncmd{\xydots}[1][r]{\ar@{.}[#1]}
\ncmd{\xyleadsto}[1][r]{\ar@{~>}[#1]}
\ncmd{\xyeq}[1][r]{\ar@{=}[#1]} \ncmd{\xyequal}[1][r]{\ar@{=}[#1]}
\ncmd{\xyequals}[1][r]{\ar@{=}[#1]}
\ncmd{\xymapsto}[1][r]{l\ar@{|->}[#1]}\ncmd{\xyimplies}[1][r]{\ar@{=>}[#1]}
\ncmd{\xyiso}{\ar[r]_-{\sim}}
\def\injxy{\ar@{^(->}}

\newcommand{\pMX}[4]{\begin{pmatrix}
{#1}& {#2}\\
{#3}&{#4}\end{pmatrix} }

\newcommand{\seesaw}[4]{{#1}\ar@{-}[rd]\ar@{-}[d]&{#2}\ar@{-}[d]\\
{#3}\ar@{-}[ru]&{#4}}

\def\x{{\times}}

\newcommand\stt[1]{\left\{#1\right\}}

\renewcommand\pmod[1]{\,(\mbox{mod }{#1})}

\usepackage{color}
\usepackage{hyperref}

\numberwithin{equation}{section}

\def\G{{\rm GL}_2}
\def\B{{\rm B}}
\def\N{{\rm N}}

\def\x{\times}
\def\h{\hat}
\def\b{\bar}
\def\t{\tilde}
\def\bt{\boxtimes}

\def\itPi{\mathit \Pi}

\def\Z{\mathbb{Z}}
\def\Q{\mathbb{Q}}
\def\R{\mathbb{R}}
\def\C{\mathbb{C}}
\def\A{\mathbb{A}}

\def\sI{\mathscr{I}}
\def\sR{\mathscr{R}}

\title[Special value formula and restricted $L^2$-norm]
{Special value formula for the twisted triple product $L$-function and an application to the 
restricted $L^2$-norm problem}
\author{Yao Cheng}
\date{\today}
\address{Institute of Mathematics, Academia Sinica, 6F, Astronomy-Mathematics Building, No. 1, Sec. 4, Roosevelt Road, 
Taipei 106319, TAIWAN}
\email{briancheng@gate.sinica.edu.tw}

\begin{document}
\maketitle

\begin{abstract}
We establish explicit Ichino's formulae for the central values of the 
triple product $L$-functions with emphasis on the calculations for the real place. 
The key ingredient for our computations is \propref{P:regularization} 
which generalizes a result in \cite{MV2010}. 
As an application we prove the optimal upper bound of a sum of restricted $L^2$-norms 
of the $L^2$-normalized newforms on certain quadratic extensions with prime level and bounded
spectral parameter following the methods in \cite{Blomer2013}.
\end{abstract}
\tableofcontents

\section{Main results}\label{S:main result}
The aim of this article is to establish explicit Ichino's formulae for the central values of the 
triple product $L$-functions with emphasis on the calculations for the real place. 
The key ingredient for our computations is \propref{P:regularization} 
which generalizes a result in \cite[Lemma 3.4.2]{MV2010}. 
As an application, we prove the optimal upper bound of a sum of restricted  
$L^2$-norms of the $L^2$-normalized newforms on certain quadratic extensions 
with prime level and bounded spectral parameter following the methods of Blomer in
\cite{Blomer2013}.  

\subsection{Special value formula}
In this introduction, we state our special value formulae (\thmref{T: special value formula})
in the following cases for the sake of simplicity.  
Fix a quadratic field extension $\cK$ over $\Q$ with the discriminant $\cD$. 
We assume that the (narrow) classic number of $\cK$ is $1$ according to the sign of $\cD$. 
Let $q$ be a prime number. 
Denote by $L^2(X_0(q))$ the $L^2$-space equipped with the inner product
\begin{equation}\label{E:inner product for X_0(q)}
\langle f,g\rangle=\int_{X_0(q)}f(\tau)\ol{g(\tau)}\,d\mu(\tau)
\quad (d\mu(\tau):=\frac{dxdy}{y^2},\,\,\tau=x+iy).
\end{equation}
Here $X_0(q):=\Gamma_0(q)\backslash\frak{H}$. Denote by $\cB_q$ (resp. $\cB_1$) an orthonormal basis
of cuspidal Hecke-Maass newforms for $\Gamma_0(q)$ (resp. ${\rm SL}_2(\Z)$) with
respect to the inner product \eqref{E:inner product for X_0(q)}. 
Given $f\in\cB_q$, we denote by $f_\cK$ the unique $L^2$-normalized newform on $\cK$
associated to $f$ via the base change lift (see \S\ref{SS:base change}). 
Then $f_\cK$ is an analytic function on $\frak{H}_{\cD}$, where 
$\frak{H}_\cD$ is the two copies of the upper half plane $\frak{H}$ when $\cD>0$, and is the
hyperbolic three space when $\cD<0$. In any case, the upper half plane $\frak{H}$ sits naturally inside
$\frak{H}_\cD$ as in \eqref{E:from H to H'}. We denote by $f_\cK|_{\frak{H}}$ the restriction of
$f_\cK$ to $\frak{H}$. One has that $f_\cK|_{\frak{H}}$ is $\Gamma_0(q)$-invariant on the left. 

If $g$ is a cuspidal Hecke-Maass newform for $\Gamma_0(q)$, then we let $w_g\in\stt{\pm 1}$ be the 
eigenvalue of $g$ under the Atkin-Lehner involution given by \eqref{E:Atkin-Lehnner}.
Let $\delta_\cD=0$ if $\cD>0$ and $\delta_\cD=1$ if $\cD<0$. 

\begin{thm}\label{T:special case of special formula}
Suppose that $q$ is coprime to $\cD$. 
Let $f\in\cB_q$ and $g\in\cB_q\cup\cB_1$. 
If $g\in\cB_q$, then we assume that $w_g=1$ (resp. $w_g=-1$) when $q$ 
is split (resp. inert) in $\cK$. Let $c=0$ or $1$ according to $g\in\cB_q$ or $g\in\cB_1$ respectively. Then
\[
|\langle f_\cK|_\frak{H},g\rangle|^2
=
2^{-1}(8\pi^{-1})^{\delta_\cD}q^{-2}(1+q^{-1})^{-c}|\cD|^{-3/2}
\frac{\Lambda(\frac{1}{2},{\rm As}f_\cK \x g)}{\Lambda(1,{\rm Ad}^2 f_\cK)\Lambda(1,{\rm Ad}^2 g)}.
\]
Here the $\Lambda$-functions
\footnote{
In this article, we use the convention that automorphic $L$-functions $L(s,\pi)$ refer to the finite part of the 
$L$-function, omitting the $L$-factors at the archimedean places. On the other hand, we will write $\Lambda(s,\pi)$ 
when such factors are included.}
are the complete $L$-functions defined from the Galois theoretic side.
\end{thm}

\begin{Remark}\noindent
\begin{itemize}
\item[(1)]
The assumptions on $w_g$ imply that the local root number of the triple product $L$-function is $1$ at 
each place of $\Q$. Indeed, by the results of \cite{Prasad1990}, \cite{Prasad1992}, \cite{Rama2000} and \cite{CCI2020},
the local root number is related to the existence of certain trilinear form and the assertion is clear except the root number
at $q$ when $g\in\cB_q$, in which case one can use the relation \eqref{E:Atkin-Lehner eigenvalue for special} and 
\cite[Proposition 8.6]{Prasad1990}, \cite[Remark 4.1.2]{Prasad1992} to find that it is $w_g$ or $-w_g$ according 
to $q$ is split or $q$ is inert in $\cK$, respectively. For more details, see \S\ref{SSS:root number}.
\item[(2)]
\thmref{T:special case of special formula} is a special version of \thmref{T: special value formula}, which does 
not require $q$ to be coprime to $\cD$.
\end{itemize}
\end{Remark}

\thmref{T:special case of special formula} follows from establishing explicit 
Ichino's formula (\cite{Ichino2008}), which relates the period integrals of triple products of 
automorphic forms on certain quaternion algebras along the diagonal cycles 
with the central values of the triple product $L$-functions for $\GL_2$ 
together with product of the local period integrals over all places. 
To derive explicit Ichino's formulae, we need to carry 
out explicit computations for the local period integrals. 
In the literature, the local period integrals for non-archimedean
places have been computed by many authors (\cite{Nelson2011}, \cite{Nelson2014}, \cite{Hu2017}, 
\cite{ML2017}). In this article; however, we pay our attention to the calculations 
for the real place. In particular, by combining with the results in \cite{CC2019} and  \cite{S-YChen2019} , 
the explicit computations for the real place are almost completed. The key ingredient for our computations 
is \propref{P:regularization} for the archimedean places, which reduces 
the calculations of  local period integrals to that of certain local Rankin-Selberg integrals as in
\cite{JLbook2}, \cite{Flicker1988} and \cite{Kable2004}. 
We point out that in \cite{ML2017}, he used an analogue of \propref{P:regularization} to 
compute the local period integrals for the non-archimedean places under quite 
general settings. The computations for the real place were also treated in \cite{Woodbury2016}
with different setups. 

\subsection{Application to the restricted $L^2$-norm problem}
Explicit special value formulae for the triple product $L$-functions have
applications to both the algebraic number theory and the analytic number theory. 
In addition to the articles mentioned in the previous paragraph, there are also 
\cite{Garrett1987}, \cite{Orloff1987}, \cite{GrossKudla1992} and \cite{Watson2008}.
In this article, we give an application to the optimal upper bound of the average of the restricted $L^2$-norm 
\[
\|f_\cK|_\frak{H}\|_2^2:=\langle f_\cK|_\frak{H},f_\cK|_\frak{H}\rangle
\]
in the level aspect. Indeed, by combining \thmref{T:special case of special formula} with
the arguments in \cite{Blomer2013}, we immediately obtain:

\begin{thm}\label{T:main result}
Suppose that $q$ is coprime to $\cD$.
Fix any real number $T>1$ and $\epsilon>0$. Then
\[
\sum_{f\in\cB_q,\,t_f\ll T}
\|f_\cK|_{\frak{H}}\|_2^2\ll_{\cD,T,\epsilon}q^\epsilon.
\]
Here $t_f$ is the spectral parameter of $f$ given by \eqref{E:spectral parameter}. 
\end{thm}

We remark here that the proof \thmref{T:main result} depends on a bound $\alpha$ toward the 
Ramanujan conjecture for ${\rm GL}_2$. By a result of Kim in \cite{Kim2003}, one can take 
$\alpha=7/64$. In this article, the bound $\alpha<1/6$ is enough to obtain 
\thmref{T:main result}. 
Observe that \thmref{T:main result} implies the best individual bound 
(in the sense of assuming the Lindelof hypothesis) for almost all $f\in\cB_q$ 
appearing in the sum. By Parseval's identity and
\thmref{T:special case of special formula}, the restricted $L^2$-norm $\|f_\cK|_{\frak{H}}\|_2^2$ is
roughly given by
\[
\|f_\cK|_{\frak{H}}\|_2^2
\approx
\frac{1}{q^2}
\sum_{g\in\cB_q,\,t_g\ll 1}
L(1/2,{\rm As}f_\cK\x g).
\]
Weyl's law tells us that the sum of the RHS has $O(q)$ terms, so that the Lindelof hypothesis for 
the $L$-functions would imply that $\|f_\cK|_{\frak{H}}\|_2\ll q^{-1/2+\epsilon}$. 
The Lindelof hypothesis for the $L$-functions seems far beyond reach with present technology;
however, we have

\begin{cor}
Suppose that $q$ is coprime to $\cD$.
For any $\delta>0$, the bound $\|f_\cK|_{\frak{H}}\|_2\ll_{\cD,T,\delta} q^{-1/2+\delta}$ holds for all but
$O(q^{1-2\delta+\epsilon})$ terms of $f\in\cB_q$ occurring in the sum in \thmref{T:main result}.
\end{cor}

Our motivation for bounding the $L^p$-norms of functions on Riemannian manifolds or 
their restriction to submanifolds comes from quantum chaos.
In the arithmetic setting, these manifolds have
additional symmetries, such as a commutative algebra of Hecke operators commuting with the Laplacian.
Therefore one can consider joint eigenfunctions which may rule out high multiplicity eigenspaces.   
One searches for the bound of these ($L^2$-normalized) eigenfunctions in the different aspects, 
i.e. in the spectral aspect or in the level aspect or in both. In the literature, there are many 
results on bounding sup-norms of eigenfunctions on the entire domain. In particular, there is a 
complete list of references in the introduction of \cite{Saha2017}. In contrast, there are only a
small handful results concerning the restricted $L^p$-norms (\cite{BurqGerardTzvetkov2007}, 
\cite{LiYoung2012}, \cite{BlomerMaga2015}, \cite{Marshall2016}, \cite{LiLiuYoung}), 
all of which are in the spectral aspect. In this article, we approach the problem from a different point of view.
We keep the spectral data fixed, but search the bound in the level aspect following the ideas in \cite{Blomer2013}.

\subsection{Outline of the article}
In \S\ref{S:notation}, we set up general notations and give some preparations of this article.
In \S\ref{S:preliminary on Maass form}, we review some basic definitions and facts about Maass forms, which
are used in the proof of \thmref{T:main result}.
In \S\ref{S:proof of main theorem}, we sketch the proof of \thmref{T:main result} following the idea of Blomer. 
As the substantial parts are already contained in \cite{Blomer2013},
we mainly focus on the differences between these two cases and give a proof whenever the differences occur.
In \S\ref{S:special value formula}, we state our special value formula 
\thmref{T: special value formula} in the adelic settings. 
In \S\ref{S:local computation}, we compute the local period integrals at the real place. The key ingredient is \propref{P:regularization}, which simplifies our calculations significantly.
\propref{P:regularization} is proved in \S\ref{S:regularization} following the idea in 
\cite[Proposition 5.1]{ML2017}. In the appendix \S\ref{S:appendix}, we 
give explicit formulae for certain Whittaker functions on $\G$ over archimedean local fields which are used 
in the calculations of the local period integrals. We also provide a proof on the estimate of Whittaker functions 
when the representation is unitary, as we can not find a proper reference. Such estimate is used in the proof 
of \propref{P:regularization}.

\subsection*{Acknowledgements}
This article presents part of my Ph.D. thesis and I would like to thank professor Ming-Lun Hsieh for his support and 
guidance. Thanks also to professor Sheng-Chi Liu for pointing out a possible application to the restricted $L^2$-norm 
problem, and to professor P. Nelson for his comments. Finally, I would like thank the referee for 
many helpful comments and suggestions, which greatly improve the original draft.

\section{Notations and preparations}
\label{S:notation}
\subsection{General notations}\label{SS:general notation}
We give a list of our most frequently used notations.
In this article, we only consider fields with characteristic zero. Let $F$ be  a number field. 
We denote $F_v$ the completion of $F$ at a place $v$ of $F$ and $\A_F$ the ring of adeles of $F$. 
We also write $\A=\A_\Q$. The profinite completion of $\Z$ is denoted by $\wh{\Z}$.
If $M$ is an abelian group, then we put $\wh{M}=M\ot_\Z\wh{\Z}$. 
If $F_v$ is non-archimedean, then we let $\varpi_{F_v}$, $\cO_{F_v}$ and $q_{F_v}$ be 
a uniformizer, the valuation ring and the cardinality of the residue field of $F_v$, respectively. 
Let $|\cdot|_{F_v}$ be the absolute value on $F_v$ given by 
$|x|_\C=x\b{x}$, $|x|=|x|_\R={\rm max}\stt{x,-x}$ and $|\varpi_{F_v}|_F=q^{-1}_{F_v}$. 
Define the local zeta functions by
\[
\zeta_\C(s)=2(2\pi)^{-s}\Gamma(s);
\quad
\zeta_\R(s)=\pi^{-s/2}\Gamma(s/2);
\quad
\zeta_{F_v}(s)=(1-q_{F_v}^{-s})^{-1}.
\]
The (complete) Dedekind zeta function is then given by
\[
\zeta_F(s)=\prod_{v<\infty}\zeta_{F_v}(s)
\quad\text{and}\quad
\xi_F(s)=
\prod_{v}\zeta_{F_v}(s)
\] 
where $v$ runs over all finite places (resp. places) of $F$ in the definition of $\zeta_F(s)$ 
(resp. $\xi_F(s)$).
Let $K\subset\G(F_v)$ be the compact subgroup given by
\begin{equation}\label{E:compact subgroup}
K={\rm SU}(2);
\quad 
K={\rm SO}(2);
\quad
K=\G(\cO_{F_v}),
\end{equation}
according to $F_v=\C$ or $F_v=\R$ or $F_v$ is non-archimedean, respectively.

If $R$ is a commutative unital ring, then we let $R^{\x}$ be the group of invertible elements in 
$R$ and ${\rm M}_2(R)$ be the ring of $2\x 2$ matrices with the entries in $R$. 
If $\frak{n}\subset R$ is an ideal, then we define
\[
{\rm M}_2(\frak{n})
=
\stt{\pMX abcd\in{\rm M}_2(R)\mid c\in\frak{n}}
\quad\text{and}\quad
K_0(\frak{n})={\rm M}_2(\frak{n})\cap\G(R).
\]
If $\chi:R^{\x}\to\C^{\x}$ is a homomorphism, then we abuse the notation to use $\chi$ to denote
the homomorphism on $\G(R)$ defined by $\chi(h)=\chi({\rm det}(h))$. 
If $f$ is a function on $\G(R)$ and $g\in\G(R)$, then we let 
$f\ot\chi$ and $\rho(g)f$ be the functions on $\G(R)$ given by $(f\ot\chi)(h)=f(h)\chi(h)$ and 
$\rho(g)f(h)=f(hg)$, respectively.  Put
\[
I_2=\pMX 1001
\quad\text{and}\quad
J_2=\pMX{-1}{0}{0}{1}.
\]

If $\pi$ is a representation of a group $G$, then we often identify the representation space of $\pi$
with $\pi$ itself and the central character (if exist) of $\pi$ is denoted by $\omega_\pi$. 
If $\chi$ is a character of $G$, then we let $\pi\ot\chi$ be the representation of $G$ on the same representation space of
 $\pi$ with the action $\pi\ot\chi(g)=\pi(g)\chi(g)$.  If $H$ is a subgroup of $G$, then the subspace of $H$-fixed elements is
 denoted by $\pi^H$. When $G=\G(F)$ where $F$ is a local field, 
and $\pi$ is an admissible representation of finite length, we let $\t{\pi}$ be the admissible 
dual of $\pi$. When $F$ is archimedean, we always consider $\pi$ to be a Harish-Chandra module.
 
We use the notation $A\ll_{X_1,X_2,\ldots,X_n} B$ to indicate that there exists a constant $C>0$,
depending at most upon $X_1,X_2,\cdots,X_n$ so that $|A|\leq C|B|$. 

\subsection{Irreducible representations of ${\rm SO}(2)$ and ${\rm SU}(2)$}
Let $n$ be an integer. Let $\chi_n$ be the character of ${\rm SO}(2)$ defined by 
$\chi_n(k(\theta))=e^{in\theta}$, where 
\begin{equation}\label{E:k(theta)}
k(\theta)=\pMX{{\rm cos}\,\theta}{{\rm sin}\,\theta}{-{\rm sin}\,\theta}{{\rm cos}\,\theta}
\in{\rm SO}(2).
\end{equation}

Suppose that $n>0$. Let ${\rm Sym}^n(\C^{\oplus 2})$ denote the $n$-th symmetric
power of the standard two-dimensional representation of $\G(\C)$. By abuse of notation, we also write 
${\rm Sym}^n(\C^{\oplus 2})$ for the unique irreducible $n+1$ dimension representation of 
${\rm SU}(2)$. Let $\cL_n(\C)\subset\C[X,Y]$ be the subspace of 
homogeneous polynomials of degree $n$. Then ${\rm Sym}^n(\C^{\oplus 2})\cong (\rho_n,\cL_n(\C))$ with
$\rho_n(g)P(X,Y)=P((X,Y)g)$, where $P(X,Y)\in\cL_n(\C)$ and $g\in\G(\C)$.
Let $(\cdot,\cdot)_n$ be the non-degenerate hermitian
pairing on $\cL_n(\C)$ defined by
\begin{equation}\label{E:hermitian invariant pairing for symmetric power}
(X^{i}Y^{n-i}, X^{j}Y^{n-j})_n
=
\delta_{ij}\begin{pmatrix}n\\i\end{pmatrix}^{-1},
\end{equation} 
and then extend to $\cL_n(\C)$ by the linearity.
Here $\delta_{ij}$ is the Kronecker delta. One verifies that 
\begin{equation}\label{E:SU(2) equivariant property}
(\rho_n(u)P, \rho_n(u)Q)_n
=
(P, Q)_n,
\end{equation}
for $u\in{\rm SU}(2)$ and $P,Q\in\cL_n(\C)$.
In particular, $(\cdot,\cdot)_n$ is ${\rm SU}(2)$-invariant. 

For $0\leq j\leq n$, we define
\begin{equation}\label{E:another basis for symmetric power}
v_{n,j}=(X+\sqrt{-1}Y)^j(X-\sqrt{-1}Y)^{n-j}=(\sqrt{-1})^{j-n}
\rho_n\left(\pMX{1}{\sqrt{-1}}{\sqrt{-1}}{1}\right) X^jY^{n-j}.
\end{equation}
Then $\rho_n(k(\theta))v_{n,j}=\chi_{2j-n}(k(\theta))v_{n,j}$ and we have 
\begin{equation}\label{E:norm for v_n,j}
(v_{n,j},v_{n,j})_n=2^n\begin{pmatrix}n\\j\end{pmatrix}^{-1}.
\end{equation}

\subsection{Irreducible constituents of induced representations}\label{SS:classification}
Let $F$ be a local field. Denote by $\rho(\mu,\nu):={\rm Ind}_{\B(F)}^{\G(F)}(\mu\bt\nu)$ an induced representation of 
$\G(F)$, where $\mu$, $\nu$ are characters of $F^{\x}$.  The underlying space $\cB(\mu,\nu)$ of $\rho(\mu,\nu)$
consists of right $K$-finite $\C$-valued functions $f$ on $\G(F)$ satisfying the following rule:
\[
f
\left(
\begin{pmatrix}
y_1&x\\0&y_2
\end{pmatrix}g
\right)
=
\mu(y_1)\nu(y_2)\left|\frac{y_1}{y_2}\right|^{\frac{1}{2}}_F f(g),
\]  
for $y_1,y_2\in F^{\x}$, $x\in F$ and $g\in\G(F)$. The group $\G(F)$ acts on this space via the right 
translation. It is known (\cite[Theorems 3.3, 5.11, 6.2]{JLbook}) that $\rho(\mu,\nu)$ is reducible if and only if
\begin{itemize}
\item
$F$ is non-archimedean and $\mu\nu^{-1}=|\cdot|_F^{\pm}$;
\item
$F=\R$ and $\mu\nu^{-1}(t)=t^{\pm a}{\rm sgn}(t)$ for some positive integer $a$;
\item
$F=\C$ and $\mu\nu^{-1}(t)=(t^a\b{t}^b)^\pm$ for some positive integers $a,b$.
\end{itemize}

The representation $\rho(\mu,\nu)$, when irreducible, is called a $principal$ $series$ $representation$, 
except the case where $F=\R$ and $\mu\nu^{-1}(t)={\rm sgn}(t)$, for which is called a $limit$ $of$ $discrete$ 
$series$ $representation$. We denote by $\pi(\mu,\nu)$ 
\footnote{
In \cite[Theorems 3.3, 5.11, 6.2]{JLbook}, the notation $\pi(\mu,\nu)$ also stands for any irreducible finite-dimensional 
representation of $\G(F)$ which occurs in a reducible $\rho(\mu,\nu)$. In this article, however; representations
$\pi(\mu,\nu)$ are always infinite-dimensional.}
any representation which is isomorphism to $\rho(\mu,\nu)\cong\rho(\nu,\mu)$.
On the other hand, if $\rho(\mu,\nu)$ is reducible, then it possesses a unique 
infinite-dimensional irreducible constituent, and we let $\sigma(\mu,\nu)$ be any representation isomorphic to such a 
constituent. When $F$ is non-archimedean, we have $\sigma(\mu,\nu)\cong {\rm St}_F\ot\chi$ for some character $\chi$
of $F^{\x}$, where ${\rm St}_F$ is the $Steinberg$ $representation$ of $\G(F)$.
Let $F=\R$ or $\C$ and $\pi$ be an irreducible admissible representation of $\G(F)$. 
By \cite[Theorems 5.11, 6.2]{JLbook},  $\pi$ is isomorphic to a constituent of 
$\rho(\mu,\nu)$ for some $\mu,\nu$. 
Suppose now that $\pi$ is infinite-dimensional. Then there is an integer $k\geq 0$ such that 
$\pi|_K$ is isomorphic (as $K$-modules) to 
\[
\bigoplus_{|n|\geq k,\,n\equiv k\pmod 2}\C\chi_n
\quad\text{or}\quad
\bigoplus_{n\geq k,\,n\equiv k\pmod 2}{\rm Sym^n(\C^{\oplus 2})}
\]
according to $F=\R$ or $\C$. We call such a $k$ the $minimal$ $weight$ of $\pi$. 
When $F=\R$, any non-zero $\chi_n$-eigenvector in $\pi$ is called a weight $n$ element. 
The minimal weight $k$ of $\pi(\mu,\nu)$ is given by $\mu\nu^{-1}(-1)=(-1)^k$ with
$k\in\stt{0,1}$ when $F=\R$ and $\mu\nu^{-1}(e^{i\theta})=e^{\pm ik\theta}$ when $F=\C$. 
Since $\pi$ is assumed to be infinite-dimensional, it follows from \cite[Theorem 6.2 (vi), (vii)]{JLbook} that
every such $\pi$ is isomorphic to some $\pi(\mu,\nu)$ when $F=\C$.
On the other hand, when $F=\R$ and 
$\mu\nu^{-1}(t)=t^{\pm a}{\rm sgn}(t)$ for some positive integer $a$, then $\sigma(\mu,\nu)$ is a $discrete$ $series$ 
$representation$ with the minimal weight $k=a+1$. Note that a limit of discrete series representation is also a discrete 
series representation. 
 
 \subsection{Whittaker model}\label{SS:Whittaker model}
Let $F$ be a local field and $\psi$ be a non-trivial additive character of $F$. 
Let $\cW(\psi)$ be the space of smooth $\C$-valued functions 
$W$ on $\G(F)$ such that
\begin{itemize}
\item
$W\left(\pMX{1}{x}{0}{1}g\right)=\psi(x)W(g)$;
\item
$W\left(\pMX{y}{0}{0}{1}\right)={\rm O}(|y|_F^N)$ for some $N>0$ as $|y|_F\to\infty$ when $F$ is archimedean.
\end{itemize}
Then $\G(F)$ acts on $\cW(\psi)$ by the right translation.
Let $\pi$ be an irreducible admissible representation of $\G(F)$. We call $\pi$ $generic$ if it occurs in 
$\cW(\psi)$. In the $\G$ case, this is equivalent to that $\pi$ is infinite-dimensional. 
By \cite[Theorems 2.14, 5.3, 6.3]{JLbook}, we know that $\pi$ occurs in $\cW(\psi)$ with
multiplicity one. Denote by $\cW(\pi,\psi)$ the $\pi$-isotypic component of $\cW(\psi)$. 
Such space is called the Whittaker model of $\pi$ with respect to $\psi$.
If $\psi'$ is another non-trivial additive character, then $\psi'(x)=\psi(ax)$ for some $a\in F^{\x}$ and the map
$W\mapsto W'$ defines an $\G(F)$-isomorphism from $\cW(\pi,\psi)$ onto $\cW(\pi,\psi')$, where 
$W'(g):=W\left(\pMX{a}{0}{0}{1} g\right)$. 

We need the following estimate of the Whittaker functions. 
Suppose that $\pi$ is a unitary irreducible admissible 
generic representation of $\G(F)$. Then for every $W\in\cW(\pi,\psi)$, there is a Bruhat-Schwarz function $\Phi$ 
on $F$ such that
\begin{equation}\label{E:estimate Whittaker function}
W\left(\pMX{y}{0}{0}{1}k\right)
\ll_{\lambda(\pi),W,\epsilon}|y|_F^{\frac{1}{2}-\lambda(\pi)-\epsilon}\Phi(y)
\end{equation}
for every $k\in K$. Here $0\leq \lambda(\pi)<1/2$ is the real number defined by
\begin{equation}\label{E:ramanujan}
\lambda(\pi)
=
\begin{cases}
0\quad&\text{if $\pi$ is tempered},\\
|\lambda|\quad&\text{if $\pi\cong\pi(\chi|\cdot|_F^{\lambda},\chi|\cdot|_F^{-\lambda})$ is complementary},
\end{cases}
\end{equation}
where $\chi$ is some unitary character of $F^{\x}$.  
When $F$ is non-archimedean, \eqref{E:estimate Whittaker function}
follows from \cite[Lemma 14.3]{JLbook2}. On the other hand, we provide a 
proof for \eqref{E:estimate Whittaker function} in 
\S\ref{A:majorization of Whittaker functions} (\propref{P:estimate whittaker function}) when $F$ is archimedean, 
as we can not find a proper reference. 

\subsection{Integration formulae involving modified Bessel functions}
In this subsection, we state some integration formulae involving modified Bessel functions. These 
formulae will be used frequently in the computations of the local period integrals for the real places.
Let $K_s(z)$ be the modified Bessel
function of order $s$. Then for ${\rm Re}(z)>0$, we have the following integral representation
\begin{equation}\label{E:K-bessel function}
K_s(z)
=
\frac{1}{2}\int_0^\infty{\rm exp}\left(-\frac{z}{2}\left(t+t^{-1}\right)\right)t^{s-1}dt.
\end{equation}
Here ${\rm exp}(z)=e^z$ is the exponential function. 

\begin{lm}\label{L:integration formula for Bessel function}
Let $a, \mu$, $\nu$ and $\lambda$ be complex numbers with ${\rm Re}(a)>0$.
\begin{itemize}
\item[(1)]
Suppose that ${\rm Re}(\lambda+1\pm\mu)>0$. Then
\[
\int_0^\infty y^\lambda K_\mu(ay)dy
=
2^{\lambda-1}a^{-\lambda-1}
\Gamma\left(\frac{1+\lambda+\mu}{2}\right)\Gamma\left(\frac{1+\lambda-\mu}{2}\right).
\]
\item[(2)]
Suppose that ${\rm Re}(\lambda)>|{\rm Re}(\mu)|$. Then
\[
\int_0^\infty y^{\lambda-1}e^{-ay}K_\mu(ay)dy
=
\pi^{1/2}(2a)^{-\lambda}
\frac{\Gamma(\lambda+\mu)\Gamma(\lambda-\mu)}{\Gamma(\lambda+1/2)}.
\]
\item[(3)]
Suppose that ${\rm Re}(\lambda)>|{\rm Re}(\mu)|+|{\rm Re}(\nu)|-1$. Then
\[
\int_0^\infty y^{\lambda}K_\mu(ay)K_\nu(ay)dy
=
\frac{2^{-2+\lambda}a^{-\lambda-1}}{\Gamma(1+\lambda)}
\prod_{\epsilon_1,\epsilon_2\in\stt{\pm}}
\Gamma\left(\frac{1+\lambda+\epsilon_1\mu+\epsilon_2\nu}{2}\right).
\]
\end{itemize}
\end{lm}

\begin{proof}
$(1)$ is given by \cite[6.561.16]{Table2007}, while $(2)$ and $(3)$ are the special cases of 
\cite[6.621.3]{Table2007} and \cite[6.576.4]{Table2007}, respectively.
\end{proof}

\section{Preliminaries on Maass forms}\label{S:preliminary on Maass form}
In this section, we set up some notations for the classical Maass 
forms and their base change lifts to a quadratic extension. Besides recalling some basic facts about them,
we also give a dictionary between the classical Maass forms and automorphic forms.  
The notations and the assumptions will be the same as in \S\ref{S:main result}.
Therefore $\cK$ denotes a fixed quadratic field extension over $\Q$ whose (narrow) 
class number is one. Let $\cO$ be the ring of integers of $\cK$ and $\tau_\cK$ be the quadratic 
Dirichlet character associated to $\cK/\Q$ by the global class field theory. 
Recall that $\delta_\cD\in\stt{0,1}$ is given by $\cD=(-1)^{\delta_\cD}|\cD|$.

\subsection{Maass forms}
Let $q>0$ be a prime. 
Let $\G^+(\R)$ be the subgroup of $\G(\R)$ consisting of the matrices with positive determinants.
We let $\G^+(\R)$ acts on $\frak{H}$ by the usual linear fractional transformation, and the 
action is denoted by $h\cdot \tau$, where $h\in\G^+(\R)$ and $\tau\in\frak{H}$. 
Write $\|g\|_2^2=\langle g,g\rangle$ for every $g\in L^2(X_0(q))$.
For any $g\in\cB_q\cup\cB_1$ with the Laplacian eigenvalue $\lambda_g$, we denote by 
\begin{equation}\label{E:spectral parameter}
t_g:=\sqrt{\lambda_g-1/4}\in\cT:=\R\cup (-1/2,1/2)i,
\end{equation}
the spectral parameter of $g$ and by $\lambda_g(n)$ the $n$-th Hecke eigenvalue. We put 
$\delta_g=0$ or $1$ according to $g$ is even or odd. When $g\in\cB_q$, let $w_g\in\stt{\pm 1}$ be 
the eigenvalue of $g$ under the Atkin-Lehner involution, i.e.
\begin{equation}\label{E:Atkin-Lehnner}
g(\sigma_0\cdot\tau)=w_g g(\tau)
\quad\text{where}\quad
\sigma_0=\pMX{0}{-1}{q}{0}.
\end{equation}
Note that we have the relation
$w_g=-\lambda_g(q)q^{1/2}$.
For $g\in\cB_1$, we define $w_g=1$ and 
\begin{equation}\label{E:g*}
g^*(\tau)
=
\left(1-\frac{q\lambda_g(q)^2}{(q+1)^2}\right)^{-1/2}
\left(g(q\tau)-\frac{q^{1/2}\lambda_g(q)}{(q+1)}g(\tau)\right).
\end{equation}
By \cite[Proposition 2.6]{IwaniecLuoSarnak2000}, $g$ and $g^*$ have the same norm and are orthogonal to each other. 
Furthermore,
\[
\cB_q\cup\cB_1\cup\cB^*_1,
\quad
\cB_1^*:=\stt{g^*\mid g\in\cB_1}
\]
is an orthonormal basis (with respect to \eqref{E:inner product for X_0(q)}) of the non-trivial
cuspidal spectrum of $L^2(X_0(q))$.

\subsection{Hyperbolic three space}
Recall the hyperbolic three space
\[
\frak{H}'=\stt{\pMX{z}{-y}{y}{\b{z}}\mid z\in\C, y\in\R^\x_+}.
\]
Here $\R^\x_+$ is the group of positive real numbers.
The group ${\rm SL}_2(\C)$ acts transitively on $\frak{H}'$ by
\begin{equation}\label{E:action on three space}
h\cdot\tau=(t(a)\tau+t(b))(t(c)\tau+t(d))^{-1}
\end{equation}
where $h=\pMX abcd$, $\tau\in\frak{H}'$ and $t(z):=\pMX{z}{0}{0}{\b{z}}$ for $z\in\C$. 
Note that the stabilizer of $\pMX{0}{-1}{1}{0}\in\frak{H}'$ is ${\rm SU}(2)$,  
so that we may identify the symmetric spaces ${\rm SL}_2(\C)/{\rm SU}(2)\cong\frak{H}'$.
The action \eqref{E:action on three space} extends naturally to the group
$\G^+(\C):=\stt{g\in\G(\C)\mid {\rm det}(g)\in\R^{\x}_+}\cong\R^\x_+ \x{\rm SL}_2(\C)$. 
Observe that $\frak{H}$ sits naturally inside $\frak{H}'$ via the $\G^+(\R)$-equivariant embedding
\begin{equation}\label{E:from H to H'}
x+iy\hookrightarrow\pMX{x}{-y}{y}{x}.
\end{equation}
By \cite[(1.4e)]{Hida1993}, there is a ${\rm SL}_2(\C)$-invariant measure $d\mu(\tau)$ on $\frak{H}'$  
given by
\begin{equation}\label{E:SL(2,C)-invariant measure}
d\mu(\tau)=\frac{|dz\wedge dy\wedge d\b{z}|}{2y^3},
\quad
\tau=\pMX{z}{-y}{y}{\b{z}}.
\end{equation}

\subsection{Base change to $\cK$}\label{SS:base change}
We review the base change lift of $f\in\cB_q$ from $\Q$ to $\cK$ in this subsection. A basic reference is 
\cite[Section 20]{JLbook2}. We assume $q$ and $\cD$ to be coprime in the rest of this subsection.
Let $\B\subset\G$ be the Borel subgroup, and $\N\subset\B$ be the unipotent subgroup.
We use $v$ to indicate an arbitrary place of 
$\Q$, and we write $p$ (resp. $\infty$) for a finite place (resp. an infinite place) of $\Q$.

\subsubsection{Adelic automorphic form}  
Let $f\in\cB_q$. For a prime $p\neq q$, we let 
$\stt{\alpha_p, \alpha^{-1}_p}$ be the Satake parameter of $f$ at $p$. Then 
\[
1-\lambda_f(p)X-p^{-1}X^2=(1-p^{-1/2}\alpha_pX)(1-p^{-1/2}\alpha_p^{-1}X).
\]
We fix $s_p\in\C$ so that $\alpha_p=p^{-s_p}$. 
The Maass form $f$ determines a cusp form $\mathbf{f}$ on $\G(\A)$ by the formula
\[
\mathbf{f}(h)=f(h_\infty\cdot\sqrt{-1})
\]
for $h=\gamma h_\infty k$ with $\gamma\in\G(\Q)$, $h_\infty\in\G^+(\R)$ and $k\in K_0(q\wh{\Z})$. 
Observe that this is well-defined since $\G(\Q)\cap\G^+(\R)K_0(q\wh{\Z})=\Gamma_0(q)$.
Let $\pi=\ot'_v\pi_v$ be the unitary irreducible cuspidal automorphic representation of 
${\rm PGL}_2(\A)$ generated by $\mathbf{f}$. Then we have
\[
\pi_\infty
\cong
\pi(|\cdot|_\R^{it_f}{\rm sgn}^{\delta_f}, |\cdot|_\R^{-it_f}{\rm sgn}^{\delta_f});
\quad
\pi_p
\cong
\pi(|\cdot|_{\Q_p}^{s_p}, |\cdot|_{\Q_p}^{-s_p});
\quad
\pi_q\cong {\rm St}_{\Q_q}\ot\chi_q,
\]
where $\chi_q$ is the unramified quadratic character of $\Q_q^{\x}$ with 
\begin{equation}\label{E:Atkin-Lehner eigenvalue for special}
\chi_q(q)=-w_f.
\end{equation}
Notice that $\bff\in\pi$ is the unique (up to constants) element such that
\[
\bff(hk_\infty k)=\bff(h)
\]
for $h\in\G(\A)$, $k_\infty\in{\rm SO}(2)$ and $k\in K_0(q\h{\Z})$.

Let $\cK_v=\cK\ot_\Q\Q_v$ for each $v$ and $\pi_\cK=\ot'_v\pi_{\cK,v}$ be the base 
change lift of $\pi$ to $\cK$. Then $\pi_\cK$ is a unitary irreducible cuspidal automorphic 
representation of ${\rm PGL}_2(\A_\cK)$ while $\pi_{\cK,v}$ is a unitary irreducible admissible
generic representation of ${\rm P}\G(\cK_v)$ for each $v$. For $p\neq q$, we have
\[
\pi_{\cK,p}\cong{\rm Ind}_{\B(\cK_p)}^{\G(\cK_p)}(|\cdot|_{\cK_p}^{s_p}\bt|\cdot|_{\cK_p}^{-s_p}).
\]
For  $v\in\stt{\infty,q}$, we have the followings.
If $\cK_v=\Q_v\oplus\Q_v$, then $\pi_{\cK,v}=\pi_v\bt\pi_v$. On the other hand, if $\cK_v$ is a field, then
\[
\pi_{\cK_,\infty}\cong\pi(|\cdot|^{it_f}_\C,|\cdot|^{-it_f}_\C);
\quad
\pi_{\cK,q}\cong{\rm St}_{\cK_q}.
\] 
By the theory of newform (\cite{Casselman1973}) and the $K_\infty$-type of $\pi_{\cK,\infty}$, 
there is a unique (up to constants) element $0\neq \bff_\cK\in\pi_\cK$ such that
\begin{equation}\label{E:characterize f_K}
\bff_\cK(hk_\infty k)=\bff_\cK(h)
\end{equation}
for $h\in\G(\A_\cK)$, $k_\infty\in K_\infty$ and $k\in K_0(q\h{\cO})$.
Here $K_\infty={\rm SO}(2)\x{\rm SO}(2)$ when $\cD>0$, and $K_\infty={\rm SU}(2)$ when $\cD<0$.

\subsubsection{Classical automorphic form}\label{SSS:classical auto. form}
Put
\[
\frak{H}_\cD
=
\begin{cases}
\frak{H}\x\frak{H}\quad&\text{if $\cD>0$},\\
\frak{H}'\quad&\text{if $\cD<0$}.
\end{cases}
\]
Fix an embedding $\cK\hookto\C$ throughout and let $\sigma$ be the generator of ${\rm Gal}(\cK/\Q)$.
Note that $\sigma(z)=\b{z}$ when $\cD<0$. Let ${\rm SL}_2(\cO)$ acts on $\frak{H}_\cD$ as follows. 
Suppose that $\cD<0$. Then ${\rm SL}_2(\cO)\subset{\rm SL}_2(\C)$ and the action is given by 
\eqref{E:action on three space}.
Suppose that $\cD>0$. We embed ${\rm SL}_2(\cO)$ into ${\rm SL}_2(\R)\x{\rm SL}_2(\R)$ via the map
\[
\pMX abcd\hookrightarrow\left(\pMX abcd, \pMX{\sigma(a)}{\sigma(b)}{\sigma(c)}{\sigma(d)}\right).
\]
The action of ${\rm SL}_2(\cO)$ on $\frak{H}_\cD$ is then induced from this embedding and 
the component-wise action of ${\rm SL}_2(\R)\x{\rm SL}_2(\R)$ on $\frak{H}\x\frak{H}$.
We associate $\bff_\cK$ with a unique
\footnote{In general, there are $h_\cK$ classical automorphic forms on $\frak{H}_\cD$ associated
to $\bff_\cK$, where $h_\cK$ is the (narrow) class number of $\cK$.}
classical automorphic form $f_\cK$ on the symmetric domain $\frak{H}_\cD$ as follows.  
Let $\tau\in\frak{H}_\cD$.  Choose any $h\in\G^+(\R)\x\G^+(\R)$
so that $h\cdot(\sqrt{-1},\sqrt{-1})=\tau$ when $\cD>0$, and $h\in{\rm SL}_2(\C)$ such that 
$h\cdot \pMX{0}{-1}{1}{0}=\tau$ when $\cD<0$.
Then we define
\[
f_\cK(\tau)=\bff_\cK(h).
\]
It follows immediately from \eqref{E:characterize f_K}
that this is well-defined and $f_\cK$ is 
$\Gamma_0(q\cO)$-invariant on the left.
We denote by $f_\cK|_\frak{H}\in L^2(X_0(q))$ the pull-back of $f_\cK$ to $\frak{H}$ 
via the diagonal embedding or the embedding in \eqref{E:from H to H'} depending on the sign of $\cD$. 
The following lemma describes the eigenvalue of $f_\cK|_{\frak{H}}$ under the Atkin-Lehner involution. 
For this we first recall a result in \cite[Proposition 3.1.2]{Schmidt2002}. Let $F$ be a non-archimedean local field, 
$\chi$ be a unramified quadratic character of $F^{\x}$ and $\pi={\rm St}_F\ot\chi$ be a special representation of $\G(F)$.
Then we have
\[
\pi\left(\pMX{0}{-1}{\varpi_F}{0}\right)v=-\chi(\varpi_F)v.
\]   
for every non-zero newform $v\in\pi$.

\begin{lm}\label{L:Atkin-Lehnner for base change}
Let $a=0$ (resp. $a=1$) if $q$ is split (resp. inert) in $\cK$. Then we have
\[
f_\cK|_\frak{H}\left(\sigma_0\cdot\tau\right)=(-1)^a f_\cK|_\frak{H}(\tau).
\]
Here $\sigma_0$ is the matrix in \eqref{E:Atkin-Lehnner}. 
\end{lm}

\begin{proof}
Fix an isomorphism $\pi_\cK\cong\ot'_v\pi_{\cK,v}$. When $v=q$, we have 
\[
\pi_{\cK,q}\cong
\begin{cases}
({\rm St}_{\Q_q}\ot\chi_q)\bt({\rm St}_{\Q_q}\ot\chi_q)\quad&\text{if $q$ is split in $\cK$},\\
{\rm St}_{\cK_q}\quad&\text{if $q$ is inert in $\cK$}.
\end{cases}
\]
Here $\chi_q$ is the unramified quadratic character of $\Q_q^{\x}$ with $\chi_q(q)=-w_f$.
Under the fixed isomorphism, $\bff_\cK$ corresponds to a pure tensor $\ot'_v\bff_{\cK,v}$ with 
$\bff_{\cK,p}\in\pi_{\cK,p}$ a newform for each finite place $p$. 
Let $\alpha=(\alpha_v)_v\in\G(\A_\cK)$ be the element such that 
$\alpha_v\in\G(\cK_v)$ is the identity for every $v\neq q$, and $\alpha_q\in\G(\cK_q)$ is given by
\[
\alpha_q=
\begin{cases}
(\sigma_0,\sigma_0)\quad&\text{if $q$ is split in $\cK$},\\
\sigma_0\quad&\text{if $q$ is inert in $\cK$}.
\end{cases}
\]
By the remark in front of the lemma, we see that $\pi_\cK(\alpha)\bff_\cK$ corresponds to the element 
$\ot_v'\pi_{\cK,v}(\alpha_v)\bff_{\cK,v}=(-1)^a\ot_v'\bff_{\cK,v}$.
Therefore we find that $\pi_\cK(\alpha)\bff_\cK=(-1)^a\bff_\cK$. 
Taking into account that the central character of $\pi_\cK$ is trivial and $\sigma_0^{-1}=-qI_2\cdot\sigma_0$, the lemma now follows from these together with the dictionary between adelic and classical automorphic forms.
\end{proof}

\subsection{$L^2$-Norm and the adjoint $L$-function}
Let $f\in\cB_q$ and let $f_\cK$ be its base change lift to $\cK$ defined in \S\ref{SS:base change}.
Let $\bff_\cK$ be the adelic automorphic form associated to $f_\cK$ and
let $\pi_\cK$ be the unitary irreducible cuspidal automorphic representation of 
${\rm P}\G(\A_\cK)$ generated by $\bff_\cK$.
We assume $f_\cK$ to be $L^2$-normalized in the sense that
\begin{equation}\label{E:L^2 normalized}
\|f_\cK\|_2^2
:=
\int_{\Gamma_0(q\cO)\backslash\frak{H}_\cD}
|f_\cK(\tau)|^2d\mu(\tau)=1,
\end{equation}
where $d\mu(\tau)$ is the invariant measure given by \eqref{E:SL(2,C)-invariant measure} when
$\cD<0$, and is given by $d\mu(\tau)=d\mu(\tau_1)d\mu(\tau_2)$
for $\tau=(\tau_1,\tau_2)\in\frak{H}_\cD$ when $\cD>0$. 
Certainly this normalization is only up to norm $1$ elements, but it is enough for our purpose.
The automorphic form $f_\cK$ has the Fourier expansion
\begin{equation}\label{E:Fourier expansion for D>0}
f_\cK(\tau_1,\tau_2)
=
\rho_{f_\cK}
\sum_{0\neq \alpha\in\cO}
\lambda_{f_\cK}(\alpha)(y_1y_2)^{\frac{1}{2}}
K_{it_f}(2\pi|\alpha|\cD^{-1/2}y_1)K_{it_f}(2\pi|\sigma(\alpha)|\cD^{-1/2}y_2)
e^{2\pi i\cD^{-1/2}(\alpha x_1-\sigma(\alpha)x_2)}
\end{equation}
if $\cD>0$, and
\begin{equation}\label{E:Fourier expansion for D<0}
f_\cK(\tau)
=
\rho_{f_\cK}\sum_{0\neq \alpha\in\cO}
\lambda_{f_\cK}(\alpha)yK_{2it_f}(4\pi|\alpha|^{1/2}_\C|\cD|^{-1/2} y)
e^{2\pi i\cD^{-1/2}(\alpha z-\overline{\alpha z})}
\end{equation}
if $\cD<0$. Here $\tau_1=x_1+iy_1$, $\tau_2=x_2+iy_2$ and $\tau=\pMX{z}{-y}{y}{\b{z}}$. 
We assume that $\lambda_{f_\cK}(1)=1$ in both cases. 
Let ${\rm Ad}^2 f_\cK$ be the adjoint square lift of $f_\cK$ to a cusp form on 
${\rm GL}_3$ as in \cite{GelbartJacquet1978} and 
$L(s,{\rm Ad}^2f_\cK)$ be the associated $L$-function. It has the factorization
\begin{equation}\label{E:factorization adjoint L-fcn for base change}
L(s,{\rm Ad}^2f_\cK)
=
L(s,{\rm Ad}^2 f)L(s,{\rm Ad}^2 f\x\tau_\cK).
\end{equation}
Define the complete $L$-function
\begin{align}\label{E:complete L-fcn for adjoint}
\begin{split}
\Lambda(s,{\rm Ad}^2 f_\cK)
&=
L_\infty(s,{\rm Ad}^2 f_\cK)L(s,{\rm Ad}^2 f_\cK),\\
L_\infty(s,{\rm Ad}^2 f_\cK)
&=
\prod_{\epsilon_1\in\stt{0,\pm},\,\epsilon_2\in\stt{0,1}}
\zeta_\R(s+2\epsilon_1 it_f+\epsilon_2\delta_\cD).
\end{split}
\end{align}
\begin{lm}\label{L:norm of f_K}
We have 
\[
|\rho_{f_\cK}|^{-2}
=
\Lambda(1,{\rm Ad}^2 f_\cK)
\cdot
(q\cD)^2
\cdot
\begin{cases}
2\quad&\text{if $\cD>0$},\\
2^{-4}\pi\quad&\text{if $\cD<0$}.
\end{cases}
\]
\end{lm}

\begin{proof}
We use \cite[Proposition 6]{Wald1985} to compute the constant. Let $\psi_\Q=\prod_v\psi_v$ be the 
non-trivial additive character of $\Q\backslash\A$ such that $\psi_\infty(x)=e^{2\pi ix}$ for 
$x\in\R$. When $v=p<\infty$, let $\psi_p$ be the character on $\Q_p$ which is trivial on $\Z_p$ and  
$\psi_p(x)=e^{-2\pi i x}$ for $x\in\Z[q^{-1}]$. Put 
$\psi_\cK:=\psi_\Q\circ{\rm tr}_{\cK/\Q}=\prod_v \psi_{\cK,v}$. 
For each place $v$ of $\Q$, let $\cW(\pi_{\cK,v},\psi_{\cK,v})$ be the Whittaker model of 
$\pi_{\cK,v}$ with respect to $\psi_{\cK,v}$.
Let $W_p\in\cW(\pi_{\cK,p},\psi_{\cK,p})$ be the newform as in \cite{Schmidt2002} with
\[
W_p\left(\pMX{\cD^{-1/2}}{0}{0}{1}\right)=1.
\]
For the infinite place, we let $W_\infty\in\cW(\pi_{\cK,\infty},\psi_{\cK,\infty})$ be the right
$K_\infty$-invariant element (\S\ref{SSS:Whittaker functions}) given by
\[
W_\infty\left(\pMX y001\right)
=
\begin{cases}
{\rm sgn}(y_1y_2)^{\delta_f}|y_1y_2|^{1/2}
K_{it_f}(2\pi|y_1|)K_{it_f}(2\pi|y_2|)\quad&\text{if $\cD>0$},\\
|y|_\C^{1/2}K_{2it_f}(4\pi|y|_\C^{1/2})\quad&\text{if $\cD<0$}.
\end{cases}
\]
Here $y=(y_1,y_2)\in\cK^{\x}_\infty=\R^{\x}\x\R^{\x}$ if $\cD>0$, and $y\in\cK^{\x}_\infty=\C^{\x}$
if $\cD<0$. Put $W=\prod_v W_v$. The Fourier expansion
\begin{equation}\label{E:expansion}
\bff^0_\cK(h)=\sum_{\alpha\in\cK^{\x}}W\left(\pMX{\alpha}{0}{0}{1} h\right)
\quad (h\in\G(\A_\cK))
\end{equation}
defines a non-zero element in $\pi_\cK$. Moreover, by the choices of $W_v$, the cusp 
form $\bff^0_\cK$ satisfies \eqref{E:characterize f_K}, and hence is parallel to $\bff_\cK$.  
Let $f^0_\cK$ be the classical 
automorphic form associated to $\bff^0_\cK$ as in \S\ref{SSS:classical auto. form}. Then we have
\begin{equation}\label{E:for norm 1}
f^0_\cK=\rho_{f_\cK}^{-1}\cdot |\cD|^{-1/2}\cdot f_\cK,
\end{equation}   
by comparing the expansion \eqref{E:expansion} with \eqref{E:Fourier expansion for D>0} and 
\eqref{E:Fourier expansion for D<0}.
By \cite[Proposition 6]{Wald1985}, we have
\[
\int_{\A^{\x}_\cK\G(\cK)\backslash\G(\A_\cK)}
|\bff_\cK^0(h)|^2dh^{{\rm Tam}}
=
2\xi_\cK(2)^{-1}\Lambda(1,\pi_\cK,{\rm Ad})\prod_v\cH^0_v(W_v,W_v),
\]
where $dh^{{\rm Tam}}$ is the Tamagawa measure on ${\rm PGL}_2(\A_\cK)$, and 
\[
\cH^0_v(W_v,W_v)
=
\frac{\zeta_{\cK_v}(2)}{\zeta_{\cK_v}(1)L(1,\pi_{\cK_v},{\rm Ad})}
\int_{\cK_v^{\x}}\left|W_v\left(\pMX {y_v}{0}{0}{1}\right)\right|^2d^{\x}y_v.
\]
The measure $d^{\x}y_v$ on $\cK^{\x}_v$ is given by 
$d^{\x}y_v=\zeta_{\cK_v}(1)|y_v|^{-1}_{\cK_v}dy_v$ where $dy_v$ is the self-dual measure on 
$\cK_v$ with respect to $\psi_{\cK,v}$. By \lmref{L:integration formula for Bessel function} $(3)$
and a direct computation, we find that 
\begin{equation}\label{E:for norm 2}
\prod_v \cH^0_v(W_v,W_v)
=
2^{-2\delta_\cD}\cdot|\cD|^{-1/2}\cdot\frac{\zeta_{\cK_q}(2)}{\zeta_{\cK_q}(1)},
\end{equation}
where $v$ runs over all places of $\Q$ and $\zeta_{\cK_q}(s)=\zeta_{\Q_q}(s)^2$ if $q$ is split in $\cK$.
On the other hand, by the choices of the measures $d\mu(\tau)$, we have
\begin{equation}\label{E:for norm 3}
\int_{\A^{\x}_\cK\G(\cK)\backslash\G(\A_\cK)}
|\bff_\cK^0(h)|^2dh^{{\rm Tam}}
=
\xi_\cK(2)^{-1}q^{-2}|\cD|^{-3/2}(8\pi^{-1})^{\delta_\cD}\frac{\zeta_{\cK_q}(2)}{\zeta_{\cK_q}(1)}
\int_{\Gamma_0(q\cO)\backslash\frak{H}_\cD}
|f^0_\cK(\tau)|^2d\mu(\tau).
\end{equation}
If $\cD>0$, then \eqref{E:for norm 3} follows from \cite[Lemmas 6.1, 6.3]
{IchinoPrasanna}. If $\cD<0$, then one can follow their arguments to obtain \eqref{E:for norm 3}.
The lemma now follows from \eqref{E:for norm 1}, \eqref{E:for norm 2} and \eqref{E:for norm 3}.
\end{proof}

\section{Proof of \thmref{T:main result}}
\label{S:proof of main theorem}
We prove \thmref{T:main result} in this section. The main important ingredients of the proof are the
special value formulae for the twisted triple product $L$-functions and the Kuznetsov formula 
(\cite[page 409]{IwaniecKowalski2004}). We only sketch the proof as the substantial parts are 
already contained in \cite{Blomer2013}. Indeed, we modify the proof in \cite{Blomer2013} so that 
it also works in our case. We mainly focus on the differences between these two cases 
and give a proof whenever the differences occur.
Suppose that $q$ is coprime to $\cD$ in this section.
Given $f\in\cB_q$, denote by $f_\cK$ its $L^2$-normalized \eqref{E:L^2 normalized} base change lift to 
$\cK$ defined in \S\ref{SS:base change}, and let $\pi_\cK$ be the unitary irreducible cuspidal 
automorphic representation of ${\rm P}\G(\A_\cK)$ generated 
by $f_\cK$. The descriptions of the local components of $\pi_\cK$ can be found in
\S\ref{SS:base change}.
 
\subsection{Triple product $L$-functions}
By a result of Krishnamurthy (\cite{Krishnamurthy2003}), 
the Asai transfer ${\rm As}\,\pi_\cK$ of $\pi_\cK$ is an isobaric automorphic representation of 
${\rm GL}_4(\A)$. We write $L(s,{\rm As} f_\cK)$ for the associated $L$-function. 
Let $g\in\cB_q\cup\cB_1$.
The twisted triple product $L$-function $L(s,{\rm As}f_\cK\x g)$ associated
to $f_\cK$ and $g$ is given by
\[
L(s,{\rm As}f_\cK\x g)
=
L(s,{\rm Ad}^2f\x g)L(s,g\x\tau_\cK).
\]
Here $L(s,{\rm Ad}^2 f_\cK\x g)$ is the Rankin-Selberg L-function for ${\rm GL}_3\x\G$ (\cite{JPSS1983}). 
Observe that the $\G$ $L$-function in our case is twisted by the quadratic character $\tau_\cK$ while in 
\cite[Section 3]{Blomer2013} it is not.
\subsubsection{Complete $L$-functions and the functional equations}
Let $\epsilon\in\stt{0,1}$ such that $\epsilon\equiv\delta_g+\delta_\cD\pmod 2$.  Then one has
\begin{align}\label{E:L-fcn for Maass}
\begin{split}
L_\infty(s,g\x\tau_\cK)
&=
\zeta_\R(s+it_g+\epsilon)\zeta_\R(s-it_g+\epsilon),\\
\Lambda(s,g\x\tau_\cK)
&:=
L_\infty(s,g\x\tau_\cK)L(s,g\x\tau_\cK)
=
(-1)^{\delta_g}w_g(q^\ell\cD^2)^{1/2-s}
\Lambda(1-s,g\x\tau_\cK)
\end{split}
\end{align}
where $\ell=1$ or $0$ depends on $g\in\cB_q$ or $g\in\cB_1$, respectively. On the other hand,
\begin{align}\label{E:L-fcn for sym sq tensor Maass}
\begin{split}
L_\infty(s,{\rm Ad}^2f\x g)
&=
\prod_{\epsilon_1\in\stt{0,\pm},\,\epsilon_2\in\stt{\pm}}
\zeta_\R(s+\epsilon_1 2it_f+\epsilon_2 it_g),\\
\Lambda(s,{\rm Ad}^2 f\x g)
&:=
L_\infty(s,{\rm Ad}^2 f\x g)L(s,{\rm Ad}^2 f\x g)
=
(-1)^{\delta_g}(q^4)^{1/2-s}\Lambda(1-s,{\rm Ad}^2 f\x g).
\end{split}
\end{align}
Notice that $L(s,{\rm Ad}^2 f\x g)$ and hence $L(s,{\rm As}f_\cK \x g)$ vanishes at $s=1/2$ if $g$
is odd. The series expansion of prime-to-$q$ part $L$-function $L^{(q)}(s,{\rm Ad}^2 f\x g)$ is 
given by \cite[(3.1)]{Blomer2013}, while the 
series expansion of $L(s,{\rm Ad}^2 f\x g)$ is given by \cite[(3.2)]{Blomer2013}.
\subsubsection{Approximate functional equations}
Following Blomer, let $G(u)$ and 
$G_2(u,t_1,t_2)$ be defined as in \cite[page 1833]{Blomer2013}. On the other hand, we replace
$G_1(u,t)$ in the same page by
\[
G_1(u,t)
=
\prod_{\epsilon_1,\epsilon_2\in\stt{\pm}}\prod_{j=0}^{A}
\left(\frac{1/2+\epsilon_1 u+\epsilon_2 it +\delta_\cD}{2}+j\right).
\]
Here $A\geq 10$ is an integer.
Properties of these functions can be found in \cite[page 1833]{Blomer2013}. In particular, we have
\begin{equation}\label{E:lower bound for G_1, G_2}
G_1(0,t),\,G_2(0,t_1,t_2)\gg 1\quad\text{for}\quad t,t_1,t_2\in\cT.
\end{equation}
For this lower bound, we need any non-trivial bound towards the Ramanujan conjecture for the infinite
place of the ${\rm GL}_6$ cusp form ${\rm Ad}^2 f\x g$ as in \cite{Kim2003}, \cite{IwaniecLuoSarnak2000}.  
Let $V_1(y;t)$ and $V_2(y;t_1,t_2)$ be the weight functions defined in 
\cite[pages 1833-1834]{Blomer2013}. 
Then \cite[Lemma 1]{Blomer2013} holds for $V_1$ and $V_2$ except the equation (3.5) in that lemma;
however, we do not need this property (see \S\ref{SSS:contribution of oldforms and Eisenstein series}). 

\begin{lm}\label{L:weighted sum of L-fcn}
Let $g\in\cB_q$. We have
\[
G_1(0,t_g)L(1/2,g\x\tau_\cK)
=
(1-\lambda_g(q)q^{1/2})
\sum_{n=1}^\infty
\frac{\lambda_g(n)\tau_\cK(n)}{n^{1/2}}V_1\left(\frac{n}{q^{1/2}|\cD|};t_g\right)
\]
and
\[
G_2(0,t_g,t_f)L(1/2,{\rm Ad}^2 f\x g)
=
2\sum_{m=0}^\infty
\frac{\lambda_{{\rm Ad}^2 f\x g}(m)}{m^{1/2}}V_2\left(\frac{m}{q^2};t_g,t_f\right).
\]
\end{lm}

\begin{proof}
The proof is standard and can be found in \cite[page 98]{IwaniecKowalski2004}, 
\cite[Section 2]{Blomer2012}.
\end{proof}

\subsection{Parseval's identity}
We now fix a $T>0$ and assume that $t_f\ll T$. 
The aim of this subsection and the next is to use the special value formulae 
and the Kuznetsov formula to transform the quantity 
of interest, $\sum_{t_f\ll T}\|f_\cK|_\frak{H}\|_2^2$, into the character sums, following the method in 
\cite[Section 5]{Blomer2013}. The spectrum of $L^2(X_0(q))$ consists of the constant function $1$, the Maass forms, 
and the Eisenstein series $E_\frak{a}(\tau,1/2+it)$.
The definition and the Fourier expansion of $E_\frak{a}(\tau,1/2+it)$ can be found in 
\cite[page 1830]{Blomer2013}, \cite[pages 1187-1188]{ConreyIwnaiec2000}.
The Parseval's identity in \cite[page 391]{IwaniecKowalski2004} gives
\begin{equation}\label{E:Parseval identity}
\|f_\cK|_\frak{H}\|_2^2
=
{\rm Vol}(X_0(q))^{-1}|\langle f_\cK|_\frak{H},1\rangle|^2
+
\sum_{g\in\cB}|\langle f_\cK|_\frak{H},g\rangle|^2
+
\sum_{\frak{a}}
\frac{1}{4\pi}\int_\R
|\langle f_\cK|_\frak{H}, E_\frak{a}(\cdot,1/2+it)\rangle|^2dt
\end{equation}
where ${\rm Vol}(X_0(q))=\frac{\pi}{3}(q+1)$ is the volume of $X_0(q)$. By a result of 
Filcker (\cite[Section 5]{Flicker1988}), we have
\begin{equation}\label{E:distinguishable}
\langle f_\cK|_\frak{H},1\rangle=0.
\end{equation}
Observe that $|\langle f_\cK|_\frak{H},g\rangle|^2=0$ when $g$ is odd. This is because  it relates to
the central value $L(1/2,{\rm As}f_\cK \x g)$ by \thmref{T: special value formula}, and we have
remarked that the $L$-function $L(s,{\rm As}f_\cK \x g)$ vanishes at $s=1/2$ when $g$ is odd in 
the previous subsection. 
The following proposition concerns the contributions of the oldforms and the Eisenstein series.

\begin{prop}\label{P:first estimate}
We have
\[
\|f_\cK|_\frak{H}\|_2^2
\ll_{\cD,T,\epsilon}
\sum_{g\in\cB_q,\,g\,\text{even}}|\langle f_\cK|_\frak{H},g\rangle|^2
+
q^{-1+\epsilon}.
\]
\end{prop}

In the rest of this subsection, we are going to prove the 
\propref{P:first estimate}. 

\subsubsection{Contribution of oldforms}\label{SSS:contribution of oldforms}
We show that 
$\sum_{g\in\cB_1\cup\cB_1^*}|\langle f_\cK|_{\frak{H}},g\rangle|^2\ll_{\cD,T,\epsilon} 
q^{-1+\epsilon}$.
One may either follow the argument in \cite[page 1837]{Blomer2013}, or one can argue as follows.
Let $g\in\cB_1$.
By the functoriality of the adjoint square lift (\cite{GelbartJacquet1978}) and the non-existence of
the Siegel zeros (\cite{Banks1997}), we have 
\begin{equation}\label{E:lower bound for adjoint}
L(1,{\rm Ad}^2 f_\cK)\gg_{\cD,T,\epsilon} q^{-\epsilon}
\quad\text{and}\quad
L(1,{\rm Ad}^2 g)\gg_\epsilon (1/4+t_g^2)^{-\epsilon}
\end{equation}
according to \cite[Proposition 2.2.2 (iii)]{Molteni1999} and \cite{HoffsteinLockhart1994}.
On the other hand, it's a deep result of \cite{KimShahidi2002} that 
${\rm Ad}^2 f\x g$ corresponds to an cusp forms on ${\rm GL}_6$ and hence we can exploit the convexity
bound (\cite[Theorem 5.41]{IwaniecKowalski2004}) to obtain
\[
L(1/2,{\rm As}f_\cK\x g)
=
L(1/2,{\rm Ad}^2 f\x g)L(1/2,g\x\tau_\cK)
\ll_{\cD,\epsilon}
q^{1+\epsilon}(4+2T+|t_g|)^{1+\epsilon}.
\]
As the implicit constants are independent of $g$, we find by 
\thmref{T:special case of special formula} and \eqref{E:lower bound for adjoint} that  
\begin{align}\label{E:contribution of level one form}
\begin{split}
\sum_{g\in\cB_1}|\langle f_\cK|_{\frak{H}},g\rangle|^2
&\ll_{\cD,T,\epsilon}
q^{-1+\epsilon}\sum_{g\in\cB_1}
(1+t_g^2)^\epsilon(4+2T+|t_g|)^{2+\epsilon}
\frac{L_\infty(\frac{1}{2},{\rm Ad}^2 f\x g)L_\infty(\frac{1}{2},g\x\tau_\cK)}{L_\infty(1,{\rm Ad}^2 g)}\\
&\ll_{\cD,T,\epsilon}
q^{-1+\epsilon}. 
\end{split}
\end{align}
Here we have used the fact that $|\Gamma(x+iy)|\sim\sqrt{2\pi}|y|^{x-1/2}e^{-\pi|y|/2}$ when 
$|y|\to\infty$. Note that by the bound toward the Ramanujan conjecture for $\G$, the archimedean
local factors are non-zero complex numbers.

Recall that for $g\in\cB_1$, we have defined $g^*\in L^2(X_0(q))$ by 
\eqref{E:g*}, and $\cB_1^*=\stt{g^*\mid g\in\cB_1}$. Let $g'(\tau):=g(\sigma_0\cdot\tau)=g(q\tau)$.
By \lmref{L:Atkin-Lehnner for base change} and changing the variable, we see that 
$\langle f_\cK|_{\frak{H}}, g'\rangle=\pm\langle f_\cK|_{\frak{H}}, g\rangle$.  Therefore
the formula in \thmref{T:special case of special formula} remains the same if we 
replace $g$ by $g'$. Because $\left(1-\frac{q\lambda_g(q)^2}{(q+1)^2}\right)\asymp 1$, we have
\begin{equation}\label{E:contribution of raising form}
\sum_{g\in\cB_1}
|\langle f_\cK|_{\frak{H}},g^*\rangle|^2
\ll
2\sum_{g\in\cB_1}
|\langle f_\cK|_{\frak{H}},g\rangle|^2
+2\sum_{g\in\cB_1}
|\langle f_\cK|_{\frak{H}},g'\rangle|^2
\ll_{\cD,T,\epsilon} q^{-1+\epsilon}
\end{equation}
by  \eqref{E:contribution of level one form}. Thus we find that 
$\sum_{g\in\cB_1\cup\cB_1^*}|\langle f_\cK|_{\frak{H}},g\rangle|^2\ll_{\cD,T,\epsilon} 
q^{-1+\epsilon}$ as desired.

\subsubsection{Contribution of Eisenstein series}
In \cite[page 1838]{Blomer2013}, the contribution of Eisenstein series relates to 
the Rankin-Selberg $L$-function $L(s,f\x\b{f})$. In this article; however, it relates to the $Asai$ 
$L$-$function$ $L(s,{\rm As}f_\cK)$ (\cite{Asai1977}, \cite{Flicker1988}). 
By \lmref{L:Atkin-Lehnner for base change} and the fact that  $\sigma_0$ 
is a scaling matrix for the cusp $\frak{a}=0$, the contributions of the two cusps are essentially 
the same, and hence it suffices to consider the contribution of the $\infty$ cusp. 
  
\begin{lm}\label{L:unfold Asai L-fcn}
We have
\[
\ol{\langle f_\cK|_{\frak{H}},E_\infty(\cdot,s)\rangle}
=
2^{-2}\ol{\rho_{f_\cK}}|\cD|^{s/2}\cdot
\frac{\zeta_{\Q_q}(2s)}{L(s,\tau_q)}\cdot
\frac{\Lambda(s,{\rm As}f_\cK)}{\xi_\Q(2s)}.
\]
Here $\tau_q$ is the quadratic character of $\Q_q^{\x}$ associated to $\cK_q/\Q_q$ by the local
class field theory. 
\end{lm}

\begin{proof}
Suppose that $\cD<0$. By the Fourier expansion \eqref{E:Fourier expansion for D<0} of $f_\cK$ and 
the fact that $\lambda_{f_\cK}(\alpha)\in\R$ for all $0\neq\alpha\in\cO$, one has for $\tau=x+iy$
\begin{align*}
\ol{\langle f_\cK|_{\frak{H}},E_\infty(\cdot,s)\rangle}
&=
\int_{X_0(q)}
\ol{f_\cK|_\frak{H}(\tau)}E_\infty(\tau,s)\,d\mu(\tau)
=
\int_{\Gamma_\infty\backslash\frak{H}}
\ol{f_\cK|_\frak{H}(\tau)}y^{s-2}dxdy\\
&=
\ol{\rho_{f_\cK}}\sum_{0\neq\alpha\in\cO}
\lambda_{f_\cK}(\alpha)
\int_0^\infty
K_{2it_f}(4\pi|\alpha|^{1/2}_\C|\cD|^{-1/2}y)y^{s-1}dy
\int_0^1 e^{-4\pi i|\cD|^{-1/2}{\rm Im}(\alpha)x}dx.
\end{align*}
Since $2|\cD|^{-1/2}{\rm Im}(\alpha)\in\Z$ for every $\alpha\in\cO$, it follows that the inner integral vanishes
when ${\rm Im}(\alpha)\neq 0$ and is equal to $1$ otherwise. On the other hand, for $0\neq\alpha=n\in\Z$, we have
\[
\int_0^\infty
K_{2it_f}(4\pi |n|^{1/2}_\C|\cD|^{-1/2}y)y^{s-1}dy
=
2^{-3}|n|^{-s}|\cD|^{s/2}\frac{L_\infty(s,{\rm As}\,f_\cK)}{\zeta_\R(2s)}
\]
by \lmref{L:integration formula for Bessel function} (1) and the shape 
$L_\infty(s,{\rm As}\,f_\cK)=\zeta_\R(s+2it_f)\zeta_\R(s-2it_f)\zeta_\C(s)$ of the Asai $L$-factor at $\infty$.
Using the fact that 
$\lambda_{f_\cK}(-\alpha)=\lambda_{f_\cK}(\alpha)$ for all $0\neq\alpha\in\cO$, one can argue as in 
\cite[Theorem 2]{Asai1977} to obtain
\begin{align*}
\ol{\langle f_\cK|_{\frak{H}},E_\infty(\cdot,s)\rangle}
=
2^{-2}\ol{\rho_{f_\cK}}|\cD|^{s/2}\frac{L_\infty(s,{\rm As}\,f_\cK)}{\zeta_\R(2s)}
\sum_{n=1}^\infty
\frac{\lambda_{f_\cK}(n)}{n^s}
=
2^{-2}\ol{\rho_{f_\cK}}|\cD|^{s/2}
\cdot
\frac{\zeta_{\Q_q}(s+1)\Lambda^{(q)}(s,{\rm As}f_\cK)}{\xi^{(q)}_\Q(2s)}.
\end{align*} 
Here $\Lambda^{(q)}(s,{\rm As}\,f_\cK)$ and $\xi^{(q)}(2s)$ denote the partial $L$-functions without the local $L$-factors
at $q$. To conclude,  we remind that 
\[
L_q(s,{\rm As}f_\cK)=\zeta_{\Q_q}(s+1)L(s,\tau_q).
\]
The computation for the case $\cD>0$ is similar so we left it to the reader.
\end{proof}

The $L$-function $L(s,{\rm As}f_\cK)$ has the factorization
\[
L(s,{\rm As}f_\cK)=L(s,{\rm Ad}^2 f)L(s,\tau_\cK).
\]
By \lmref{L:norm of f_K}, \lmref{L:unfold Asai L-fcn} and the convexity bound 
(\cite[Theorem 5.41]{IwaniecKowalski2004}) of $L(s,{\rm Ad}^2 f)$ at $s=1/2+it$ for $t\in\R$
together with the lower bound \eqref{E:lower bound for adjoint} of $L(1,{\rm Ad}^2 f_\cK)$, we obtain 
\begin{equation}\label{E:contribution of Eisenstein series}
\int_\R
|\langle f_\cK|_\frak{H},E_\infty(\cdot,1/2+it)\rangle|^2dt
\ll_{\cD,T,\epsilon}
q^{-2+\epsilon}\int_\R
\frac{|\Lambda(\frac{1}{2}+it, {\rm As}\,f_\cK)|^2}{|\xi_\Q(1+2it)|^2}dt
\ll_{\cD,T,\epsilon}
q^{-1+\epsilon}.
\end{equation}
Now \propref{P:first estimate} follows from \eqref{E:contribution of level one form}, 
\eqref{E:contribution of raising form} and \eqref{E:contribution of Eisenstein series}.

\subsection{The main terms}
Recall the Wely's law $\stt{f\in\cB_q\mid |t_f|\leq T}\ll qT^2$. It then follows from 
\propref{P:first estimate} that
\[
\sum_{\underset{t_f\ll T}{f\in\cB_q}}
\|f_\cK|_\frak{H}\|_2^2\ll_{\cD,T,\epsilon}
\sum_{\underset{t_f\ll T}{f\in\cB_q}}\sum_{\underset{g\,\text{even}}{g\in\cB_q}}
|\langle f_\cK|_\frak{H},g\rangle|^2
+q^\epsilon,
\]
and hence \thmref{T:main result} will follow from the following proposition.
\begin{prop}\label{P:main term estimate}
We have
\[
\sum_{\underset{t_f\ll T}{f\in\cB_q}}\sum_{\underset{g\,\text{even}}{g\in\cB_q}}
|\langle f_\cK|_\frak{H},g\rangle|^2
\ll_{\cD,T,\epsilon}
q^\epsilon.
\]
\end{prop}
The rest of this subsection is devoted to prove \propref{P:main term estimate}.

\subsubsection{Main transformation}
Let $f,g\in\cB_q$ with $g$ even. 
Note that $L(1,{\rm Ad}^2 f\x\tau_\cK)>0$ since both $L(1,{\rm Ad}^2 f_\cK)$ and 
$L(1,{\rm Ad}^2 f)$ are positive. Furthermore, we have a lower bound
\[
L(1,{\rm Ad}^2 f\x\tau_\cK)\gg_{\cD,T,\epsilon}q^{-\epsilon}.
\]
By \thmref{T:special case of special formula} and the factorization 
\eqref{E:factorization adjoint L-fcn for base change} and the above lower bound, 
we find that \footnote{The term $e^{-\frac{3\pi}{2}|t_g|}$ in 
\cite[page 1838]{Blomer2013} should be replaced by $e^{-\pi|t_g|}$; however, this does not affect the
proofs.}
\begin{align*}
\sum_{\underset{t_f\ll T}{f\in\cB_q}}\sum_{\underset{g\,\text{even}}{g\in\cB_q}}
|\langle f_\cK|_\frak{H},g\rangle|^2
&\ll_{\cD,T,\epsilon}
q^{-2+\epsilon}
\sum_{\underset{t_f\ll T}{f\in\cB_q}}\sum_{\underset{g\,\text{even}}{g\in\cB_q}}
\frac{L(\frac{1}{2},{\rm As}f_\cK\x g)}{L(1,{\rm Ad}^2 f)L(1,{\rm Ad}^2 g)}e^{-\pi|t_g|}\\
&\ll_{\cD,T,\epsilon}
q^{-2+\epsilon}
\sum_{\underset{t_f\ll T}{f\in\cB_q}}\sum_{\underset{g\,\text{even}}{g\in\cB_q}}
\frac{L(\frac{1}{2},{\rm As}f_\cK\x g)}{L(1,{\rm Ad}^2 f)L(1,{\rm Ad}^2 g)}
G_1(0,t_g)G_2(0,t_g,t_f)h(t_g).
\end{align*}
Here we have inserted artificially the factor $G_1(0,t_g)G_2(0,t_g,t_f)$ by the inequality
\eqref{E:lower bound for G_1, G_2} and the non-negativity of $L(1/2,{\rm As}f_\cK\x g)$.
We also replace the weight function $e^{-\pi|t|}$ by the function
\[
h(t):={\rm cosh}\left(\frac{t}{2A}\right)^{-2\pi A}
\prod_{j=0}^A\left(t^2+\left(\frac{1}{2}+j\right)^2\right).
\] 
Note that $h(t)$ is holomorphic in $|{\rm Im}(t)|<\pi A$ and has zeros at those of 
$\prod_{j=0}^A\left(t^2+\left(\frac{1}{2}+j\right)^2\right)$ in this region. Moreover,
$h(t)\gg e^{-\pi|t|}$ for $t\in\cT$. Applying \lmref{L:weighted sum of L-fcn}, we arrive at
\begin{equation}\label{E:main term}
\sum_{\underset{t_f\ll T}{f\in\cB_q}}\sum_{\underset{g\,\text{even}}{g\in\cB_q}}
|\langle f_\cK|_\frak{H},g\rangle|^2
\ll_{\cD,T,\epsilon}
q^\epsilon
\sum_{f\in\cB_q}
\frac{2h(t_f)}{q'L(1,{\rm Ad}^2 f)}
\sum_{\underset{g\,\text{even}}{g\in\cB_q}}
\frac{2h(t_g)}{q'L(1,{\rm Ad}^2 g)}
(1-\lambda_g(q)q^{1/2})S
\end{equation}
where $q'=q^2/(q-1)\asymp q$ and 
\[
S=\sum_{n,m=0}^\infty
\frac{\lambda_g(n)\tau_\cK(n)\lambda_{{\rm Ad}^2 f\x g}(m)}{(mn)^{1/2}}
V_1\left(\frac{n}{q^{1/2}|\cD|};t_g\right)
V_2\left(\frac{m}{q^2};t_g,t_f\right).
\]
The main difference between our case and the case in \cite[page 1838]{Blomer2013}
is that we have the additional factors $\tau_\cK(n)$; however, this difference is mild and 
the proof in \cite{Blomer2013} still works in our case. 
It will be convenient to remove the terms with $q\mid nm$ in $S$. In fact, as argued in 
\cite[page 1839]{Blomer2013}, the terms $n,m$ with $q\mid n$ and $q\mid m$ in $S$
contribute at most $O(q^{-1/4+\epsilon})$. Thus by 
\cite[(3.1)]{Blomer2013} and \eqref{E:main term}, 
we are left with estimating $\Sigma(q,q):=\Sigma_1(q,q)-\Sigma_2(q,q)$ where
\begin{align*}
\Sigma_1(q,q)
:=&
\sum_{f\in\cB_q}
\frac{2h(t_f)}{q'L(1,{\rm Ad}^2 f)}
\sum_{\underset{g\,\text{even}}{g\in\cB_q}}
\frac{2h(t_g)}{q'L(1,{\rm Ad}^2 g)}
\sum_{q\nmid abdk}
\frac{\mu(d)}{ab^2d^{3/2}k}\\
&\x
\sum_{q\nmid nm}
\frac{\tau_\cK(n)\lambda_g(n)\lambda_g(a^2d m)\lambda_f(k^2)\lambda_f(m^2)}{(nm)^{1/2}}
V_1\left(\frac{n}{q^{1/2}|\cD|};t_g\right)
V_2\left(\frac{a^2b^4d^3mk^2}{q^2};t_g,t_f\right)
\end{align*}
and
\begin{align*}
\Sigma_2(q,q)
:=&
q^{1/2}\sum_{f\in\cB_q}
\frac{2h(t_f)}{q'L(1,{\rm Ad}^2 f)}
\sum_{\underset{g\,\text{even}}{g\in\cB_q}}
\frac{2h(t_g)}{q'L(1,{\rm Ad}^2 g)}
\sum_{q\nmid abdk}
\frac{\mu(d)}{ab^2d^{3/2}k}\\
&\x
\sum_{q\nmid nm}
\frac{\tau_\cK(n)\lambda_g(qn)\lambda_g(a^2d m)\lambda_f(k^2)\lambda_f(m^2)}{(nm)^{1/2}}
V_1\left(\frac{n}{q^{1/2}|\cD|};t_g\right)
V_2\left(\frac{a^2b^4d^3mk^2}{q^2};t_g,t_f\right).
\end{align*}
We would like to apply the formulae in \cite[(2.16)-(2.19)]{Blomer2013} to the spectral sum
over $f$ and $g$ in $\Sigma_1(q,q)$ and $\Sigma_2(q,q)$; however, since the terms 
involving the oldforms
and the Eisenstein series are missing, we complete the formula by adding and subtracting the missing 
terms. This gives us $8$ other quantities $\Sigma(*,?)=\Sigma_1(*,?)-\Sigma_2(*,?)$ with 
$*,?\in\stt{q,1,\cE}$ in analogy with $\Sigma(q,q)=\Sigma_1(q,q)-\Sigma_2(q,q)$. For example,
we define 
\begin{align*}
\Sigma_1(1,\cE)
:=&
\sum_{f\in\cB_1}
\frac{2h(t_f)}{(q+1)L(1,{\rm Ad}^2 f)}
\int_{\R}
\frac{h(t)}{q''|\zeta^{(q)}_\Q(1+2it)|^2}
\sum_{q\nmid abdk}
\frac{\mu(d)}{ab^2d^{3/2}k}\\
&\x
\sum_{q\nmid nm}
\frac{\tau_\cK(n)c_1(f,q)\lambda_f(k^2)\lambda_f(m^2)\eta(n,t)\eta(a^2dm,-t)}{(nm)^{1/2}}
V_1\left(\frac{n}{q^{1/2}|\cD|};t\right)
V_2\left(\frac{a^2b^4d^3mk^2}{q^2};t,t_f\right)
\frac{dt}{\pi}
\end{align*}
and
\begin{align*}
\Sigma_2(1,\cE)
:=&
\sum_{f\in\cB_1}
\frac{2h(t_f)}{(q+1)L(1,{\rm Ad}^2 f)}
\int_{\R}
\frac{h(t)}{q^{3/2}|\zeta^{(q)}_\Q(1+2it)|^2}
\sum_{q\nmid abdk}
\frac{\mu(d)}{ab^2d^{3/2}k}\\
&\x
\sum_{q\nmid nm}
\frac{\tau_\cK(n)c_2(f,q)\lambda_f(k^2)\lambda_f(m^2)\eta(n,t)\eta(a^2dm,-t)}{(nm)^{1/2}}
V_1\left(\frac{n}{q^{1/2}|\cD|};t\right)
V_2\left(\frac{a^2b^4d^3mk^2}{q^2};t,t_f\right)
\frac{dt}{\pi}
\end{align*}
and similarly for all the combinations.
Here $\eta(n,t)=\sum_{ad=|n|}(a/d)^{it}\in\R$ appears in the Fourier coefficients of 
the Eisenstein series as in \cite[page 1830]{Blomer2013} and $q''=q^2/(q+1)$. 
The terms $c_1(f,q)$, $c_2(f,q)$ are given by \cite[(2.11), (2.12)]{Blomer2013} with 
$g$ replaced by $f$. Since $c_1(f,q)\asymp 1$ and $c_2(f,q)\ll 1$, these terms certainly can be 
ignored in the estimates below. Now the Kuznetsov formula can be applied and we obtain
\begin{equation}\label{E:sigma(q,q)}
\Sigma(q,q)
=-\sum_{(*,?)\in\stt{q,1,\cE}^2,\,(*,?)\neq (q,q)}
(\Sigma_1(*,?)-\Sigma_2(*,?))
+
\sum_{q\nmid abdkmn}
\frac{\mu(d)\tau_\cK(n)}{(nm)^{1/2}ab^2d^{3/2}k}
\sum_{\alpha,\beta,\gamma\in\stt{1,2}}M_{\alpha}^{\beta,\gamma}
\end{equation} 
where the terms $M_\alpha^{\beta,\gamma}$ are exactly those in \cite[pages 1840-1841]{Blomer2013} 
with $\frac{n}{q^{1/2}}$ replaced by $\frac{n}{q^{1/2}|\cD|}$.
To complete the proof of  \propref{P:main term estimate}, we show that the terms on the RHS of
\eqref{E:sigma(q,q)} are all $O(q^\epsilon)$. 

\subsubsection{Contributions of oldforms and Eisenstein series}
\label{SSS:contribution of oldforms and Eisenstein series}
We consider the terms $\Sigma_1(*,?)$ and $\Sigma_2(*,?)$ on the RHS of \eqref{E:sigma(q,q)}.
As in \cite[Section 6]{Blomer2013}, the idea is to use the inverse Mellin transforms and 
the convexity bound of the $L$-functions. 
We consider the cases $\Sigma_1(\cE,q)$ and $\Sigma_2(1,\cE)$ as the other cases 
require only notational changes. 
By the inverse Mellin transform, we find that (see the remark in \cite[page 1833]{Blomer2013})
\begin{align*}
\Sigma_1(\cE,q)
=&
\int_\R
\frac{h(t)}{q''|\zeta_\Q^{(q)}(1+2it)|^2}
\sum_{\underset{g\,\text{even}}{g\in\cB_q}}
\frac{2h(t_g)}{q'L(1,{\rm Ad}^2 g)}
\int_{2-i\infty}^{2+i\infty}\int_{2-i\infty}^{2+i\infty}
\prod_{\epsilon_1\in\stt{\pm, 0}}
L^{(q)}(1/2+2\epsilon_1 it+u,g)\\
&
\x L^{(q)}(1/2+v, g\x\tau_\cK)
\wh{V}_1(v;t_g)\wh{V}_2(u;t_g,t)q^{2u+v/2}|\cD|^v\frac{dudv}{(2\pi i)^2}\frac{dt}{\pi}.
\end{align*}
We shift both contours to ${\rm Re}(u)={\rm Re}(v)=\epsilon$ and use the convexity
bound (\cite[Theorem 5.41]{IwaniecKowalski2004}) of $L(s,g)$ and $L(s,g\x\tau_\cK)$ in 
${\rm Re}(s)>1/2$ together with \cite[Lemma 1]{Blomer2013} to deduce that
$\Sigma_1(\cE,q)\ll_{\cD,\epsilon} q^{\epsilon}$. 
Similar argument shows
\begin{align*}
\Sigma_2(1,\cE)
=&
\sum_{f\in\cB_1}
\frac{2c_2(f,q)h(t_f)}{(q+1)L(1,{\rm Ad}^2 f)}
\int_{\R}
\frac{h(t)}{q^{3/2}|\zeta^{(q)}_\Q(1+2it)|^2}
\int_{2-i\infty}^{2+i\infty}\int_{2-i\infty}^{2+i\infty}
\prod_{\epsilon_1\in\stt{\pm}}
L^{(q)}(1/2+\epsilon_1 it+u,{\rm Ad}^2 f)\\
&\x
\prod_{\epsilon_2\in\stt{\pm}}
L^{(q)}(1/2+\epsilon_2 it+v,\tau_\cK)
\wh{V}_1(v;t)\wh{V}_2(u;t,t_f)q^{2u+v/2}|\cD|^v\frac{dudv}{(2\pi i)^2}\frac{dt}{\pi}.
\end{align*}
We also shift both contours to ${\rm Re}(u)={\rm Re}(v)=\epsilon$ to deduce 
$\Sigma_2(1,\cE)\ll_{\cD,T,\epsilon} q^{-5/2+\epsilon}$. Unlike the case given in 
\cite[section 6]{Blomer2013}, where he needs $\wh{V}_1(1/2\pm it, t)=0$ to kill poles 
of the Riemann zeta function, the $L$-function $L(s,\tau_\cK)$ in our case is $entire$. 
Therefore the above argument still works in our case even when our $V_1$ does not satisfy
$\wh{V}_1(1/2\pm it, t)=0$ when $\cD<0$.

\subsubsection{Estimating character sums}
We estimate the terms $M_\alpha^{\beta,\gamma}$ on the RHS of \eqref{E:sigma(q,q)}. We explain how
to modify the proof in \cite[Section 8]{Blomer2013} so that it can apply to our case.
The bound
\[
\sum_{q\nmid abdkmn}
\frac{1}{(nm)^{1/2}ab^2d^{3/2}k}(|M_1^{1,1}|+|M_1^{1,2}|+|M_1^{2,1}|)\ll_{\cD,\epsilon} q^\epsilon
\]
in \cite[page 1845]{Blomer2013} certainly holds in our case. Note that $M^{1,1}_2=M^{2,1}_2=0$.
For the terms $M_2^{1,2}$, $M_1^{2,2}$ and $M_2^{2,2}$, the idea is to replace $\tau_\cK$ by any bounded 
smooth function $\tau^*_\cK$ on $\R$ such that $\tau^*_\cK(n)=\tau_\cK(n)$ for every $n\in\Z$. Then the proof 
in \cite[Section 8]{Blomer2013} works in our case. We illustrate this by considering the term 
$M_2^{1,2}$. Recall that we need to estimate
\[
\sum_{q\nmid abdkmn}
\frac{\mu(d)\tau_\cK(n)M_2^{1,2}}{ab^2 d^{3/2}m^{3/2}n^{1/2}}
=
q^{1/2}\sum_{q\nmid abdmn}\sum_{q\mid c}
\frac{\mu(d)\tau_\cK(n)S(qn,a^2dm,c)}{ab^2cd^{3/2}m^{3/2}n^{1/2}}
\cW^{1,2}\left(\frac{n}{q^{1/2}|\cD|},\frac{a^2b^4d^3m^3}{q^2};\frac{a\sqrt{dmnq}}{c}\right).
\]
As argued in \cite[page 1845]{Blomer2013}, we need to bound
\begin{align*}
\frac{1}{q^{1/2}N^{1/2}}
\sum_{q\nmid abcdm}
\frac{w_2(a/A)w_3(c/C)}{ab^2cd^{3/2}m^{3/2}}
\left|
\sum_{n}S(n,a^2dm\b{q},c)\tau_\cK(n)w_1(n/N)
\cW^{1,2}\left(\frac{n}{q^{1/2}|\cD|},\frac{a^2b^4d^3m^3}{q^2};\frac{a\sqrt{dmx}}{cq^{1/2}}\right)
\right|
\end{align*}
where $w_1$, $w_2$ and $w_3$ are smooth weight functions with supports in $[1,2]$ and we may 
assume that
\[
a\sqrt{dmnq}\leq q^{7/4+\epsilon},
\quad
A\leq\frac{q^{1+\epsilon}}{b^2(dm)^{3/2}},
\quad
C\leq\frac{A\sqrt{dmN}}{q^{1/2-\epsilon}}
\] 
by the decay property of $\cW^{1,2}$ in \cite[Lemma 5]{Blomer2013} . Using \cite[Lemma 2]{Blomer2013} with
$\alpha=a^2dm\b{q}$ and $\gamma=c$, the $n$-sum is less than
\[
\sum_{h\neq 0}
\left|
\int_0^\infty
w^*_1\left(\frac{x}{N}\right)
\cW^{1,2}\left(\frac{x}{q^{1/2}|\cD|},\frac{a^2b^4d^3m^3}{q^2};\frac{a\sqrt{dmx}}{cq^{1/2}}\right)
{\rm exp}\left(-2\pi i\frac{xh}{c}\right)dx
\right|.
\]
Here $w^*_1(x):=w_1(x)\tau^*_\cK(Nx)$ is again a smooth function with support in $[1,2]$. Therefore
the bound in \cite[(7.4)]{Blomer2013} is still applicable (with $\rho_1=\frac{1}{q^{1/2}|\cD|}$) 
and the rest of the proof is identical to that of \cite[page 1846]{Blomer2013}. 
The modifications of the proofs for the terms $M_1^{2,2}$ and 
$M_2^{2,2}$ are similar. In this way, the proof of \propref{P:main term estimate} and hence
\thmref{T:main result} is completed.

\section{Special value formula}\label{S:special value formula}
In this section, we state our special value formulae for the triple product $L$-functions 
in the adelic setting. 



\subsection{Notations}
We introduce some notations which will be constantly used in the rest of this article. Let $F$ be 
a number field or a local field. We denote by $E$ an \etale cubic $F$-algebra
and by $D$ a quaternion $F$-algebra. Put
\begin{equation}\label{E:c_E}
c_E=\begin{cases}
   3\quad\text{if}\,\,E=F\x F\x F,\\
   2\quad\text{if}\,\,E=F'\x F,\,\text{where $F'$ is a quadratic field extension of $F$},\\
   1\quad\text{if}\,\,E\,\,\text{is a cubic field extension of $F$}.
 \end{cases} 
\end{equation}
Let $R$ be a $F$-algebra. If $F$ is a local field, we put $D(R)=D\ot_F R$. If $F$ is a number field, 
we let $R_v=F_v\ot_F R$, $\A_R=\A_F\ot_F R$ and $D(\A_R)=D\ot_F\A_R$. When $F$ is a number
field or a non-archimedean local field, the maximal order of $E$ (resp. $D(E)$) is denoted by $\cO_E$
 (resp. $\cO_{D(E)}$).

\subsection{Ichino's formula and root number}
We briefly review (a variant of) Ichino's formula as well as the root number for the triple product $L$-function 
in this subsection. The story of the triple product $L$-function began with the work of Garrett
(\cite{Garrett1987}) in which he gave an integral representation of the triple product $L$-function 
attached to three elliptic newforms. Piatetski-Shapiro, Rallis (\cite{PSR1987}) and 
Ikeda (\cite{Ikeda1989}, \cite{Ikeda1992}) later reformulated his work in the adelic setting. 
Using the integral representation of the triple product $L$-function and the Siegel-Weil formula, 
Harris and Kudla proved the Jacquet's conjecture for the triple product $L$-function in
\cite{HarrisKudla1991}, \cite{HarrisKudla2004}. See also \cite{Prasad2008}. 
Based on their works, Ichino gave a formula which relates
certain global period integral with the central value of the triple product $L$-function. 
Indeed, Ichino's formula is a special case of the so-called refined Gross-Prasad conjecture formulated in
\cite{IchinoIkeda2010}.  

\subsubsection{Ichino's formula}\label{SSS:Ichino formula}
Let $F$ be a number field and $\itPi$ be an irreducible unitary cuspidal automorphic representation 
of $\G(\A_E)$. We assume that
\begin{itemize}
\item
the central character of $\itPi$ is trivial on $\A_F^{\x}$;
\item
the Jacquet-Langlands lift $\itPi^D$ of $\itPi$ to $D^{\x}(\A_E)$ exists. 
\end{itemize} 
Define the global period integral by
\[
I(\phi\otimes\t{\phi})
=
\int_{\A_F^\x D^\x(F)\backslash D^\x(\A_F)}
\int_{\A_F^\x D^\x(F)\backslash D^\x(\A_F)}
\phi(x)\t{\phi}(\t{x})d^{\x}x\,d^{\x}\t{x}
\]
for $\phi\in\itPi^D$ and $\t{\phi}\in\t{\itPi}^D$, where $\t{\itPi}^D$ is the 
contragredient representation of $\itPi^D$, and $d^{\x}x$ and $d^{\x}\t{x}$ are the 
Tamagawa measures on $\A_F^\x\backslash D^\x(\A_F)$.
For each place $v$ of $F$, we choose a $D^\x(E_v)$-equivariant pairing 
$\langle\,\,,\rangle_v$ between $\itPi^D_v$ and $\t{\itPi}^D_v$. Define the local period integral by
\[
I_v(\phi_v\otimes\t{\phi}_v)
=
\frac{\zeta_{F_v}(2)}{\zeta_{E_v}(2)}\cdot
\frac{L_v(1, \itPi_v, {\rm Ad})}{L_v(\frac{1}{2},\itPi_v, r)}\cdot
\int_{F_v^\x\backslash D^\x(F_v)}
\langle\itPi^D_v(x_v)\phi_v,\,\t{\phi}_v\rangle_v d^{\x}x_v  
\]
for $\phi_v\in\itPi_v^D$ and $\t{\phi}_v\in\t{\itPi}^D_v$. For the definitions of the $L$-factors, 
see \cite[Pages 282-283]{Ichino2008}. In particular, $L(s,\itPi_v, r)$ is the local factor of the 
triple product $L$-function attached to $\itPi$, and the $L$-functions appearing in this
article are those defined from the Galois theoretic side.

Let $\langle\,\,,\rangle$ be the $D^\x(\A_E)$-equivariant pairing between
$\itPi^D$ and $\t{\itPi}^D$ defined by
\[
\langle \phi, \t{\phi}\rangle
=
\int_{\A^\x_E D^\x(E)\backslash D^\x(\A_E)}
\phi(h)\t{\phi}(h)d^{\x}h
\]
for $\phi\in\itPi^D$ and $\t{\phi}\in\t{\itPi}^D$, where $d^{\x}h$ is the Tamagawa measure
on $\A^\x_E\backslash D^\x(\A_E)$.
Assume that the Haar measures $dx_v$ are chosen so that the volume of  
${\rm P}\G(\cO_{F_v})$ is $1$ for almost all finite places of $F$.

\begin{thm}[\cite{Ichino2008}]\label{T:Ichino formula}
Let $\phi_j=\otimes_v\phi_{j,v}\in\itPi^D$, $\t{\phi}_j=\otimes_{v}\t{\phi}_{j,v}\in\t{\itPi}^D$ with 
$\langle\phi_j,\t{\phi}_j\rangle\neq 0$ for $j=1,2$. Then we have
\[
\frac{I(\phi_1\otimes\t{\phi}_1)}{\langle\phi_2,\t{\phi}_2\rangle}=
\frac{C}{2^{c_E}}\cdot
\frac{\xi_E(2)}{\xi_F(2)}\cdot
\frac{\Lambda(\frac{1}{2},\itPi, r)}{\Lambda(1,\itPi,{\rm Ad})}\cdot
\prod_{v}\frac{I_v(\phi_{1,v}\otimes\t{\phi}_{1,v})}{\langle\phi_{2,v},\t{\phi}_{2,v}\rangle_v}                                                     
\]
where $v$ runs through all places of $F$ and $C>0$ is the constant so that $d^{\x}x=C\prod_v d^{\x}x_v$.
\end{thm}

\begin{proof}
The case $\phi_1=\phi_2$, $\t{\phi}_1=\t{\phi}_2$ is the original Ichino's formula, while the formula stated here is a
corollary of that, which can be argued as follows. Fix isomorphisms $\itPi^D\cong\ot'_v\itPi^D_v$  and 
$\t{\itPi}^D\cong\ot'_v\t{\itPi}^D_v$, where $v$ runs 
through all places of $F$. Let $S$ be a finite set of places of $F$ containing all archimedean places so that for all 
$v\notin S$, we have
\begin{itemize}
\item
$D_v\cong{\rm M}_2(F_v)$;
\item
$\itPi_v$ and $\t{\itPi}_v$ are unramified;
\item
$\phi_{1,v}=\phi_{2,v}=\phi_v^0$ and $\t{\phi}_{1,v}=\t{\phi}_{2,v}=\t{\phi}_v^0$, where $\phi_v^0$, $\t{\phi}^0_v$
are the fixed spherical elements. 
\end{itemize}
For each $v$, define a $D^\x(E_v)$-equivariant pairing $\langle\,\,,\rangle'_v$ between $\itPi^D_v$ and $\t{\itPi}^D_v$
by $\langle\,\,,\rangle'_v=\langle\,\,,\rangle_v$ when $v\in S$, and 
$\langle\,\,,\rangle'_v=c^{-1}_v\langle\,\,,\rangle_v$ when $v\notin S$, where 
$c_v:=\langle\phi_v^0,\t{\phi}_v^0\rangle_v\neq 0$. Now $\prod_v\langle\,\,,\rangle'_v$ gives rise to a well-defined 
$D^\x(\A_E)$-equivariant pairing between $\itPi^D$ and $\t{\itPi}^D$ via the fixed isomorphisms, and hence there is a 
constant $c>0$ such that $\langle\,\,,\rangle=c\prod_v\langle\,\,,\rangle'_v$. It then follows from the choice of $S$ that 
\begin{equation}\label{E:ratio of norm}
\langle \phi_j,\t{\phi}_j\rangle=c\prod_{v\in S}\langle\phi_{j,v},\t{\phi}_{j,v}\rangle_v
\end{equation}
for $j=1,2$. On the other hand, if we let $I'_v(\phi_{1,v}\ot\t{\phi}_{1,v})$ be the local period integral defined as above
with $\langle\,\,,\rangle_v$ replaced by $\langle\,\,,\rangle'_v$ for each $v$, then we have 
\begin{equation}\label{E:variant Ichino formula}
\frac{I(\phi_1\otimes\t{\phi}_1)}{\langle\phi_1,\t{\phi}_1\rangle}=
\frac{C}{2^{c_E}}\cdot
\frac{\xi_E(2)}{\xi_F(2)}\cdot
\frac{\Lambda(\frac{1}{2},\itPi, r)}{\Lambda(1,\itPi,{\rm Ad})}\cdot
\prod_{v}I'_v(\phi_{1,v}\otimes\t{\phi}_{1,v})\cdot\prod_{v\in S} \langle\phi_{1,v},\t{\phi}_{1,v}\rangle^{-1}_v                                                    
\end{equation}
by \cite[Remark 1.3]{Ichino2008}. Note that the infinite product $\prod_{v}I'_v(\phi_{1,v}\otimes\t{\phi}_{1,v})$
is well-defined by \cite[Lemma 2.2]{Ichino2008}. The theorem now follows from 
\eqref{E:ratio of norm} and \eqref{E:variant Ichino formula}.
\end{proof}

\begin{Remark}\label{R:remark on Ichino's formula}
Since $\itPi^D$ is unitary, we have $\t{\itPi}^D\cong\b{\itPi}^D$, where $\b{\itPi}^D$ is the conjugate representation of 
$\itPi^D$. By the multiplicity one property, one sees that the space of $\t{\itPi}^D$ consists of the automorphic forms
$\b{\phi}$, for some $\phi\in\itPi^D$. As a consequence, one can replace the bilinear pairing between $\itPi^D$ and 
$\t{\itPi}^D$ (resp. $\itPi^D_v$ and $\t{\itPi}^D_v$) by the hermitian pairing on $\itPi^D$ (resp. $\itPi^D_v$) in Ichino's
formula, and one still get the same result.
\end{Remark}

\subsubsection{Root number}\label{SSS:root number}
Let $F$ be a local field and 
$\itPi$ be an irreducible admissible generic representation of $\G(E)$ whose central character is 
assumed to be trivial on $F^{\x}$. Suppose that $D$ is the 
division algebra, and let $\itPi^D$ be the local Jacquet-Langlands lift of $\itPi$ to $D^{\x}$. Thus 
$\itPi^D=0$ if $\itPi$ is not a discrete series representation of $\G(E)$. We fix a non-trivial 
additive character $\psi$ of $F$. Let $\omega_{E/F}$ be the quadratic character of $F^{\x}$ defined 
in \cite[Pages 1318-1319]{Prasad1992}
\footnote{The quadratic character $\omega_{E/F}$ in \cite{Prasad1992} is defined 
when $F$ is a non-archimedean local field; however, we can define $\omega_{E/F}$ in a 
similar way when $F$ is an archimedean local field.}.
Note that $\omega_{E/F}$ is trivial when $c_E=3$.
Let $\epsilon(s,\itPi,r,\psi)$ be the $\epsilon$-factor of the triple product 
$L$-function associated to $\itPi$ and $\psi$ defined from the Galois theoretic side 
(\cite[Introduction]{Ichino2008}, \cite[Section 3]{Tate1979}).
\footnote{Our $\epsilon$-factor is the $\epsilon_L$ given in \cite[(3.6)]{Tate1979}.}
In the following theorem, by a $\G(F)$-equivariant form on $\itPi$ when $F$ is archimedean, we 
actually mean an $(\frak{g},K)$-equivariant form on $\itPi$. The following theorem is essentially due to
Prasad, Loke and Gan.

\begin{thm}
\label{T:root number}
Suppose that $F\neq \C$.  
\begin{itemize}
\item[(1)]
We have
${\rm dim}_\C{\rm Hom}_{\G(F)}(\itPi,\C)+{\rm dim}_\C{\rm Hom}_{D^{\x}}(\itPi^D,\C)=1$.
\item[(2)]
The value $\epsilon(1/2,\itPi,r,\psi)\omega_{E/F}(-1)\in\stt{\pm 1}$ is independent of 
the choice of $\psi$. 
\item[(3)]
Assume that $\itPi$ is not a supercuspidal representation when $E$ is a field. Then we have
\[
\epsilon(1/2,\itPi,r,\psi)\omega_{E/F}(-1)=1 \quad\text{if and only if}\quad {\rm Hom}_{\G(F)}(\itPi,\C)\neq 0.
\]
\end{itemize}
\end{thm}

\begin{proof}
If $F$ is a non-archimedean local field, then the theorem follows from the results in the literature.
Indeed, the first two assertions are proved in 
\cite{Prasad1990}, \cite{Prasad1992} and \cite{WTG2008}. 
On the other hand, the third assertion is proved in \cite{Prasad1990}, \cite{Prasad1992} unless
$\itPi$ is a supercuspidal representation of $\G(E)$, where $E=F\x F\x F$ or $E=F'\x F$.
The validity of the third assertion for the remaining cases can be argued as follows. 
Gan in \cite{WTG2008} proves the third assertion without 
any assumption, but with $\epsilon(1/2,\itPi,r,\psi)$ replaced by $\epsilon_{{\rm PSR}}(1/2,\itPi,r,\psi)$, where
$\epsilon_{{\rm PSR}}(s,\itPi,r,\psi)$ is the $\epsilon$-factor defined by the local zeta integrals 
(\cite{PSR1987}, \cite{Ikeda1992}). On the other hand, by results in \cite{Rama2000} and \cite{CCI2020}, we have 
 \[
 \epsilon(s,\itPi,r,\psi)=\epsilon_{{\rm PSR}}(s,\itPi,r,\psi)
\]
if $\itPi$ is a local component of an irreducible cuspidal automorphic representation. Since it is known that every 
supercuspidal representation of $\G(E)$ can be embedded as a local component of an irreducible cuspidal 
automorphic representation (\cite[Proposition 5.1]{Shahidi1990}), the third assertion for the remaining cases is proved.

Suppose that $F=\R$. If $E=\R\x\R\x\R$, then the theorem 
follows from combining the results of \cite[Proposition 8.4, Theorem 9.3]{Prasad1990}
\footnote{The proof of Proposition 8.4 works for the archimedean case.} and
\cite[Theorem 1.2]{Loke2001}. It remains to prove the theorem when $E=\C\x\R$. 
Note that the second assertion can be argued as in the beginning of \cite[Section 8]{Prasad1992}. We now prove 
the first assertion. Write $\itPi=\pi'\bt\pi$, where $\pi'$ (resp. $\pi$) is an irreducible admissible generic representation
of $\G(\C)$ (resp. $\G(\R)$) with the minimal weight $k'$ (resp. $k$). 
We note that the assumption
on the central character implies $k'\equiv k\pmod 2$. The first assertion follows from
\cite[Theorem 1.3]{Loke2001} except that we need to show that if $\pi$ is an discrete series 
representation and $k>k'$, then ${\rm dim}_\C{\rm Hom}_{D^{\x}}(\itPi^D,\C)=1$. In this case,
we have $\itPi^D=\pi'\bt\pi^D$ with $\pi^D|_{{\rm SU}(2)}\cong{\rm Sym}^{k-2}(\C^{\oplus 2})$.
By \cite[Lemma 6.1 (ii)]{JLbook}, we have
\[
\pi'|_{{\rm SU}(2)}\cong\bigoplus_{n\geq k',\, n\equiv k'\pmod 2}{\rm Sym}^n(\C^{\oplus 2}).
\]
Since $D^{\x}\cong\R^{\x}_+\x{\rm SU}(2)$ and the central character of $\pi'\bt\pi^D$ is trivial on $\R^{\x}$, 
we find that 
\[
{\rm Hom}_{D^{\x}}(\pi'\bt\pi^D,\C)
=
{\rm Hom}_{{\rm SU}(2)}(\pi'\bt\pi^D,\C),
\]
which by the observation just mentioned is isomorphic to
\[
\prod_{n\geq k',\,n\equiv k'\pmod 2}
{\rm Hom}_{{\rm SU}(2)}({\rm Sym}^n(\C^{\oplus 2})\bt{\rm Sym}^{k-2}(\C^{\oplus 2}),\C).
\]
It's clear that the above space is one-dimensional. This proves the first assertion for the case $E=\C\x\R$.

The proof of the last assertion for the case $E=\C\x\R$ is similar to that of \cite[Proposition 8.2]{Prasad1992} 
modulo a little work. By \cite[Theorem 1.3]{Loke2001}, we know that $\pi'\bt\pi$ admits a 
non-trivial $\G(\R)$-invariant form if and only if $\pi$ is a principal series representation or $\pi$ is a 
discrete series representation with $k\leq k'$. We show that this is equivalent to that the character 
$(\mu\nu^\sigma)^{-1}$ of $\C^{\x}$ appears in $\pi$, where $\mu, \nu$ are characters of $\C^{\x}$ so that
$\pi'\cong\pi(\mu,\nu)$, and $\nu^\sigma(z):=\nu(\b{z})$.  Then the assertion follows 
from the proof of \cite[Proposition 8.2]{Prasad1992} together with the results of 
\cite{Tunnell1983}, \cite[Theorem 9.2]{WTGnote}. Note that we have 
$\mu\nu^\sigma(e^{i\theta})=e^{\pm ik'\theta}$. We embed $\C^{\x}$ into $\G(\R)$  by sending
$z=re^{i\theta}\in\C^{\x}$ to $\iota(z)=rk(\theta)\in\G(\R)$. Suppose that $\pi$ is the irreducible admissible 
generic subrepresentation of $\rho(\chi,\eta)$.  We have $\chi\eta^{-1}(-1)=(-1)^k$
and
(\cite[Theorem 5.11]{JLbook})
\[
\pi\cong\bigoplus_{|n|\geq k,\,n\equiv k\pmod 2}\C\cdot f_n
\]
where $f_n$ is the function on $\G(\R)$ characterized by
\[
f_n\left(\pMX{a}{b}{0}{d}k(\theta)\right)=\chi(a)\eta(d)\left|\frac{a}{d}\right|_\R^{\frac{1}{2}}e^{in\theta}.
\]
It's now clear that $(\mu\nu^\sigma)^{-1}$ appears in $\pi$. Indeed, if
$\mu\nu^\sigma(e^{i\theta})=e^{ik'\theta}$, then we have $f_{-k'}\in\pi$ and 
\[
\pi(\iota(z)) f_{-k'}=\chi\eta(r)e^{-ik'\theta} f_{-k'}=(\mu\nu^\sigma)^{-1}(r)e^{-ik'\theta} f_{-k'}
=(\mu\nu^\sigma)^{-1}(z) f_{-k'}.
\]
Similar arguments show that if $\mu\nu^\sigma(e^{i\theta})=e^{-ik'\theta}$, then $(\mu\nu^\sigma)^{-1}$ 
appears in $\pi$. This finishes the proof.
\end{proof}

\begin{Remark}
By \cite[Theorem 1.2]{Loke2001}, \thmref{T:root number} holds when $F=\C$ under a mild assumption.
\end{Remark}

\subsection{Global settings and assumptions} 
In principal, one can use \propref{P:regularization} and the results in \cite[Section 18]{JLbook2} to compute local 
period integrals for the complex place. By combining the results in the literature mentioned in the 
introduction, one can obtain special value formulae over arbitrary number fields under quite general settings. 
In this article; however, we let $F=\Q$ and make the following assumptions for simplicity.

As before, we use $v$ to indicate an arbitrary place of $\Q$ 
and $p$ (resp. $\infty$) to denote the finite place (resp. the real place) of $\Q$. 
We call the case $E_\infty\cong\R\x\R\x\R$ the real case,
while the case $E_\infty\cong\C\x\R$ the complex case.
Let $\itPi=\ot'_v\itPi_v$ be a unitary irreducible cuspidal automorphic representation of 
$\G(\A_E)$ with trivial central character 
\footnote{
In the computation of the local period 
integrals for the real place, we only assume the central character to be trivial on 
$\R^{\x}$. For more details, see \S\ref{SS:Calculations of local period integrals: archimedean case}}.  
We write $\itPi_v=\pi'_v\bt\pi_v$ when $c_{E_v}=2$ and 
$\itPi_v=\pi_{1,v}\bt\pi_{2,v}\bt\pi_{3,v}$ when $c_{E_v}=3$, where $\pi_v$ and $\pi_{j,v}$ for 
$j=1,2,3$ are unitary irreducible admissible generic representations of $\G(\Q_v)$, and $\pi'_v$
is a unitary irreducible admissible generic representation of $\G(F'_v)$, where $F'_v$ is the 
quadratic field extension of $\Q_v$ such that $E_v=F'_v\x\Q_v$ when $c_{E_v}=2$.   
Let $\psi=\prod_v\psi_v:\Q\backslash\A\to\C^{\x}$ be a non-trivial additive character.
 
\subsubsection{Assumptions}\label{SSS:global assumption}
We make the following assumptions on $\itPi$ throughout this section.
\begin{itemize}
\item
We assume $\pi_\infty$ and one of $\pi_{1,\infty}, \pi_{2,\infty}, \pi_{3,\infty}$ to be a principal 
series representation of $\G(\R)$.
\item
We assume the conductor $\frak{c}_{\itPi_p}$ of $\itPi_p$ to be $square$ $free$ for 
all $p$.
\item
We assume that the global root number $\epsilon(1/2,\itPi,r)=\prod_{v}\epsilon(1/2,\itPi_v,r,\psi_v)=1$,
where $v$ runs over all places of $\Q$. 
\end{itemize}

We explain in more detail for the second assumption. By a result of Casselman (\cite{Casselman1973}),
the conductor $\frak{c}_{\itPi_p}$ of $\itPi_p$ is the unique ideal of $\cO_{E_p}$ such that
${\rm dim}\,\itPi^{K_0(\mathfrak{c}_{\itPi_p})}=1$.
We say that $\frak{c}_{\itPi_p}$ is $square$ $free$ if it satisfies one of the following conditions:
\begin{itemize}
\item[(i)] $c_{E_p}=1$ and $\frak{c}_{\itPi_p}=\varpi^a_{E_p}\cO_{E_p}$ with $a\leq 1$,
\item[(ii)] $c_{E_p}=2$ and $\frak{c}_{\itPi_p}=(\varpi^a_{F_p'},p^b)\cO_{E_p}$ with $a,b\leq 1$,
\item[(iii)]  $c_{E_p}=3$ and $\frak{c}_{\itPi_p}=(p^a,p^b,p^c)\cO_{E_p}$ with $a,b,c\leq 1$.
\end{itemize}


\subsubsection{The quaternion algebra}\label{SSS:quaternion algebra}
By \thmref{T:root number} $(3)$ and the assumptions, there is a unique $indefinite$ quaternion 
$\Q$-algebra $D$ such that 
\[
p\in\Sigma_D\Longleftrightarrow\epsilon(1/2,\itPi_p,r,\psi_p)\omega_{E_p/\Q_p}(-1)=-1.
\]
Here $\Sigma_D$ is the ramification set of $D$.
We fix a such $D$ and the various isomorphisms and embeddings as in \cite[\S 5.2, \S 6.3]{CC2019} in the rest 
of this section. We then identify $D_v$ as a subalgebra of $D(E_v)$ via these embeddings. 

\subsection{Automorphic forms and raising elements} 
In this subsection, we introduce distinguished automorphic forms and raising elements appearing in
the global period integrals of the special value formulae. 
By our assumption,
the global Jacquet-Langlands lift $\itPi^D=\ot'_v\itPi^D_v$ of $\itPi$ 
to $D^{\x}(\A_E)$ is non-zero. Note that the central character of $\itPi^D$ is also trivial.

\subsubsection{An Eichler order}
We define an Eichler order $R_{\itPi^D}\subset\cO_{D(E)}$ by requiring it to satisfy the following
\[
R_{\itPi^D}\ot_{\cO_E}\cO_{E_p}\cong
\begin{cases}
{\rm M}_2(\frak{c}_{\pi'_p})\x\cO_{D_p}
&\quad\text{if $p\in\Sigma_{D}$ and $c_{E_p}=2$},\\
\cO_{D(E_p)}
&\quad\text{if $p\in\Sigma_{D}$ but $c_{E_p}\neq 2$},\\
{\rm M}_2(\frak{c}_{\itPi_p})
&\quad\text{if $p\notin\Sigma_{D}$}.
\end{cases}
\]
Here $\cO_{D(E_p)}$ is the maximal order of $D(E_p)$ and $\frak{c}_{\pi'_p}$ is the conductor of 
$\pi'_p$. Observe that when $p\in\Sigma_D$, one always has 
$D\ot_{\Q_p}F'_p\cong {\rm M}_2(F'_p)$, where $F'_p$ is any quadratic field extension over $\Q_p$. 
This justifies our choice of the order in the first case.

\subsubsection{Distinguished automorphic form: the real case} 
\label{SSS:distinguished automorphic form: the real case}
We fix a unique (up to constants) automorphic form $\bff^D\in\itPi^D$ as follows.
Note that $\itPi^D_\infty=\itPi_\infty=\pi_{1,\infty}\bt\pi_{2,\infty}\bt\pi_{3,\infty}$.
Let $k_j\geq 0$ be the minimal weight of $\pi_{j,\infty}$ for $j=1,2,3$. 
Then $\bff^D: D^{\x}(\A_E)\to\C$ is the automorphic form in $\itPi^D$ characterized by
\begin{align*}
\bff^D(z\gamma h u_\infty u)
=
\bff^D(h)e^{i(k_1\theta_1+k_2\theta_2+k_3\theta_3)}
\end{align*}
for $z\in\A^{\x}_E$, $\gamma\in D^{\x}(E)$, $u\in\wh{R}^{\x}_{\itPi^D}$ and 
$u_\infty=(k(\theta_1),k(\theta_2),k(\theta_3))\in{\rm SO}(2)\x{\rm SO}(2)\x{\rm SO}(2)$.

\subsubsection{Distinguished automorphic form: the complex case} 
\label{SSS:distinguished automorphic form: the complex case}
We fix a unique (up to constants) automorphic form $\bff^D\in\itPi^D$ as follows. 
Observe that $\itPi^D_\infty=\itPi_\infty=\pi'_\infty\bt\pi_\infty$. 
Let $k'$ (resp. $k$) be the minimal weight of $\pi'_\infty$ (resp. $\pi_\infty$). 
Suppose that $\pi'_\infty\cong\pi(\mu'_\infty,\nu'_\infty)$ and
$\pi_\infty\cong\pi(\mu_\infty,\nu_\infty)$ and let 
$\epsilon\in\stt{0,1}$ so that $\mu'_\infty\mu_\infty(-1)=(-1)^{k'+\epsilon}$. 
We have two cases (I) $k'>0$ or $\epsilon=0$ and (II) $k'=0$ and $\epsilon=1$, which 
are disjoint by Remark \ref{R:well-defined of m}. Put
\begin{equation}\label{E:l for CxR}
m_{\itPi_\infty}=m=
\frac{k'+k}{2}+\epsilon\,\frac{(1+(-1)^k)}{2}\quad\text{if $k'>0$ or $\epsilon=0$}.
\end{equation}
Let $\bff^D: D^{\x}(\A_E)\to\C$ be the automorphic form in $\itPi^D$ characterized by
\begin{itemize}
\item the ${\rm SU}(2)$ span of $\bff^D$ is isomorphic to ${\rm Sym}^{k'}(\C^{\oplus 2})$
(resp. ${\rm Sym}^{2}(\C^{\oplus 2})$) in the case (I) (resp. case (II)),
\item 
\[
\bff^D(z\gamma h u_\infty u)
=
\begin{cases}
\bff^D(h)e^{i((2m-k')\theta'+k\theta)}\quad&\text{case (I)},\\ 
\bff^D(h)\quad&\text{case (II)}.
\end{cases}
\]
\end{itemize}
Here $z\in\A^{\x}_E$, $\gamma\in D^{\x}(E)$, $u\in\wh{R}^{\x}_{\itPi^D}$ and 
$u_\infty=(k(\theta'),k(\theta))\in{\rm SU}(2)\x{\rm SO}(2)$. In other words, $\bff^D$ corresponds  
to the vector $v_{k',m}\in\cL_{k'}(\C)$ (resp. $v_{2,1}\in\cL_2(\C)$) in its complex place in case (I) 
(resp. case (II)).

\subsubsection{Raising element}\label{SSS:raising element}
Let
\begin{align}\label{E:raising element}
\begin{split}
\t{V}_+=\left(-\frac{1}{8\pi}\right)\cdot V_+
\quad\text{with}\quad
V_+
=
\begin{pmatrix} 1&0\\0&-1 \end{pmatrix}\otimes 1+
\begin{pmatrix} 0&1\\1&0 \end{pmatrix}\otimes \sqrt{-1}\in\frak{U}_\R
\end{split}
\end{align}
be the (normalized) weight raising element as in \cite[Lemma 5.6]{JLbook}. We use $\frak{U}_F$ to denote  
the universal enveloping algebra of ${\rm Lie}(\G(F))\ot_\R\C$ where $F=\R$ or $\C$. In any case,
we write ${\rm Id}$ for the identity element in $\frak{U}_F$.
We define the raising element $\bft=(\bft_v)_v$ as follows.

\begin{itemize}
\item
Suppose that we are in the real case. Let $k_1,k_2,k_3$ be as in 
\S\ref{SSS:distinguished automorphic form: the real case}, which are all even integers by our assumption on the central 
character. Re-indexing if necessary we may assume that $k_1={\rm max}\stt{k_1,k_2,k_3}$ and $\pi_{3,\infty}$
is a principal series representation. Suppose that $\pi_{1,\infty}$ is a discrete series representation.
We put $\ell=\frac{k_1-k_2-k_3}{2}\geq 0$. 
\[
\bft_\infty
=
({\rm Id}\ot{\rm Id}\ot\t{V}_+^{\ell},(J_2,I_2,I_2)).
\]
Suppose that $\pi_{1,\infty}, \pi_{2,\infty}$ are principal series representations, so that $\pi_{3,\infty}$ is also a principal 
series representation, and we have $k_1=k_2=k_3=0$. Let $\pi_{j,\infty}\cong\pi(\mu_{j,\infty},\nu_{j,\infty})$ for some 
characters $\mu_{j,\infty},\nu_{j,\infty}$ of $\R^\x$ for $j=1,2,3$. Set
\footnote{In the case where $k_1=k_2=k_3=0$ and $\mu_{1,\infty}\mu_{2,\infty}\mu_{3,\infty}(-1)=-1$, the local 
period at the real place and hence the global period integral vanishes if we do not use the raising 
elements, see \propref{P:local period integral for RxRxR}.
\[
\bft_\infty
=
\begin{cases}
({\rm Id}\ot{\rm Id}\ot{\rm Id}),(I_2,I_2,I_2))
\quad&\text{if $\mu_{1,\infty}\mu_{2,\infty}\mu_{3,\infty}(-1)=1$},\\ 
({\rm Id}\ot\t{V}_+\ot\t{V}_+),(I_2,I_2,J_2))
\quad&\text{if $\mu_{1,\infty}\mu_{2,\infty}\mu_{3,\infty}(-1)=-1$}.
\end{cases}
\]
}
\item
Suppose that we are in the complex case. 
Let $\pi'_\infty$, $k'$ and $\pi_\infty$, $k$ and $\epsilon$ be as in 
\S\ref{SSS:distinguished automorphic form: the complex case}. Define
\[
\bft_\infty
=
\begin{cases}
({\rm Id}\ot{\rm Id},(I_2,I_2))
\quad&\text{if $k'\geq 0$ is even and $\epsilon=0$},\\
({\rm Id}\ot\t{V}_+,(I_2,J_2))
\quad&\text{if $k'\geq 2$ is even and $\epsilon=1$},\\
({\rm Id}\ot{\rm Id},(I_2,I_2))
\quad&\text{if $k'=0$ and $\epsilon=1$}.
\end{cases}
\]

\item
Suppose that $v=p$ is a finite place. 
Let $\bft_p\in D^{\x}(E_p)$ be as in \cite[\S 2.3]{CC2019}.
\end{itemize}
\begin{Remark}\label{R:well-defined of m}
In the definitions of the integer $m$ in \eqref{E:l for CxR} and the raising elements $\bft_\infty$, we have used the signs
of the characters appearing in the various principal series representations. However, since 
$\pi(\mu,\nu)\cong\pi(\nu,\mu)$, we need to check that these definitions depend only on the isomorphism class of 
$\itPi_\infty$, i.e. independent of the order of the characters appearing in the various principal series representations.
These can be check easily from the fact that
\[
\mu\nu^{-1}(-1)=(-1)^k
\] 
where $k$ is the minimal weight of $\pi(\mu,\nu)$. For example, we check that $m$ is well-defined. For this we note that 
$k'\equiv k\pmod 2$ and 
\[
\mu'_\infty\mu_\infty(-1)
=
\nu'_\infty\nu_\infty(-1)
=
(-1)^{k'+\epsilon}
\quad\text{and}\quad
\mu'_\infty\nu_\infty(-1)
=
\nu'_\infty\mu_\infty(-1)
=
(-1)^{\epsilon}.
\]
In particular, $\epsilon$ and hence $m$ is well-defined if  $k'$ is even. On the other hand, if $k'$ is odd, then $\epsilon$
is not well-defined; however, we have $m=(k'+k)/2$, which is independent of $\epsilon$. 
\end{Remark}

\subsection{The formula}
Define $c_p=0$ if $\frak{c}_{\itPi_p}=\cO_{E_p}$, and $c_p=1$ if 
$\frak{c}_{\itPi_p}\subsetneq\cO_{E_p}$. Note that $c_p=0$ for almost all $p$.
Let $N:=\prod_{p\notin\Sigma_D} p^{c_p}$ and $R_N\subset\cO_D$ be the Eichler order of level $N$, and put
\begin{itemize}
\item 
$C(\itPi)
=
\prod_p
\left[\Z_p:p^{c_p}\Z_p\right]
\x\left[\cO_{E_p}:\frak{c}_{\itPi_p}\right]^{-1}$.
\end{itemize}
Suppose for the moment that $F$ is a number field and $B$ is a quaternion $F$-algebra. Let 
$R=F_1\x\cdots\x F_r$, where $F_j$ are finite field extensions over $F$. 
Put $B(R)=B\ot_F R=B(F_1)\x\cdots\x B(F_r)$. Then $B(F_j)=B\ot_F F_j$ is a quaternion $F_j$-algebra.
Define
\begin{itemize}
\item $C_B=\prod_{w\in\Sigma_B,\,w<\infty}(q_{F_w}-1)$,
\item $C_{B(R)}=\prod_{1\leq j\leq r}C_{B(F_j)}$.
\end{itemize}
We understand that $C_{{\rm M}_2}=1$.
Let $\nu(\itPi)$ be the number of finite places $p$ such that
\begin{itemize}
\item
$c_{E_p}=1$ and $\frak{c}_{\itPi_p}=\varpi_{E_P}\cO_{E_P}$, 
\item
$c_{E_p}=2$,  $F'_p/\Q_p$ is unramified and $\frak{c}_{\itPi_p}=(p,p)\cO_{E_p}$, 
\item
$c_{E_p}=2$, $F'_p/\Q_p$ is ramified and $\frak{c}_{\itPi_p}=(1,p)\cO_{E_p}$, 
\item
$c_{E_p}=3$ and $\frak{c}_{\itPi_p}=(p,p,p)\cO_{E_p}$.
\end{itemize}

\begin{thm}\label{T: special value formula}
{For notations as above, we have}
\begin{align*}
\frac{
\left|
\int_{\A^{\x}D^{\x}(\Q)\backslash D^{\x}(\A)}\itPi^D(\bft) \bff^D(h) dh^{{\rm Tam}}
\right|^2}
{\int_{\A_E^{\x} D^{\x}(E)\backslash D^{\x}(\A_E)}|\bff^D(h)|^2 dh^{{\rm Tam}}}
&=
2^{-c_E+\nu(\itPi)}C(\itPi)
\cdot
\frac{C_{D(E)}\left[\wh{\cO}_{D(E)}^{\x}:\wh{R}_{\itPi^D}^{\x}\right]}
{C^2_D\left[\wh{\cO}_D^{\x}:\wh{R}_{N}^{\x}\right]^2}
\cdot
I^*(\itPi_\infty,\bft_\infty)\\
&\x
\frac{\xi_E(2)}{\xi_\Q(2)^2}
\cdot
\frac{\Lambda(\frac{1}{2},\itPi,r)}{\Lambda(1,\itPi,{\rm Ad})}.
\end{align*}
Here $dh^{\rm Tam}$ are the Tamagawa measures on $ PD^{\x}$, 
and $I^*(\itPi_\infty,\bft_\infty)$ are defined in \S\ref{SS:Calculations of local period integrals: archimedean case} 
and computed in \propref{P:local period integral for RxRxR} and \propref{P:local period integral for CxR}.
\end{thm}

\begin{proof}
For each place $v$ of $\Q$, let $dh_v$ be the Haar measure on $F_v^{\x}\backslash D_v^{\x}$ given by
\cite[\S 4.1 and \S 5.1]{CC2019}.
By \cite[Lemma 6.1]{IchinoPrasanna}, we have $dh^{{\rm Tam}}=C^{-1}_D\xi_\Q(2)^{-1}\prod_vdh_v$.
By \thmref{T:Ichino formula} and Remark \ref{R:remark on Ichino's formula}, we find that
\begin{align*}
\frac{
\left|
\int_{\A^{\x}D^{\x}(\Q)\backslash D^{\x}(\A)}\itPi^D(\bft) f^D(h) dh^{{\rm Tam}}\right|^2}
{\int_{\A_E^{\x} D^{\x}(E)\backslash D^{\x}(\A_E)}|\bff^D(h)|^2 dh^{{\rm Tam}}}
=
2^{-c_E}C_D^{-1}
\cdot
\frac{\xi_E(2)}{\xi_\Q^2(2)}
\cdot
\frac{\Lambda(\frac{1}{2},\itPi,r)}{\Lambda(1,\itPi,{\rm Ad})}
\cdot
\prod_v I^*(\itPi^D_v,\bft_v)
\end{align*}
where $v$ runs over all places of $\Q$, and $I^*(\itPi^D_v,\bft_v)$ are essentially the local period integrals 
appearing in Ichino's formula whose definition can be found in 
\S\ref{SS:Calculations of local period integrals: archimedean case} and \cite[Section 2]{CC2019}. 
By 
\lmref{L:bi=her}, the definition of $I^*(\itPi^D_p,\bft_p)$ agrees
with the one given in \cite[Section 2]{CC2019}, so that we can apply the results in 
\cite[Sections 4, 5]{CC2019} to obtain
\[
\prod_p I^*(\itPi_p,\bft_p)
=
2^{\nu(\itPi)}C(\itPi)C^{-1}_D C_{D(E)}
\left[\wh{\cO}_{D(E)}^{\x}:\wh{R}_{\itPi^D}^{\x}\right]
\left[\wh{\cO}_D^{\x}:\wh{R}_{N}^{\x}\right]^{-2}
\]
where $p$ runs over all finite places. Finally, we note that $\itPi^D_\infty=\itPi_\infty$ since $D$ is indefinite by our 
assumption. This proves the theorem.
\end{proof}

The following corollary is a direct consequence of \thmref{T: special value formula}.

\begin{cor}\label{C:non-negativity of the central value}
The central values $\Lambda\left(1/2,\itPi,r\right)$ are non-negative real numbers.
\end{cor}

\section{Local calculation}\label{S:local computation}
The purpose of this section is to compute the local period integrals 
$I^*(\itPi_v,\bft_v)$ appeared in the proof of \thmref{T: special value formula}. 
Since our computations work for every local field, we drop
$v$ from our notations in this section and assume $F$ to be a local field.
We begin with some preparations.

\subsection{Explicit formulae of certain Whittaker functions}
In this subsection, we give explicit formulae of certain Whittaker functions which relates to the test vectors appearing  
in the local period integrals when $F=\R$ or $\C$. The proof of these formulae occurs in the appendix.
Some of the formulae were also obtained in \cite{Popa2008} and 
\cite{Woodbury2016} via different methods.
Let $\psi$ be the additive character of $F$ given by 
\begin{equation}\label{E:additive character of archimedean}
\psi(x)=
\begin{cases}
e^{2\pi i x}\quad&\text{if $F=\R$},\\
e^{2\pi i(x+\b{x})}\quad&\text{if $F=\C$}.
\end{cases}
\end{equation}
Let $\pi$ be an irreducible admissible $generic$ representation of $\G(F)$
with the minimal weight $k\geq 0$. 
For each integer $n$, let $\cV_\psi(\pi,n)$ denote the following spaces
\begin{equation}\label{E:K-invariant space}
\left(\cW(\pi,\psi)\ot\chi_{-n}\right)^{{\rm SO}(2)}
\quad\text{or}\quad
\left(\cW(\pi,\psi)\ot\rho_{n}\right)^{{\rm SU}(2)}
\end{equation}
depending on $F=\R$ or $F=\C$. Note that $n\geq 0$ when $F=\C$, and we have 
${\rm dim}_\C\cV_\psi(\pi,n)\leq 1$, which is equal to $1$ if and only if $|n|\geq k$ with $n\equiv k\pmod 2$.
Suppose that $0\neq W\in\cV_\psi(\pi,n)$. Then it's clear that $W$ is a weight $n$ element when $F=\R$.
On the other hand, if $F=\C$, then $W$ is a $\cL_n(\C)$-$valued$ function on $\G(\C)$ satisfying
\begin{equation}\label{E:SU(2)-equivalent}
W(gu)=\rho_n(u^{-1})W(g)
\end{equation} 
for every $g\in\G(\C)$ and $u\in{\rm SU}(2)$. Following \cite[Section 18]{JLbook2}, we refer to such $W$ as the 
Whittaker function of type $\rho_n$ attached to $\pi$. We would like to describe $W(g)$ for some $n$.
By the above observation and the Iwasawa decomposition, it suffices to consider $g=\pMX y001$
for $y\in F^{\x}$. We may assume that $\pi$ is isomorphic to a constituent of $\rho(\mu,\nu)$ for some $\mu,\nu$.
When $F=\C$, we may further assume that $\rho(\mu,\nu)$ is irreducible so that $\pi\cong\pi(\mu,\nu)\cong\pi(\nu,\mu)$ 
is a principal series representation and also that $\mu\nu^{-1}(e^{i\theta})=e^{ik\theta}$ by the uniqueness of the 
Whittaker model. Recall that $V_+\in\frak{U}_\R$ is the weight raising element given by \eqref{E:raising element}.

\begin{prop}\label{P:test vector in Whittaker model}
Let $s\in\C$ such that $\mu\nu^{-1}(y)=|y|_F^{2s}\left(y/|y|_F^{\frac{1}{d}}\right)^k$ where $d:=[F:\R]$.
\begin{itemize}
\item[(1)]
Suppose that $F=\R$ and $\pi$ is a (limit of) discrete series representation so that 
$\mu\nu^{-1}=|\cdot|^{\pm (k-1)}{\rm sgn}^k$. Let $\chi=\mu |\cdot|^{\pm\frac{1-k}{2}}$.
Then $\cV_\psi(\pi,k)$ cantains a non-zero generator $W_\pi$ which is  given by
\[
W_\pi\left(\pMX{y}{0}{0}{1}\right)
=
\chi(y) y^{\frac{k}{2}}e^{-2\pi y}\cdot\bbI_{\R^{\x}_+}(y).
\]
More generally, if $\ell\geq 0$ is an integer, then
$\rho(V^\ell_+)W_\pi$ is a non-zero generator of $\cV_\psi(\pi,k+2\ell)$ and is given by 
\[
\rho(V^\ell_+)W_\pi\left(\pMX y001\right)
=
\chi(y) 2^\ell\, P^{(\ell)}_{\pi}(y) e^{-2\pi y}\cdot \bbI_{\R^\x_+}(y)
\]
with
\[
P^{(\ell)}_{\pi}(y)
=
\sum_{j=0}^\ell (-4\pi)^j\begin{pmatrix} \ell\\j \end{pmatrix}
\frac{\Gamma(k+\ell)}{\Gamma(k+j)}\, y^{\frac{k}{2}+j}.
\]
\item[(2)]
Suppose that $F=\R$ and $\pi\cong\pi(\mu,\nu)$ is a principal series representation. 
Then $\cV_\psi(\pi,k)$ cantains a non-zero generator $W_\pi$ which is  given by
\[
W_{\pi}\left(\pMX y001\right)
=
\sum_{j=0}^k
\begin{pmatrix}
k\\ j
\end{pmatrix}
\mu(y)y^{k-j}|y|^{\frac{1-k}{2}-s+j}
K_{s+\frac{k}{2}-j}\left(2\pi |y|\right)
\]
If $k=0$, then $\rho(V_+)W_\pi$ is a non-zero generator of $\cV_\psi(\pi,2)$ and is given by
\begin{align*}
\begin{split}
\rho(V_+)W_\pi\left(\pMX y001\right)
&=
-2\pi\,\mu(y)|y|^{\frac{3}{2}-s}K_{s+1}(2\pi|y|)
-
4\pi\,\mu(y)|y|^{\frac{3}{2}-s}{\rm sgn}(y)K_s(2\pi|y|)\\
&+
\mu(y)|y|^{\frac{1}{2}-s}K_s(2\pi|y|)
-
2\pi\,\mu(y)|y|^{\frac{3}{2}-s}K_{s-1}(2\pi|y|).
\end{split}
\end{align*}
\item[(3)]
Suppose that $F=\C$. 
Then $\cV_\psi(\pi,k)$ cantains a non-zero generator $\vec{W}_\pi$ which is  given by
\[
\vec{W}_\pi(g)
=
\sum_{j=0}^{k}
\begin{pmatrix}
k\\j 
\end{pmatrix}
W_j(g) X^{j} Y^{k-j}
\]
with
\[
W_j\left(\pMX y001\right)
=
\left(\sqrt{-1}\right)^j
\mu(y)\b{y}^{j}
|y|_\C^{-\left(s-\frac{k}{4}+\frac{j}{2}-\frac{1}{2}\right)}
K_{2s-\frac{k}{2}+j}\left(4\pi |y|_\C^\frac{1}{2}\right)
\]
If $k=0$, then $\cV_\psi(\pi,2)$ cantains a non-zero generator $\vec{W}^{(2)}_\pi$ which is  given by
\[
\vec{W}^{(2)}_\pi(g)
=
W_0^{(2)}(g)X^2
+
W_1^{(2)}(g)XY
+
W_2^{(2)}(g)Y^2
\]
with
\begin{align*}
W_0^{(2)}\left(\pMX y001\right)
&=
(-\sqrt{-1})\mu(y)\b{y}|y|_\C^{\frac{1}{2}-s}K_{2s}\left(4\pi|y|_\C^{\frac{1}{2}}\right)\\
W_2^{(2)}\left(\pMX y001\right)
&=
(\sqrt{-1})\mu(y)y|y|_\C^{\frac{1}{2}-s}K_{2s}\left(4\pi|y|_\C^{\frac{1}{2}}\right)
\end{align*}
and 
\begin{align*}
W_1^{(2)}\left(\pMX y001\right)
&=
-\mu(y)|y|_\C^{1-s}K_{2s+1}\left(4\pi|y|_\C^{\frac{1}{2}}\right)
-
\mu(y)|y|_\C^{1-s}K_{2s-1}\left(4\pi|y|_\C^{\frac{1}{2}}\right)\\
&+
(2\pi)^{-1}\mu(y)|y|_\C^{\frac{1}{2}-s}K_{2s}\left(4\pi|y|_\C^{\frac{1}{2}}\right)
\end{align*}
\end{itemize}
\end{prop}

\begin{proof}
The first case is proved in \cite[Lemma 3.3]{CC2019}, while the others are proved in the appendix.
\end{proof}

\begin{Remark}\noindent
\begin{itemize}
\item[(1)]
By the isomorphism $W\mapsto W'$ defined in \S\ref{SS:Whittaker model}, 
there is no loss of generality for the assumptions on $\psi$.
\item[(2)]
Note that when $F=\R$, we have $W\in\cV_\psi(\pi,n)$ if and only if $\rho(J_2)W\in\cV_\psi(\pi,-n)$. In particular, 
\propref{P:test vector in Whittaker model} $(1)$ provides complete descriptions of elements in $\cW(\pi,\psi)$ when 
$\pi$ is a (limit of) discrete series representation. 
\end{itemize}
\end{Remark}

\subsection{Calculation of the norms}
\label{SS:Calculation of the norms}
In this subsection, we let $F=\R$ or $\C$. The aim is to compute the $L^2$-norm of various functions on 
$\G(F)$ appearing in the local period integrals. 
Fix $\psi$ to be the character of $F$ as in \eqref{E:additive character of archimedean}. 
Let $dx$ be the Haar measure on $F$ which is self-dual with respect to $\psi$. 
Let $d^{\x} x$ be the measure on $F^{\x}$ defined by $d^{\x}x=\zeta_F(1)|x|_F^{-1}dx$.
Let $\pi$ be a $unitary$ irreducible admissible generic representation of $\G(F)$ with the minimal weight $k$. 
We assume that $\pi$ is isomorphic to a constitute of $\rho(\mu,\nu)$ with $\mu\nu^{-1}|_{\R_+^{\x}}=|\cdot|^{2s}_F$ 
for some $s\in\C$. 
When $F=\C$, we may further assume that $\rho(\mu,\nu)$ is irreducible so that 
$\pi\cong\pi(\mu,\nu)$ is a principal series representation. 
Note that since $\omega_\pi=\mu\nu$ is unitary, we have
\begin{equation}\label{E:unitary character}
|\mu|^2=|\mu\nu^{-1}\cdot\omega_\pi|=|\cdot|_F^{2{\rm Re}(s)}.
\end{equation}

\begin{Remark}\label{R:unified formula}
Suppose that $\pi\cong\pi(\mu,\nu)$ is a unitary principal series representation of $\G(F)$ with 
$\mu\nu^{-1}|_{\R_+^{\x}}=|\cdot|^{2s}_F$ for some $s\in\C$. It is important to observe that we have either 
${\rm Re}(s)=0$ or $s\in\R$ and $k=0$. This observation helps us to give unified formulae whether $\pi$ is tempered 
or not.
\end{Remark}

\subsubsection{Norm of Whittaker functions}\label{SSS:Whittaker functions} 
Let $\cW(\pi,\psi)$ be the Whittaker model of $\pi$ with respect to $\psi$ and $\cH_\pi$ be the 
$\G(F)$-invariant hermitian pairing on $\cW(\pi,\psi)$ defined by \eqref{E:hermitian for Whittaker} with $\cK$ replaced 
by $F$. We would like to compute the $L^2$-norm of certain $W\in\cW(\pi,\psi)$ appearing in the local period integrals.

Suppose that $F=\C$. We may assume that $\mu\nu^{-1}(e^{i\theta})=e^{ik\theta}$ by the uniqueness of the 
Whittaker model. Let $n\geq k$ be an integer so that $n\equiv k\pmod 2$. If $\vec{W}$ is a non-zero generator of 
$\cV_\psi(\pi,n)$, where $\cV_\psi(\pi,n)$ is the ${\rm SU}(2)$-invariant space given by \eqref{E:K-invariant space}, 
then we define the $L^2$-norm of $\vec{W}$ by
\[
\|\vec{W}\|^2
=
\int_{\C^{\x}}
(\vec{W}\left(\pMX y001\right), \vec{W}\left(\pMX{y}{0}{0}{1}\right))_n d^{\x}y
\]
with $(\cdot,\cdot)_n$ the hermitian pairing on $\cL_n(\C)$ given by
\eqref{E:hermitian invariant pairing for symmetric power}. 
When $F=\R$, the $L^2$-norm of an element $W\in\cW(\pi,\psi)$ is the usual one $\|W\|^2:=\cH_\pi(W,W)$.

\begin{lm}\label{L:norm of Whittaker function for archimedean}
Let notations be as above and \propref{P:test vector in Whittaker model}. 
\begin{itemize}
\item[(1)]
Suppose that $F=\R$. We have
\[
\|W_\pi\|^2
=
\begin{cases}
2^{-2k}\pi^{-k}\Gamma(k)
&\text{if $\pi$ is a discrete series representation},\\
2^{-2}\pi^{-k}
\Gamma\left(s+\frac{k+1}{2}\right)
\Gamma\left(-s+\frac{k+1}{2}\right)
&\text{if $\pi$ is a principal series representation}.
\end{cases}
\]
\item[(2)]
Suppose that $F=\C$. We have
\begin{equation}\label{E:norm 1}
\|\vec{W}_\pi\|^2
=
2^{-k-3}\pi^{-k-2}
\Gamma\left(2s+k/2+1\right)
\Gamma\left(-2s+k/2+1\right).
\end{equation}
If $k=0$, then
\begin{equation}\label{E:norm 2}
\|\vec{W}^{(2)}_\pi\|^2=2^{-6}\pi^{-4}\Gamma(2s+2)\Gamma(-2s+2).
\end{equation}
Let $W\in\cW(\pi,\psi)$ be an element given by $W(h)=(W_\pi(h),v_{k,j})_k$, where 
$v_{k,j}\in\cL_k(\C)$ is the element defined by \eqref{E:another basis for symmetric power}. Then we have
\begin{equation}\label{E:norm 3}
\|W\|^2=\cH_\pi(W,W)
=
2^{-3}\pi^{-k-2}\frac{\Gamma(j+1)\Gamma(k+1-j)}{\Gamma(k+2)}
\Gamma\left(2s+k/2+1\right)
\Gamma\left(-2s+k/2+1\right).
\end{equation}
\end{itemize}
\end{lm}

\begin{proof}
Suppose that $F=\R$. If $\pi$ is a discrete series representation, then
$\|W_\pi\|^2=2^{-2k}\pi^{-k}\Gamma(k)$ follows immediately
from \propref{P:test vector in Whittaker model} $(1)$ and a simple computation. 
Suppose that $\pi$ is a principal series representation. Keeping the Remark \ref{R:unified formula} 
in mind,  we then use \propref{P:test vector in Whittaker model} $(2)$, \lmref{L:integration formula for Bessel function} $(3)$ together with \eqref{E:unitary character} to find that 
\begin{align*}
\|W_\pi\|^2
=
\sum_{j=0}^k\int_{\R^{\x}}
|y|^{k+1}
K_{s-\frac{k}{2}+j}\left(2\pi|y|\right)
K_{\b{s}-\frac{k}{2}+j}\left(2\pi|y|\right) 
d^{\x}y
=
2^{-2}\pi^{-k}
\Gamma\left(s+\frac{k+1}{2}\right)
\Gamma\left(-s+\frac{k+1}{2}\right).
\end{align*}
This shows $(1)$.

Suppose that $F=\C$. The calculation of the $L^2$-norm of a vector-valued Whittaker function can be reduced to 
that of a scalar-valued Whittaker function via the Schur's orthogonal relation.  
Indeed, if $0\neq\vec{W}\in\cV_\psi(\pi,n)$, $0\neq v\in\cL_n(\C)$ and $W(g):=(\vec{W}(g),v)_n$ for 
$g\in\G(\C)$ and if $du$ is the Haar measure on ${\rm SU}(2)$ such that ${\rm Vol}({\rm SU}(2),du)=1$, then we have
\begin{align}\label{E:Schur's orthogonal relation}
\begin{split}
\|W\|^2
&=
\int_{{\rm SU}(2)}\cH_\pi(\rho(u)W,\rho(u)W) du\\
&=
\int_{\C^{\x}}\int_{{\rm SU}(2)}
(\rho_n(u)v,\vec{W}\left(\pMX y001\right))_n
\ol{(\rho_n(u)v,\vec{W}\left(\pMX y001\right))_n}dud^{\x}y\\
&=
\frac{(v,v)_n}{n+1}\|\vec{W}\|^2.
\end{split}
\end{align}
From this we see that \eqref{E:norm 3} follows immediately from \eqref{E:norm 1} and \eqref{E:norm for v_n,j}.
Now we prove \eqref{E:norm 1}. By \eqref{E:Schur's orthogonal relation} with $v=Y^k$ and 
\propref{P:test vector in Whittaker model} $(3)$, we have $\|\vec{W}_\pi\|^2=(k+1)\|W_0\|^2$.
On the other hand, 
\begin{align*}
\|W_0\|^2
&=
\int_{\C^{\x}}
|y|_\C^{\frac{k}{2}+1}K_{2s-\frac{k}{2}}\left(4\pi|y|_\C^{\frac{1}{2}}\right)
K_{2\b{s}-\frac{k}{2}}\left(4\pi|y|_\C^{\frac{1}{2}}\right)d^{\x}y\\
&=
4\int_0^\infty r^{k+1}K_{2s-\frac{k}{2}}(4\pi r)K_{2\b{s}-\frac{k}{2}}(4\pi r) dr\\
&=
2^{-k-3}\pi^{-k-2}\frac{\Gamma(k+1)}{\Gamma(k+2)}
\Gamma\left(2s+\frac{k}{2}+1\right)\Gamma\left(-2s+\frac{k}{2}+1\right)
\end{align*}
by \eqref{E:unitary character} and \lmref{L:integration formula for Bessel function} $(3)$. This shows \eqref{E:norm 1}.
The norm $\|\vec{W}_\pi^{(2)}\|^2$ can be computed in a similar way. 
\end{proof}

\subsubsection{Norm of sections}
Suppose that $F=\R$ and $\pi\cong\pi(\mu,\nu)$ is a principal series representation.
If $n$ is an integer with $n\equiv k\pmod 2$, then we let 
$f^{(n)}_\pi\in\cB(\mu,\nu)$ be the element characterized by $f^{(n)}_\pi(k(\theta))=e^{in\theta}$.
When $n=k$, we simply denote $f_\pi=f_\pi^{(k)}$. By \cite[Lemma 5.6]{JLbook}, we have
\begin{equation}\label{E:raising element on section}
\rho\left(\t{V}_+^\ell\right)f_\pi^{(n)}
=
(-1)^\ell 2^{-2\ell} \pi^{-\ell}
\frac{\Gamma\left(s+\frac{n+1}{2}+\ell\right)}{\Gamma\left(s+\frac{n+1}{2}\right)}
f_\pi^{(n+2\ell)}
\end{equation}
for an integer $\ell\geq 0$.

\begin{lm}\label{L:norms for principal series for R}
Let notations be as above. We have
\[
\|f_\pi\|^2:=\cH_\pi(f_\pi,f_\pi)
=
\begin{cases}
1&\quad\text{if $\pi$ is tempered},\\
\pi^{2s}
\Gamma\left(-s+1/2\right)
\Gamma\left(s+1/2\right)^{-1}&\quad\text{if $\pi$ is complementary}.
\end{cases}
\]
Here $\cH_\pi$ is the hermitian pairing on $\cB(\mu,\nu)$ defined in 
\S\ref{SSS:An identity between invariant forms: hermitian}.
\end{lm}

\begin{proof}
Evidently we have $\|f_\pi\|^2=1$ if $\pi$ is tempered. Suppose that $\pi$ is
complementary. Then $s\in\R$ and $k=0$.
Let $f_\Phi\in\cB(\mu,\nu)$ be the Godement section attached to $\Phi$ defined in 
\subsubsecref{SSS:Godement section and intertwining operator}, where 
$\Phi(x,y)=e^{-\pi(x^2+y^2)}$. Then $f_\Phi$ is ${\rm SO}(2)$-invariant on the right
and one verifies that $f_\Phi(I_2)=\zeta_\R\left(2s+1\right)$.
It follows that $f_\pi=\zeta_\R\left(2s+1\right)^{-1}f_\Phi$.
Since $\wh{\Phi}=\Phi$ and $M_\psi^*(\mu,\nu)f_{\Phi}=\t{f}_{\wh{\Phi}}=\t{f}_\Phi$, we find that 
\[
M_\psi^*(\mu,\nu)f_\pi(I_2)
=
\zeta_\R\left(2s+1\right)^{-1}
\t{f}_\Phi(I_2)
=
\zeta_\R(-2s+1)\zeta_\R(2s+1)^{-1}.
\]
It follows that
\begin{align*}
\|f_\pi\|^2
=
\ol{M_\psi^*(\mu,\nu)f_\pi(I_2)}
=
\pi^{2s}
\Gamma(-s+1/2)
\Gamma(s+1/2)^{-1}.
\end{align*}
This completes the proof.
\end{proof} 

\subsection{Local trilinear and Rankin-Selberg integral}
In this subsection, we describe an identity between the local trilinear integral and a product of two Rankin-Selberg 
integrals. This generalizes the result of Michel and Venkatesh \cite[Lemma 3.4.2]{MV2010},
which reduces the calculations to the Rankin-Selberg integrals.

Let $\cK$ be an \etale quadratic $F$-algebra. When $\cK=F\x F$, we identify $F$ 
with a subfield of $\cK$ via the diagonal embedding. Let $z\mapsto\sigma(z)$ denote the non-trivial 
$F$-automorphism of $\cK$. Then $|z|_{\cK}=|z\cdot\sigma(z)|_F$.
Fix an element $\delta\in\cK^{\x}$ so that $\sigma(\delta)=-\delta$ and put  
$\Delta=\delta^2\in F^{\x}$. Fix a non-trivial additive character $\psi$ of $F$ and let $\psi_\cK$ 
be an additive character of $\cK$ defined by $\psi_\cK(z)=\psi(z+\sigma(z))$.
Put $E=\cK\x F$ and let $\itPi=\pi_\cK\bt\pi$ be a unitary irreducible admissible generic representation of $\G(E)$, 
where $\pi_\cK$ (resp. $\pi$) is a unitary irreducible admissible generic representation of $\G(\cK)$
(resp. $\G(F)$). 
We assume that 
\begin{itemize}
\item 
$\omega_\itPi$ is trivial on $F^{\x}$.
\item 
$\pi\cong\rho(\mu,\nu)$ is an irreducible induced representation;
\item $\Lambda(\itPi)<1/2$,
\end{itemize}
where $\Lambda(\itPi)$ is the non-negative real number associated to $\itPi$ defined in \cite[page 285]{Ichino2008}.

\begin{Remark}
Suppose that we merely assume $\pi$ to be a constituent of an induced representation $\rho(\mu,\nu)$. Then there exsit 
counter-examples such that \propref{P:regularization} fail. For example, one can consider the case where $F=\R$, 
$\cK=\R\x\R$ and $\pi_j$ are discrete series representations with the minimal weights $k_j$ with $k_1=2k$ and 
$k_2=k_3=k$  for some $k\geq 2$. In this setting, one can compute both sides of \eqref{E:regularization} by using minimal 
weight elements. The RHS can be computed easily. On the other hand, the LHS has been computed in 
\cite[Section 12]{IchinoIkeda2010} and \cite[Proposition 4.1]{CC2019}. One then checks that \propref{P:regularization} 
fails in this setting.
\end{Remark}

\subsubsection{An identity between invariant forms: bilinear case}
\label{SSS:An identity between invariant forms}
To describe the identity, we realize $\pi_\cK$ in its Whittaker model $\cW(\pi_\cK,\psi_\cK)$ with respect to 
$\psi_\cK$. On the other hand, by our assumption, we realize $\pi$ as $\cB(\mu,\nu)$.
Let $\cB_{\pi_\cK}$ be the $\G(\cK)$-equivariant bilinear pairing between 
$\cW(\pi_{\cK},\psi_{\cK})$ and $\cW(\t{\pi}_{\cK},\psi_{\cK})$ defined by 
\begin{equation}\label{E:bilinear pairing for Whittaker}
\cB_{\pi_\cK}(W,\t{W})
=
\int_{\cK^{\x}}
W\left(\pMX y001\right)\t{W}\left(\pMX{-y}{0}{0}{1}\right)d^{\x}y
\end{equation}
for $W\in\cW(\pi_\cK,\psi_\cK)$ and $\t{W}\in\cW(\t{\pi}_\cK,\psi_\cK)$. 
Let $\cB_{\pi}$ be the $\G(F)$-invariant bilinear pairing between 
$\cB(\mu,\nu)$ and $\cB(\mu^{-1},\nu^{-1})$ given by 
\begin{equation}\label{E:bilinear pairing for induced representation}
\cB_\pi(f,\t{f})
=
\int_K f(k)\t{f}(k)dk,
\end{equation}
for $f\in\cB(\mu,\nu)$ and $\t{f}\in\cB(\mu^{-1},\nu^{-1})$.
Define the local trilinear integral by 
\begin{equation}\label{E:trilinear local integral}
\sI(W\ot f;\t{W}\ot\t{f})
=
\int_{F^{\x}\backslash\G(F)}
\cB_{\pi_\cK}(\rho(g)W,\t{W})\cB_{\pi}(\rho(g)f,\t{f})dg.
\end{equation} 
Note that this is precisely the local integral defined in \S\ref{SSS:Ichino formula}.
On the other hand, we also define the local Rankin-Selberg integrals by
\begin{equation}\label{E:rankin-selberg integral}
\sR_\delta(W\ot f)
=
\int_{F^{\x}\N(F)\backslash\G(F)}
W\left(\pMX \delta001 g\right)
f(g)dg,
\end{equation}
and
\begin{equation}\label{E:rnakin-selberg integral for dual}
\t{\sR}_\delta(\t{W}\ot\t{f})
=
\int_{F^{\x}\N(F)\backslash\G(F)}
\t{W}\left(\begin{pmatrix}-\delta&0\\0&1\end{pmatrix}g\right)
\t{f}(g)dg.
\end{equation}
The Haar measures on various groups are chosen as follows:
\begin{itemize}
\item 
on $F$ (resp. $\cK$), we choose the self-dual measure $dx$ (resp. $dz$) with respect to 
$\psi$ (resp. $\psi_\cK$).
\item
on $F^{\x}$ (resp. $\cK^{\x}$), we take $d^{\x}x=\zeta_F(1)|x|_F^{-1}dx$ 
(resp. $d^{\x}z=\zeta_\cK(1)|z|^{-1}_\cK dz$). 
\item
we identify $\N(F)$ with $F$ so that the measure $dn$ on $\N(F)$ is also defined.
\item
the choices of the measures on $F$ and $F^{\x}$ uniquely determine a 
left Haar measure $db$ on $\B(F)$ given by
\[
db=|y|_F^{-1}d^{\x}zdxd^{\x}y
\quad\text{with}\quad
b=\pMX z00z\pMX 1x01\pMX y001\in\B(F).
\]
\item
we take any Haar measure $dk$ on $K$.
\item
the measure $dg$ on $\G(F)$ is given by $dg=dbdk$. Similarly the measure $dg$ on 
${\rm PGL}_2(F)$ is $dg=dbdk$, but with $db$ the quotient measure on $F^{\x}\backslash\B(F)$
induced from the measures on $F^{\x}$ and $\B(F)$. 
\item
the measure $dg$ on $F^{\x}\N(F)\backslash\G(F)$ is the quotient measure induced from the measures 
on ${\rm PGL}_2(F)$ and $\N(F)$.
\end{itemize}
  
It was shown by Ichino \cite[Lemma 2.1]{Ichino2008} that the integral 
\eqref{E:trilinear local integral} converges absolutely when $\Lambda(\itPi)<1/2$. 
This also holds for the local Rankin-Selberg integrals.

\begin{lm}\label{L:converge for rankin-selberg integral}
If $\Lambda(\itPi)<1/2$, then the integrals $\sR_\delta(W\ot f)$ and 
$\t{\sR}_\delta(\t{W}\ot\t{f})$ are absolutely convergent.
\end{lm}

\begin{proof}
By the Iwasawa decomposition, 
\begin{align}\label{E:decomposition of rankin-selberg integral}
\begin{split}
\sR_\delta(W\ot f)
&=
\int_{K}\int_{F^{\x}}W\left(\pMX{\delta y}{0}{0}{1}k\right)f\left(\pMX{y}{0}{0}{1}k\right)\frac{d^{\x}y}{|y|_F}dk\\
&=
\int_K f(k)\int_{F^{\x}}
W\left(\begin{pmatrix}\delta y&0\\0&1\end{pmatrix}k\right)
\mu(y)|y|_F^{-\frac{1}{2}}d^{\x}y dk.
\end{split}
\end{align}
By equation \eqref{E:decomposition of rankin-selberg integral} and the fact that $K$ is compact,
it suffices to show that the inner integral in \eqref{E:decomposition of rankin-selberg integral}
converges absolutely under the assumption $\Lambda(\itPi)<1/2$. Let $\epsilon>0$ so that 
$1/2-\Lambda(\itPi)-\epsilon>0$. From \eqref{E:estimate Whittaker function}, we know that
\[
W\left(\begin{pmatrix}\delta y&0\\0&1\end{pmatrix}k\right)\mu(y)|y|_F^{-\frac{1}{2}}
\ll_{\Lambda(\itPi),W,\delta,\epsilon} |y|_F^{\frac{1}{2}-\Lambda(\itPi)-\epsilon}\Phi(y),
\]
for some Bruhat-Schwartz function $\Phi$ on $F$.
This shows that the integral \eqref{E:decomposition of rankin-selberg integral} is absolutely
convergent. Similar argument applies to $\t{\sR}_\delta(\t{W}\ot\t{f})$. 
\end{proof}

The proof of the following proposition will be postponed to \S\ref{S:regularization}.

\begin{prop}\label{P:regularization} 
We have 
\begin{equation}\label{E:regularization}
\sI(W\ot f;\t{W}\ot\t{f})
=
|\Delta|^{-\frac{1}{2}}_F\cdot\frac{\zeta_{\cK}(1)}{\zeta_F(1)}\cdot
\sR_\delta(W\ot f)\cdot\t{\sR}_\delta(\t{W}\ot\t{f}).
\end{equation}
\end{prop}

\begin{Remark}
Under certain temperedness assumption, Michel and Venkatesh proved \propref{P:regularization} by 
using the Whittaker Plancherel theorem for the case $\cK=F\x F$.  
In \cite[proposition 5.1]{ML2017}, Hsieh gave a different but more elementary proof for the case 
$\cK=F\x F$ and $F$ is non-archimedean. 
Moreover, he replaced the temperedness assumption
by a much weaker hypothesis $\Lambda(\itPi)<1/2$. Following the method in \cite{ML2017},  
I. Ishikawa in his thesis \cite[Theorem 5.1]{Ishikawa2017} proved \propref{P:regularization} 
when $\cK$ is a field and $F$ is non-archimedean. 
In this article, we prove \propref{P:regularization} when $F$ is archimedean 
in \S\ref{S:regularization} following the idea in \cite{ML2017}. 
\end{Remark}

\subsubsection{Godement section and intertwining map}
\label{SSS:Godement section and intertwining operator}
Let $\cS(F^2)$ be the space of Bruhat-Schwartz functions on $F^2$. 
Let $\cS(F^2,\psi)=\cS(F^2)$ when $F$ is non-archimedean and put
\begin{equation}\label{E:algebraic Schwartz space}
\cS(\C^2,\psi)=\C[x_1,\b{x}_1,x_2,\b{x}_2]e^{{-2\pi}|a\b{a}|^{1/2}_\C(x_1\b{x}_1+x_2\b{x}_2)}
\quad\text{and}\quad
\cS(\R^2,\psi)=\C[x_1,x_2]e^{-\pi|a|(x_1^2+x_2^2)}.
\end{equation}
Here $a$ is the non-zero number such that $\psi(x)=e^{2\pi i a x}$ when
$F=\R$ and $\psi(x)=e^{2\pi i (ax+\b{a}\b{x})}$ when $F=\C$.
Define the Fourier transform $\wh{\Phi}$ of $\Phi\in\cS(F^2)$ by
\[
\wh{\Phi}(x_1,x_2)
=
\int_{F}\int_{F}
\Phi(u,v)\psi(ux_2-vx_1)dudv,
\]
where $du, dv$ are the Haar measures on $F$ which are self-dual with respect to $\psi$.
Observe that $\cS(F^2,\psi)$ is invariant under the Fourier transform.
The following facts can be found in \cite{JLbook} and 
\cite[Section 4]{GelbartJacquet1979}.
Let $\mu,\nu$ be two characters of $F^\x$.
Let $s\in\C$ , $\Phi\in\cS(F^2,\psi)$ and $h\in\G(F)$. Define the $Godement$ $sections$
\begin{align*}
f^{(s)}_\Phi(g)
=
\mu(g)|g|_F^s
\int_{F^{\x}}\Phi((0,t)g)\mu\nu^{-1}(t)|t|^{2s}_F d^{\x}t
\quad\text{and}\quad
\t{f}^{(s)}_\Phi(g)
=
\nu(g)|g|_F^s
\int_{F^{\x}}\Phi((0,t)g)\mu^{-1}\nu(t)|t|^{2s}_F d^{\x}t.
\end{align*}
These two integrals converge absolutely when ${\rm Re}(s)\gg 0$ and have
meromorphic continuations to the whole complex plane. Moreover, whenever they are defined, we have
\[
f^{(s)}_\Phi\in\cB(\mu |\cdot|_F^{s-1/2},\nu |\cdot|_F^{1/2-s})
\quad\text{and}\quad
\t{f}^{(s)}_\Phi\in\cB(\nu |\cdot|_F^{s-1/2},\mu |\cdot|_F^{1/2-s}).
\]
Now suppose that $\rho(\mu,\nu)$ is $irreducible$. Then $f^{(s)}_\Phi$ and $\t{f}^{(s)}_\Phi$ are defined at $s=1/2$. 
We put 
\[
f_\Phi:=f^{(s)}_\Phi|_{s=1/2}
\quad\text{and}\quad
\t{f}_\Phi:=\t{f}^{(s)}_\Phi|_{s=1/2}.
\]
The linear maps $\Phi\mapsto f_\Phi$ and $\Phi\mapsto \t{f}_\Phi$ are surjective. 
Define the normalized intertwining operator 
\[
M_\psi^*(\mu,\nu,s)f^{(s)}_\Phi(g)
=
\gamma(2s-1,\mu\nu^{-1},\psi)
\int_F 
f^{(s)}_\Phi\left(
\pMX{0}{-1}{1}{0}\pMX 1x01 g
\right)dx.
\]
Here $dx$ is the Haar measure on $F$ which is self-dual with respect to $\psi$ and 
$\gamma(s,\mu\nu^{-1},\psi)$ is the Tate $\gamma$-factor attached to $\mu\nu^{-1}$ and $\psi$. 
This integral converges absolutely when ${\rm Re}(s)\gg 0$ and admits a meromorphic continuation 
to the whole complex plane. 
Moreover, we have $M_\psi^*(\mu,\nu,s)f^{(s)}_\Phi=\t{f}^{(1-s)}_{\wh{\Phi}}$
in the sense of the meromorphic continuations. We define the $\G(F)$-isomorphism
\begin{equation}\label{E:normalized intertwining map}
M^*_\psi(\mu,\nu):\cB(\mu,\nu)\stackrel{\sim}\longto\cB(\nu,\mu);
\quad
f_\Phi\mapsto\t{f}_{\wh{\Phi}}.
\end{equation}
Notice that $M^*_\psi(\mu,\nu)$ is well-defined (i.e. independent of the choice of $\Phi$).

\subsubsection{An identity between invariant forms: hermitian case}
\label{SSS:An identity between invariant forms: hermitian}
We need a variation of \propref{P:regularization} in terms of the $hermitian$ $pairings$. 
Define the $\G(\cK)$-equivariant hermitian pairing $\cH_{\pi_\cK}$ on $\cW(\pi_\cK,\psi_\cK)$ by
\begin{equation}\label{E:hermitian for Whittaker}
\cH_{\pi_\cK}(W,W')
=
\int_{\cK^{\x}}
W\left(\pMX z001\right)\ol{W'\left(\pMX{z}{0}{0}{1}\right)}d^{\x}z.
\end{equation}
On the other hand, the hermitian pairing $\cH_\pi$ on $\cB(\mu,\nu)$ is defined as follows
\begin{itemize}
\item
if $\pi$ is tempered, then 
\[
\cH_\pi(f,f')=\int_{K}f(k)\ol{f'(k)}dk;
\]
\item
if $\pi$ is complementary, then 
\[
\cH_{\pi}(f,f')=\int_K f(k)\ol{M_\psi^*(\mu,\nu)f'(k)}dk.
\] 
\end{itemize}
Define local trilinear integral in terms of the hermitian pairings by 
\[
\sI_{{\rm her}}(W\ot f;\t{W}\ot\t{f})
=
\int_{F^{\x}\backslash\G(F)}
\cH_{\pi_\cK}(\rho(g)W,W')\cH_{\pi}(\rho(g)f,f')dg.
\]

We introduce the local factors which will appear in the hermitian case.

\begin{defn}
Let ${\rm As}\,\pi_\cK$ be either the Asai transfer of $\pi_\cK$ to ${\rm GL}_4(F)$ as in
\cite{Krishnamurthy2003} or the Rankin-Selberg product as in \cite{JLbook2}, according to whether $\cK$ is a field or not. 
Then ${\rm As}\,\pi_\cK$ is an irreducible admissible generic representation of ${\rm GL}_4(F)$. 
Denote by $L(s,{\rm As}\,\pi_\cK\ot\mu)$ and $\epsilon(s,{\rm As}\,\pi_\cK\ot\mu,\psi)$ the 
$L$-factor and the $\epsilon$-factor attached to ${\rm As}\,\pi_\cK\ot\mu$ and $\psi$ as in
\cite[(3.6.4)]{Tate1979}.
The gamma factor is then given by
\[
\gamma\left(s,{\rm As}\,\pi_\cK\ot\mu,\psi\right)
=
\epsilon\left(s,{\rm As}\,\pi_\cK\ot\mu,\psi\right)
\frac{L\left(1-s,{\rm As}\,\t{\pi}_\cK\ot\mu^{-1}\right)}
{L\left(s,{\rm As}\,\pi_\cK\ot\mu\right)}.
\]
\end{defn}

The following corollary is the key of our computations for the local period integrals.  

\begin{cor}\label{C:regularization for hermitian}
Notation be as above.
Let $\lambda_{\cK/F}(\psi)$ be the Langlands constant given by \cite[Lemma 1.2]{JLbook}.
We have
\[
\sI_{{\rm her}}(W\ot f; W'\ot f')
=
C\cdot|\Delta|^{-\frac{1}{2}}_F\cdot\frac{\zeta_{\cK}(1)}{\zeta_F(1)}\cdot
\sR_\delta(W\ot f)\cdot\ol{\sR_\delta(W'\ot f')}
\]
where $C=1$ if $\pi$ is tempered, and 
$C=\omega_{\pi_\cK}(\delta)^{-1}\lambda_{\cK/F}(\psi)
\ol{\mu(\Delta)\gamma\left(1/2,{\rm As}\,\pi_\cK\ot\mu,\psi\right)}$
if $\pi$ is complementary.
\end{cor}

\begin{proof}
For $W\in\cW(\pi_{\cK},\psi_\cK)$, define $\t{W}(g)=\ol{W(J_2g)}$. Since  $\pi_\cK$ is unitary,
we have $\t{W}\in\cW(\t{\pi}_\cK,\psi_\cK)$ by the uniqueness of the Whittaker model. 
Let $W, W'\in\cW(\pi_\cK,\psi_\cK)$. Then we have $\cH_{\pi_\cK}(W,W')=\cB_{\pi_\cK}(W,\t{W}')$.
Similarly for $f\in\cB(\mu,\nu)$, define $\t{f}=\ol{f}$ if $\pi$ is tempered, and 
$\t{f}=\ol{M^*_\psi(\mu,\nu)f}$ if $\pi$ is complementary. Then $\t{f}\in\cB(\mu^{-1},\nu^{-1})$ and
we have $\cH_\pi(f,f')=\cB_{\pi}(f,\t{f}')$ for $f,f'\in\cB(\mu,\nu)$.
By these observations and \propref{P:regularization}, we find that
\[
\sI_{{\rm her}}(W\ot f; W'\ot f')
=
|\Delta|^{-\frac{1}{2}}_F\cdot\frac{\zeta_{\cK}(1)}{\zeta_F(1)}\cdot
\sR_\delta(W\ot f)\cdot\t{\sR}_\delta(\t{W}'\ot \t{f}').
\]
One checks easily that $\t{\sR}_\delta(\t{W}'\ot\t{f}')=\ol{\sR_\delta(W'\ot f')}$ when $\pi$ is 
tempered. Suppose that $\pi$ is complementary. We need to show that 
\[
\t{\sR}_\delta(\t{W}'\ot\t{f}')=C\cdot\ol{\sR_\delta(W'\ot f')}.
\] 
For this, we borrow the idea in the proof of \cite[Corollary 5.2]{ML2017}. Let $\psi^\delta_\cK$ be the 
character of $\cK$ defined by $\psi^\delta_\cK(z)=\psi_\cK(\delta z)$. Note that $\psi^\delta_\cK$ is
trivial on $F$. For $W\in\cW(\pi_\cK,\psi_\cK)$, define $W_\delta\in\cW(\pi_\cK,\psi^\delta_\cK)$ by
\[
W_\delta(g)=W\left(\pMX{\delta}{0}{0}{1}g\right)\quad (g\in\G(\cK)).
\]
Let $s\in\C$, $W\in\cW(\pi_\cK,\psi_\cK)$ and $\Phi\in\cS(F^2)$. Write $Z(s,W_\delta,\Phi)$ for the 
zeta integral associated to $W_\delta$ and $\Phi$ defined by \cite[(2.2.1)]{CCI2020}. 
This is the local zeta integral associated to the twisted tensor product $L$-function 
introduced in \cite{Flicker1988} (see also \cite{Kable2004}).
Pick $\Phi'\in\cS(F^2,\psi)$ so that $f'=f^{(s)}_{\Phi'}|_{s=1/2}$, where $f^{(s)}_{\Phi'}$ is the
Godement section defined in \S\ref{SSS:Godement section and intertwining operator}. Let $\mu'$ be 
a character of $\cK^{\x}$ such that $\mu'(y)=\mu(y)$ for every $y\in F^{\x}$. One verifies that 
\[
\sR_\delta(W'\ot f')
=
Z(s,W'_\delta\ot\mu',\Phi')|_{s=1/2}
\quad\text{and}\quad
\sR_\delta(\t{W}'\ot\t{f}')
=
\ol{Z(1-s,W_\delta'\ot\mu'^{-1}\omega_\cK^{-1},\wh{\Phi}')}|_{s=1/2}.
\]
The assertion for complementary $\pi$ follows from the functional equation of 
the zeta integrals \cite[Appendix]{Flicker1993}, \cite[Proposition 2]{Kable2004} and 
\cite[Theorem A and (2.1.2)]{CCI2020} (see also \cite{AKMSS2019}, \cite{Beuzart-Plessis2019}). 
\end{proof}

\subsection{Calculation of the local period integrals}
\label{SS:Calculations of local period integrals: archimedean case}
After the preparations in the previous subsections, we could begin our computations for the local period integrals 
$I^*(\itPi_v,\bft_v)$ appearing in the proof of \thmref{T: special value formula}.
As mentioned in the introduction, we will only focus on the computations for the real place in the followings. 
Indeed, by our assumptions \S\ref{SSS:global assumption}, the local period integrals for the non-archimedean case were 
essentially computed in \cite[Sections 4,5]{CC2019}; however, we should emphasize that our computations for the real 
place here are complementary to that of in loc. cit.

Let us explain why we can apply the results in \cite[Sections 4,5]{CC2019}. Firstly, our definitions for the local period 
integrals over non-archimedean local fields can be found in \cite[Section 2]{CC2019}, but now with the bilinear pairing 
replaced by the hermitian pairing. Secondly, the local period integrals in terms of the bilinear pairings were computed in 
\cite[Sections 4,5]{CC2019}. Finally, the results remain the same if one takes the following lemma into account.

\begin{lm}\label{L:bi=her}
Let $F$ be a non-archimedean local field and $\pi$ be a unitary irreducible admissible generic 
representation of $\G(F)$. Let $\cB_\pi$ (resp. $\cH_\pi$) be the $\G(F)$-equivariant bilinear 
 (resp. hermitian) pairing between $\pi$ and $\t{\pi}$ (resp. on $\pi$). Let $\phi_\pi\in\pi$  and 
 $\phi_{\t{\pi}}\in\t{\pi}$ be non-zero newforms and $t\in\G(F)$. 
 Then $\cB_\pi(\phi_\pi,\phi_{\t{\pi}})\neq 0$.
 Put
 \[
 \Phi_{{\rm bi}}(g;t)
 =
 \frac{\cB_\pi(\pi(gt)\phi_\pi,\t{\pi}(t)\phi_{\t{\pi}})}{\cB_\pi(\phi_\pi,\phi_{\t{\pi}})}
 \quad\text{and}\quad
 \Phi_{{\rm her}}(g;t)
 =
 \frac{\cH_\pi(\pi(gt)\phi_\pi,\pi(t)\phi_{\pi})}{\cH_\pi(\phi_\pi,\phi_{\pi})} .
 \]
 Then we have $\Phi_{{\rm bi}}(g;t)=\Phi_{{\rm her}}(g;t)$ for every $g\in\G(F)$.
\end{lm}

\begin{proof}
By the equivariant property, it suffices to prove that there is a non-zero constant $c$ such that
\[
\cB_\pi(\pi(g)\phi_\pi,\phi_{\t{\pi}})=c\cH_\pi(\pi(g)\phi_\pi,\phi_\pi)
\]
for every $g\in\G(F)$. Since $\pi$ is unitary, we have $\t{\pi}\cong\b{\pi}$, where $\b{\pi}$ is the  
complex conjugate representation of $\pi$. Now the lemma follows from the uniqueness of the $\G(F)$-equivariant
pairing and the theory of newforms.
\end{proof}

Let $F=\R$ from now on, so that either $E=\R\x\R\x\R$ or $E=\C\x\R$. 
Fix $\psi$ (resp. $\psi'$) to be the character of $\R$ (resp. $\C$) given by $\psi(x)=e^{2\pi i x}$ (resp. 
$\psi'(x)=e^{2\pi i(x+\b{x})}$) and $\delta$ in \corref{C:regularization for hermitian} to be
\begin{equation}\label{E:delta for R}
\delta
=
\begin{cases}
(-1,1)\quad&\text{if $E=\R\x\R\x\R$},\\
\sqrt{-1}\quad&\text{if $E=\C\x\R$}.
\end{cases}
\end{equation}
The Haar measures on $\G(\R)$ and $\R^\x\backslash\G(\R)$ are those described in 
\S\ref{SSS:An identity between invariant forms} with ${\rm Vol}({\rm SO}(2),dk)=1$.

Let $\itPi$ be a unitary irreducible admissible generic representations of 
$\G(E)$, whose central character of $\itPi$ is trivial on $\R^\x$ and $\Lambda(\itPi)<1/2$.
We first define the local period integrals $I^*(\itPi,\bft)$ to be computed.  
We will consider slightly general situations comparing with our global settings. 

\subsubsection{Setup for real case}
Suppose that $E=\R\x\R\x\R$ so that $\itPi=\pi_1\bt\pi_2\bt\pi_3$, where $\pi_1, \pi_2$ and $\pi_3$ are unitary
irreducible admissible generic representations of $\G(\R)$ with the minimal weights $k_1, k_2$ and $k_3$, respectively.
We make the following assumption.
\begin{itemize}
\item
One of $\pi_1,\pi_2, \pi_3$ is an irreducible induced representation.
\end{itemize}
Re-index if necessary we may assume that $k_1={\rm max}\stt{k_1, k_2, k_3}$ and 
$\pi_3\cong\rho(\mu_3,\nu_3)$ is an irreducible induced representation. 
This can be done since $k_1+k_2+k_3$ is even by our assumption on the central character. 
Note that this implies that $\pi_3$ is either a principal series representation or a limit of discrete series representation.
We may further assume that $\pi_j$ is isomorphic to a constituent of $\rho(\mu_j,\nu_j)$ for $j=1,2$, and  
$\mu_j\nu_j^{-1}=|\cdot|^{2s_j}{\rm sgn}^{k_j}$ for some $s_j\in\C$ for $j=1,2,3$.

\subsubsection{Setup for complex case}
Suppose that $E=\C\x\R$ so that $\itPi=\pi'\bt\pi$ where $\pi'$ (reps. $\pi$) is a unitary irreducible admissible generic
representation of $\G(\C)$ (resp. $\G(\R)$) with the minimal weight $k'$ (resp. $k$). We make the following assumption.
\begin{itemize}
\item
$\pi\cong\rho(\mu,\nu)$ is an irreducible induced representation.
\end{itemize}
Again, this implies that $\pi$ is either a principal series representation of a limit of discrete series representation.
We may assume that $\pi'\cong\pi(\mu',\nu')$ is a principal series representation. Let $s'\in\C$ (resp. $s\in\C$) such that 
$\mu'\nu'^{-1}(y)=|y|_\C^{2s'}\left(y/|y|_\C^{\frac{1}{2}}\right)^{\pm k'}$ (resp. $\mu\nu^{-1}=|\cdot|^{2s}{\rm sgn}^k$),
and $\epsilon\in\stt{0,1}$ such that $\mu'\mu(-1)=(-1)^{k'+\epsilon}$. 

\subsubsection{Test vectors}
We realize the representation space of $\itPi$ to be
\[
\cW(\pi_1,\psi)\ot\cW(\pi_2,\psi)\ot\cB(\mu_3,\nu_3)
\quad\text{or}\quad
\cW(\pi',\psi')\ot\cB(\mu,\nu)
\]
depending on $E=\R\x\R\x\R$ or $E=\C\x\R$.  We describe the test vectors which appear in the local period integrals 
under this realization. We follow the notations in \propref{P:test vector in Whittaker model}. If $E=\R\x\R\x\R$, then we put
\[
W^0=W_{\pi_1}\ot W_{\pi_2}\quad\text{and\quad} f^0=f_{\pi_3}.
\]
If $E=\C\x\R$, then we put
\[
W^0=W_{\pi'}\quad\text{and}\quad f^0=f_\pi,
\]
where $W_{\pi'}$ is given by
\[
W_{\pi'}(g):=
\begin{cases}
(\vec{W}_{\pi'}(g),v_{k',m})_{k'}\quad&\text{if $k'>0$ or $\epsilon=0$},\\
(\vec{W}_{\pi'}^{(2)}(g),v_{2,1})_2\quad&\text{if $k'=0$ and $\epsilon=1$}.
\end{cases}
\]
Here $m$ is defined by \eqref{E:l for CxR}.

\subsubsection{Raising elements}
We describe the raising elements $\bft$ appearing in the local period integrals. If the central character of $\itPi$ is trivial, 
then they are given by \S\ref{SSS:raising element}. But since we only assume the central character of $\itPi$ to be trivial 
on $\R^\x$ now, we need to define $\bft$ for the additional cases. 
Suppose that $E=\R\x\R\x\R$. If $\pi_1,\pi_2,$ are principal series representations, then the additional 
case occurs when $k_1=k_3=1$ and $k_2=0$, in which we put
\[
\bft=({\rm Id}\ot{\rm Id}\ot{\rm Id}),(I_2,I_2,J_2)).
\]
Suppose that $\pi_1$ is a discrete series representation. In the case where $\stt{k_2,k_3}=\stt{0,1}$,  we take $k_3=1$. 
Let $\ell=\frac{k_1-k_2-k_3}{2}$ and we put
\[
\bft
=
({\rm Id}\ot{\rm Id}\ot\t{V}_+^{\ell},(J_2,I_2,I_2)). 
\]
Suppose that $E=\C\x\R$. The additional case occurs when $k'$ is odd, in which we put 
\[
\bft=({\rm Id}\ot{\rm Id},(I_2,J_2)).
\]

\subsubsection{Local period integrals}
Given the descriptions for the test vectors and the raising elements, we now define the local period integrals 
$I^*(\itPi,\bft)$ to be 
\[
I^*(\itPi,\bft)
=
\frac{\zeta_{\R}(2)}{\zeta_{E}(2)}\cdot
\frac{L(1, \itPi, {\rm Ad})}{L(\frac{1}{2},\itPi, r)}\cdot
\frac{\sI_{{\rm her}}(\itPi(\bft)(W^0\ot f^0);\itPi(\bft)(W^0\ot f^0))}{\|W^0\|^2\cdot\|f^0\|^2}.
\]
It follows from Remark \ref{R:well-defined of m} that the local period integrals depend only on the isomorphism
class of $\itPi$. Moreover, it is independent of the constant multiple for the test vectors or the hermitian pairings, but it 
does depend on the choice of Haar measure on $\R^{\x}\backslash\G(\R)$.

\subsubsection{$L$-factors}
We need a simple lemma to deal with the $L$-factors appearing in the local period integrals. 
Let $W_\R$ (resp. $W_\C$) be the Weil group of $\R$ (resp. $\C$) (\cite[(1.4.3)]{Tate1979}). 
For a finite-dimensional semi-simple 
representation $\sigma$ of $W_\R$ or $W_\C$, we denote by $L(s,\sigma)$ the $L$-factor 
attached to $\sigma$ as in \cite[(3.3.1)]{Tate1979}. We refer to \cite{Knapp1994} for more details on
the classification of semi-simple representations of $W_\R$ or $W_\C$ and the definitions of the 
$L$-factors and the $\epsilon$-factors. The following lemma follows easily from the definition of the $L$-factor 
attached to a simple representation of $W_\R$.

\begin{lm}\label{L:property of L-factor}
Let $\sigma$ be a finite-dimensional semi-simple representation of $W_\R$. Suppose that 
$\t{\sigma}\cong\b{\sigma}$ where $\t{\sigma}$ (resp. $\b{\sigma}$) is the dual (resp. conjugate) 
representation of $\sigma$.  Then we have
\[
\ol{L(s,\sigma)}=L(\b{s},\t{\sigma}).
\]
In particular, if $\pi_\cK$ is a unitary irreducible admissible generic representation of $\G(\cK)$, 
where $\cK=\R\x\R$ or $\C$, then we have
\[
\ol{L(s,{\rm As}\,\pi_\cK)}=L(\b{s},{\rm As}\,\t{\pi}_\cK).
\]
\end{lm}


\subsubsection{Local period integrals for real case}
We compute that local period integrals when $E=\R\x\R\x\R$. We note the followings.
Firstly,  $\omega_\itPi=\prod_{j=1}^3\mu_j\nu_j$ is
trivial on $\R^{\x}$ implies that $k_1+k_2+k_3$ is even and 
$\mu_1\mu_2\mu_3=|\cdot|^{s_1+s_2+s_3}{\rm sgn}^\lambda$ for some $\lambda\in\stt{0,1}$.
Secondly, by \corref{C:regularization for hermitian} and the results in \S\ref{SS:Calculation of the norms}, it 
suffices to compute the local Rankin-Selberg integrals 
\[
\sR_\delta(\itPi(\bft)(W_{\pi_1}\ot W_{\pi_2}\ot f_{\pi_3})).
\]
Indeed, by \corref{C:regularization for hermitian} and the choices of $\delta$ in \eqref{E:delta for R} and $\psi$,
we have
\begin{equation}\label{E:reduction of local period for RxRxR}
I^*(\itPi,\bft)
=
\pi^2
\cdot
C
\cdot
\frac{L(1,\itPi,{\rm Ad})}{L(\frac{1}{2},\itPi,r)}
\cdot
\frac{\left|\sR_\delta(\itPi(\bft)(W_{\pi_1}\ot W_{\pi_2}\ot f_{\pi_3}))\right|^2}{\|W_{\pi_1}\|^2\|W_{\pi_2}\|^2\|f_{\pi_3}\|^2}.
\end{equation}
Finally, we have the following factorization of the $L$-factors
\begin{equation}\label{E:factorization of triple L for RxRxR}
L(s,\itPi,r)=L(s,{\rm As}\,(\pi_1\bt\pi_2)\ot{\mu_3})L(s,{\rm As}\,(\pi_1\bt\pi_2)\ot{\nu_3}),
\end{equation}
and the formulae of $L(s,{\rm As}\,(\pi_1\bt\pi_2)\ot{\mu_3})$ and 
$\epsilon(s,{\rm As}\,(\pi_1\bt\pi_2)\ot{\mu_3})$, where $\mu_3$ is trivial can be found in \cite[Section 17]{JLbook2},
while the non-trivial case can be deduced from that.


\begin{prop}\label{P:local period integral for RxRxR}
Suppose that $\pi_1, \pi_2$ are principal series representations. Then
\[
I^*(\itPi,\bft)
=
\begin{cases}
2^{-6}
&\quad\text{if $k_1=k_2=k_3=0$ and $\mu_1\mu_2\mu_3(-1)=-1$},\\
1
&\quad\text{otherwise}.
\end{cases}
\]
Suppose that $\pi_1$ is a discrete series representation. Then
\[
I^*(\itPi,\bft)
=
\begin{cases}
2^{-k_1+2k_2+2k_3-1}
&\quad\text{if $\pi_2$ is a principal series representation},\\
2^{-k_1+k_2+2k_3}
&\quad\text{if $\pi_2$ is a discrete series representation}.
\end{cases}
\]
\end{prop}

\begin{proof}
Suppose that $\pi_1,\pi_2$ are principal series representations.
We have the following three sub-cases:
\begin{itemize}
\item[(a)]
$k_1=k_2=k_3=0$ and $\mu_1\mu_2\mu_3(-1)=1$;
\item[(b)]
$k_1=k_2=k_3=0$ and $\mu_1\mu_2\mu_3(-1)=-1$;
\item[(c)]
$k_1=k_3=1$ and $k_2=0$.
\end{itemize}
We only compute the sub-case (b) since the others are similar. We point out why we need weight raising 
elements in this case even when the minimal weights are already matching. Indeed, since 
$W_{\pi_1}\ot W_{\pi_2}\ot f_{\pi_3}$ is ${\rm SO}(2)$-invariant on the right and
we have
\[
(W_{\pi_1}\ot W_{\pi_2})\left(\pMX{-y_1}{0}{0}{1}\right)f_{\pi_3}\left(\pMX{-y_2}{0}{0}{1}\right)
=\mu_1\mu_2\mu_3(-1)
(W_{\pi_1}\ot W_{\pi_2})\left(\pMX{y_1}{0}{0}{1}\right)f_{\pi_3}\left(\pMX{y_2}{0}{0}{1}\right)
\]
by \propref{P:test vector in Whittaker model} $(2)$. It follows from this observation and 
\eqref{E:decomposition of rankin-selberg integral} that $\sR_\delta(W_{\pi_1}\ot W_{\pi_2}\ot f_{\pi_3})=0$ if we do not 
use the raising elements.
Back to the computation for the sub-case $(b)$. We have
\[
\itPi(\bft)(W_{\pi_1}\ot W_{\pi_2}\ot f_{\pi_3})=W_{\pi_1}\ot\rho(\t{V}_+)W_{\pi_2}\ot\rho(J_2)\rho(\t{V}_+)f_{\pi_3}
\]
which is ${\rm SO}(2)$-invariant on the right. 
By 
\propref{P:test vector in Whittaker model} $(2)$,
\eqref{E:raising element on section} and
\lmref{L:integration formula for Bessel function} $(3)$, we find that
\begin{align*}
\sR_\delta&(\itPi(\bft)(W_{\pi_1}\ot W_{\pi_2}\ot f_{\pi_3}))\\
&=
\mu_2(-1)2^{-2}\pi^{-1}\frac{\Gamma(s_3+\frac{3}{2})}{\Gamma(s_3+\frac{1}{2})}
\int_{0}^\infty
y^{s_3+\frac{1}{2}}
K_{s_1}(2\pi y)K_{s_2}(2\pi y)dy\\
&=
\mu_2(-1)2^{-5}\pi^{-s_3-\frac{5}{2}}
\frac
{\Gamma\left(\frac{s_1+s_2+s_3}{2}+\frac{3}{4}\right)
\Gamma\left(\frac{s_1-s_2+s_3}{2}+\frac{3}{4}\right)
\Gamma\left(\frac{-s_1+s_2+s_3}{2}+\frac{3}{4}\right)
\Gamma\left(\frac{-s_1-s_2+s_3}{2}+\frac{3}{4}\right)}
{\Gamma\left(s_3+\frac{1}{2}\right)}\\
&=
\mu_2(-1)2^{-5}\pi^{s_3+\frac{1}{2}}
\frac
{L\left(\frac{1}{2},{\rm As}\,(\pi_1\bt\pi_2)\ot\mu_3\right)}
{\Gamma\left(s_3+\frac{1}{2}\right)}.
\end{align*}
The adjoint $L$-factor when evaluated at $s=1$ is given by
\begin{equation}\label{E:adjoint of triple L for RxRxR}
L(1,\itPi,{\rm Ad})=\pi^{-3}\prod_{j=1}^3\Gamma(1/2+s_j)\Gamma(1/2-s_j).
\end{equation}
Suppose that $\pi_3$ is tempered. Then both $\mu_3$ and $\nu_3$
are unitary characters so that  $s_3$ is a purely imaginary number. 
It follows from the assumption on the central character of $\itPi$ and 
\lmref{L:property of L-factor} that
\begin{align}\label{E:conjugate for tempered L}
\begin{split}
\ol{L(1/2,{\rm As}\,(\pi_1\bt\pi_2)\ot\mu_3)}
&=
L(1/2,{\rm As}\,(\t{\pi}_1\bt\t{\pi}_2)\ot\mu_3^{-1})\\
&=
L(1/2,{\rm As}\,(\pi_1\bt\pi_2)\ot(\omega_{\pi_1}\omega_{\pi_2}\mu_3)^{-1})\\
&=
L(1/2,{\rm As}\,(\pi_1\bt\pi_2)\ot\nu_3).
\end{split}
\end{align}
Here we have use the facts that $\t{\pi}_j\cong\pi_j\ot\omega^{-1}_{\pi_j}$ and 
${\rm As}\,(\pi_1\ot\omega_{\pi_1}^{-1}\bt\pi_2\ot\omega_{\pi_2}^{-1})\cong{\rm As}\,(\pi_1\bt\pi_2)\ot(\omega_{\pi_1}\omega_{\pi_2})^{-1}$, where $j=1,2$.
It follows that
\[
\ol{\sR_\delta(\itPi(\bft)(W_{\pi_1}\ot W_{\pi_2}\ot f_{\pi_3}))}
=
\mu_2(-1)2^{-5}\pi^{-s_3+\frac{1}{2}}
\frac{L(\frac{1}{2},{\rm As}\,(\pi_1\bt\pi_2)\ot\nu_3)}{\Gamma(-s_3+\frac{1}{2})}.
\]
From these together with 
\lmref{L:norm of Whittaker function for archimedean},
\lmref{L:norms for principal series for R} 
and
\eqref{E:factorization of triple L for RxRxR}, 
\eqref{E:adjoint of triple L for RxRxR},
 \eqref{E:conjugate for tempered L}, one finds that the RHS of \eqref{E:reduction of local period for RxRxR}
is equal to $1$.
Suppose that $\pi_3$ is complementary. Then we have $\mu_3=\chi |\cdot|^{\lambda}$ and 
$\nu_3=\chi |\cdot|^{-\lambda}$ for some unitary character $\chi$ and $\lambda\in\R$ with
$0<|\lambda|<1/2$. Note that $\lambda=s_3$. By a similar argument, we find that 
\begin{equation}\label{E:conjugate for non-tempered L1}
\ol{L(1/2,{\rm As}\,(\pi_1\bt\pi_2)\ot\mu_3)}
=
\ol{L(1/2+\lambda,{\rm As}\,(\pi_1\bt\pi_2)\ot\chi)}
=
L(1/2,{\rm As}\,(\pi_1\bt\pi_2)\ot\mu_3)
\end{equation}
and
\begin{equation}\label{E:conjugate for non-tempered L2}
\ol{L(1/2, {\rm As}\,(\t{\pi}_1\bt\t{\pi}_2)\ot\mu_3^{-1})}
=
\ol{L(1/2-\lambda,{\rm As}\,(\t{\pi}_1\bt\t{\pi}_2)\ot\chi^{-1})}
=
L(1/2,{\rm As}\,(\pi_1\bt\pi_2)\ot\nu_3).
\end{equation}
It follows that
\[
\ol{\sR_\delta(\itPi(\bft)(W_{\pi_1}\ot W_{\pi_2}\ot f_{\pi_3}))}
=
\sR_\delta(\itPi(\bft)(W_{\pi_1}\ot W_{\pi_2}\ot f_{\pi_3})).
\]
On the other hand, we have
\[
\omega_{\pi_1\bt\pi_2}(\delta)
=
\lambda_{\R\x\R/\R}(\psi)=\mu_3(\Delta)=\epsilon(1/2,{\rm As}\,(\pi_1\bt\pi_2)\ot\mu_3,\psi)=1.
\]
From these together with
 \lmref{L:norm of Whittaker function for archimedean},
 \lmref{L:norms for principal series for R} 
and  
\eqref{E:factorization of triple L for RxRxR}, 
\eqref{E:adjoint of triple L for RxRxR},
\eqref{E:conjugate for non-tempered L1}, 
\eqref{E:conjugate for non-tempered L2}, we find that 
the RHS of \eqref{E:reduction of local period for RxRxR} is $2^{-6}$.
This finishes the computations of the sub-case $(b)$.

Suppose that $\pi_1$ is a discrete series representation while $\pi_2$ is a principal series representation.
In this case, we have
\[
\itPi(\bft)(W_{\pi_1}\ot W_{\pi_2}\ot f_{\pi_3})
=
\rho(J_2)W_{\pi_1}\ot W_{\pi_2}\ot\rho(\t{V}^\ell_+)f_{\pi_3}
\]
which is ${\rm SO}(2)$-invariant on the right. We also have the following three sub-cases:
\begin{itemize}
\item[(a)]
$k_1\equiv 0\pmod 2$ and $k_2=k_3=0$;
\item[(b)]
$k_1\equiv 0\pmod 2$ and $k_2=k_3=1$;
\item[(c)]
$k_1\equiv 1\pmod 2$ and $k_2=0, k_3=1$.
\end{itemize}
We only compute the sub-case (c) as the others are similar.
By 
\propref{P:test vector in Whittaker model} $(1)$, $(2)$,
\eqref{E:raising element on section} and
\lmref{L:integration formula for Bessel function} $(2)$, one has
\begin{align*}
\sR_\delta&(\itPi(\bft)(W_{\pi_1}\ot W_{\pi_2}\ot f_{\pi_3}))\\
&=
(-1)^{\frac{k_1-1}{2}} 2^{-k_1+1}\pi^{\frac{-k_1+1}{2}}
\frac{\Gamma\left(s_3+\frac{k_1+1}{2}\right)}{\Gamma\left(s_3+1\right)}
\int_0^\infty
y^{\frac{k_1}{2}+s_3-1}e^{-2\pi y}K_{s_2}(2\pi y)dy\\
&=
(-1)^{\frac{k_1-1}{2}}2^{1-2k_1-2s_3}\pi^{1-k_1-s_3}
\frac
{\Gamma\left(\frac{k_1}{2}+s_2+s_3\right)
\Gamma\left(\frac{k_1}{2}-s_2+s_3\right)}
{\Gamma\left(s_3+1\right)}\\
&=
(-1)^{\frac{k_1-1}{2}}2^{-1-k_1}\pi^{1+s_3}
\frac
{L\left(\frac{1}{2},{\rm As}\,(\pi_1\bt\pi_2)\ot\mu_3\right)}
{\Gamma(s_3+1)}.
\end{align*}
Note that $\pi_3$ is tempered as $k_3=1$. Therefore $s_3$ is a purely imaginary number.
 Also we have
\[
L(1,\itPi,{\rm Ad})
=
2^{1-k_1}\pi^{-4-k_1}\Gamma(k_1)\Gamma(s_2+1/2)\Gamma(-s_2+1/2)\Gamma(s_3+1)\Gamma(-s_3+1).
\]
Now we can proceed as in the previous case to obtain $I^*(\itPi,\bft)=2^{1-k_1}$. This proves the sub-case $(c)$.

Suppose that $\pi_1,\pi_2$ are discrete series representations.
In this case, we have
\[
\itPi(\bft)(W_{\pi_1}\ot W_{\pi_2}\ot f_{\pi_3})
=
\rho(J_2)W_{\pi_1}\ot W_{\pi_2}\ot\rho(\t{V}^\ell_+)f_{\pi_3}
\]
which is ${\rm SO}(2)$-invariant on the right. 
By \propref{P:test vector in Whittaker model} $(1)$ and
\eqref{E:raising element on section}, we find that
\begin{align*}
\sR_\delta&(\itPi(\bft)(W_{\pi_1}\ot W_{\pi_2}\ot f_{\pi_3}))\\
&=
(-1)^{\frac{k_1-k_2-k_3}{2}}2^{-k_1+k_2+k_3}\pi^{\frac{-k_1+k_2+k_3}{2}}
\frac{\Gamma\left(\frac{k_1-k_2+1}{2}+s_3\right)}{\Gamma\left(s_3+\frac{k_3+1}{2}\right)}
\int_0^\infty
y^{\frac{k_1+k_2-1}{2}+s_3-1}e^{-4\pi y}dy\\
&=
(-1)^{\frac{k_1-k_2-k_3}{2}}2^{-2k_1+k_3-2s_3+1}\pi^{-k_1+\frac{1+k_3}{2}-s_3}
\frac{\Gamma\left(\frac{k_1+k_2-1}{2}+s_3\right)\Gamma\left(\frac{k_1-k_2+1}{2}+s_3\right)}
{\Gamma\left(s_3+\frac{k_3+1}{2}\right)}\\
&=
(-1)^{\frac{k_1-k_2-k_3}{2}}2^{-1-k_1+k_3}\pi^{\frac{k_3+1}{2}+s_3}
\frac{L(\frac{1}{2},{\rm As}\,(\pi_1\bt\pi_2)\ot\mu_3)}{\Gamma(s_3+\frac{k_3+1}{2})}.
\end{align*}
The adjoint $L$-factor when evaluated at $s=1$ is given by
\[
L(1,\itPi,{\rm Ad})
=
2^{2-k_1-k_2}\pi^{-3-k_1-k_2-k_3}\Gamma(k_1)\Gamma(k_2)
\Gamma(s_3+(k_3+1)/2)\Gamma(-s_3+(k_3+1)/2),
\]
If $\pi_3$ is complementary, then  
\[
\omega_{\pi_1\bt\pi_2}(\delta)=\epsilon(1/2,{\rm As}\,(\pi_1\bt\pi_2)\ot\mu_3,\psi)=(-1)^{k_1}
\quad\text{and}\quad
\lambda_{\R\x\R/\R}(\psi)=\mu_3(\Delta)=1.
\]
Similar arguments as in the proof of the first case show that $I^*(\itPi,\bft)=2^{-k_1+k_2+2k_3}$.
This completes the proof of the third case.
\end{proof}

\subsubsection{Local period integrals for complex case}
We compute the local period integrals when $E=\C\x\R$.
Note that our assumption on $\omega_\itPi$ implies that $k'+k$ is even and
$\mu'|_{\R^{\x}}\cdot\mu=|\cdot|_\R^{2s'+s}{\rm sgn}^\lambda$ for some $\lambda\in\stt{0,1}$.
To compute the local Rankin-Selberg integrals, we need a lemma.  
Let $W\in\cW(\pi',\psi')$. Define the following analogue of the Tate integral
\[
\zeta(W,\mu)
=
\int_{\R^{\x}}
W\left(\begin{pmatrix}
iy&0\\0&1
\end{pmatrix}\right)
\mu(y)|y|^{-\frac{1}{2}}d^{\x}y.
\]
From the proof of \lmref{L:converge for rankin-selberg integral}, one sees that the assumption 
$\Lambda(\itPi)<1/2$ implies that $\zeta(W,\mu)$ converges absolutely.

\begin{lm}\label{L:Tate integral for CxR}
Let $W_j$ be given by \propref{P:test vector in Whittaker model} $(3)$ with 
$s$ and $k$ replaced by $s'$ and $k'$ respectively with $0\leq j\leq k'$. 
If $\mu'\mu(-1)=(-1)^j$, then  
\[
\zeta(W_j,\mu)
=
2^{-1}\mu'(\sqrt{-1})
(2\pi)^{-s-\frac{k'}{2}-\frac{1}{2}}
\Gamma\left(s'+\frac{s}{2}+\frac{j}{2}+\frac{1}{4}\right)
\Gamma\left(-s'+\frac{s}{2}+\frac{k'}{2}-\frac{j}{2}+\frac{1}{4}\right).
\]
If $\mu'\mu(-1)\neq (-1)^j$, then the integral vanishes.
\end{lm}

\begin{proof}
We have
\[
\zeta(W_j,\mu)
=
\mu'(\sqrt{-1})
\int_{\R^{\x}}
K_{2s'-\frac{k'}{2}+j}(4\pi |y|)\mu'\mu(y){\rm sgn}^j(y)|y|^{-2s'+\frac{k'}{2}-\frac{1}{2}}dy.
\]
In particular, if $\mu'\mu(-1)\neq (-1)^j$, then the integral vanishes. Assume that $\mu'\mu(-1)=(-1)^j$. 
By \lmref{L:integration formula for Bessel function} $(1)$, we find that
\begin{align*}
\zeta(W_j,\mu)
&=
2\mu'(\sqrt{-1})
\int_0^\infty
y^{s+\frac{k'}{2}-\frac{1}{2}}K_{2s'-\frac{k'}{2}+j}(4\pi y)dy\\
&=
2^{-1}\mu'(\sqrt{-1})
(2\pi)^{-s-\frac{k'}{2}-\frac{1}{2}}
\Gamma\left(s'+\frac{s}{2}+\frac{j}{2}+\frac{1}{4}\right)
\Gamma\left(-s'+\frac{s}{2}+\frac{k'}{2}-\frac{j}{2}+\frac{1}{4}\right).
\end{align*}
This finishes the proof.
\end{proof}

The following lemma helps us to simplify the equations.  
  
\begin{lm}\label{L:combinatorial identity}
Let $z,w$ be two complex numbers and $N$ be a 
non-negative integer. Then 
\[
\sum_{j=0}^N
\begin{pmatrix}
N\\j
\end{pmatrix}
\Gamma\left(z+j\right)\Gamma\left(w-j\right)
=
\frac{\Gamma(z)\Gamma(w-N)\Gamma(z+w)}{\Gamma(z+w-N)}.
\] 
\end{lm}

\begin{proof}
This can be proved by using the induction on $N$.
\end{proof}

As we mentioned in the case when $E=\R\x\R\x\R$, it suffices to compute
the Rankin-Selberg integral $\sR_\delta(\itPi(\bft)(W_{\pi'}\ot f_\pi))$. 
By \corref{C:regularization for hermitian} and the choices of $\delta$ in \eqref{E:delta for R} and $\psi$, we have
\begin{equation}\label{E:reduction of local period for CxR}
I^*(\itPi,\bft)
=
2\pi
\cdot
C
\cdot
\frac{L(1,\itPi,{\rm Ad})}{L(\frac{1}{2},\itPi,r)}
\cdot
\frac{\left|\sR_\delta(\itPi(\bft)(W_{\pi'}\ot f_\pi))\right|^2}{\|W_{\pi'}\|^2\|f_\pi\|^2}.
\end{equation}
We also have the following factorization of $L$-factors
\begin{equation}\label{E:factorization of triple L for CxR}
L(s,\itPi,r)=L(s,{\rm As}\,\pi'\ot\mu)L(s,{\rm As}\,\pi'\ot\nu)
\end{equation}
and the formulae of $L(s,{\rm As}\,\pi'\ot\mu)$ and $\epsilon(s,{\rm As}\,\pi'\ot\mu)$, where $\mu$ is trivial 
can be found in \cite[Section 3]{CCI2020}, while the non-trivial case can be deduced from that.

\begin{prop}\label{P:local period integral for CxR}
Let $\epsilon\in\stt{0,1}$ so that $\mu'\mu(-1)=(-1)^{k'+\epsilon}$. 
Then
\[
I^*(\itPi,\bft)
=
\begin{cases}
2^{-k'-2\epsilon}
\frac{\Gamma(k'+2)}{\Gamma(k'-m+1)\Gamma(m+1)}
&\quad\text{if $k'>0$ or $\epsilon=0$},\\
48\pi^2(1/4-s'^2)^{-1}
&\quad\text{if $k'=0$ and $\epsilon=1$}.
\end{cases}
\]
Here $m$ is given by \eqref{E:l for CxR}.
\end{prop}

\begin{proof}
There are four cases depending on the parities of $k'$ and $\epsilon$
\begin{itemize}
\item[(a)]
$k'\geq 0$ is even and $\epsilon=0$;
\item[(b)]
$k'\geq 2$ is even and $\epsilon=1$;
\item[(c)]
$k'> 0$ is odd and $\epsilon=0$;
\item[(d)]
$k'=0$ and $\epsilon=1$.
\end{itemize}
Note that when $k'$ is odd, the two cases $\epsilon=0, 1$ are essentially the same.
Indeed, if $\mu'\mu(-1)=-1$, then $\nu'\mu(-1)=-\mu^2\mu'\nu'(-1)=1$. 
As the local period integrals only depend on the 
isomorphism class of $\itPi$, we can replace $\pi'$ by $\pi(\nu',\mu')$.
The calculations for the cases $(b)$, $(c)$ are similar, so we only demonstrate the computations
for the cases $(a)$, $(c)$ and $(d)$. We note that similar arguments show that \eqref{E:conjugate for tempered L},
\eqref{E:conjugate for non-tempered L1} and \eqref{E:conjugate for non-tempered L2} still hold if we replace
$\pi_1\bt\pi_2$ by $\pi'$.

$(a)$ In this case, we have
\[
\itPi(\bft)(W_{\pi'}\ot f_\pi)
=W_{\pi'}\ot f_\pi
\]
which is ${\rm SO}(2)$- invariant on the right. 
By \lmref{L:Tate integral for CxR} and \lmref{L:combinatorial identity}, we find that
\begin{align*}
\sR_\delta&(\itPi(\bft)(W_{\pi'}\ot f_\pi))\\
&=
\sum_{j=0}^{\frac{k'}{2}}
\begin{pmatrix}
\frac{k'}{2}\\j
\end{pmatrix}
\zeta(W_{2j},\mu)\\
&=
2^{-1}\mu'\left(\sqrt{-1}\right)
(2\pi)^{-\left(s+\frac{k'}{2}+\frac{1}{2}\right)}
\sum_{j=0}^{\frac{k'}{2}}
\begin{pmatrix}
\frac{k'}{2}\\j
\end{pmatrix}
\Gamma\left(s'+\frac{s}{2}+\frac{1}{4}+j\right)
\Gamma\left(-s'+\frac{s}{2}+\frac{k'}{2}+\frac{1}{4}-j\right)\\
&=
2^{-1}\mu'\left(\sqrt{-1}\right)
(2\pi)^{-\left(s+\frac{k'}{2}+\frac{1}{2}\right)}
\frac{
\Gamma\left(s'+\frac{s}{2}+\frac{1}{4}\right)
\Gamma\left(-s'+\frac{s}{2}+\frac{1}{4}\right)
\Gamma\left(s+\frac{k'}{2}+\frac{1}{2}\right)
}
{\Gamma\left(s+\frac{1}{2}\right)}\\
&=
\mu'\left(\sqrt{-1}\right)2^{-2}\pi^{s+\frac{1}{2}}
\frac{L(\frac{1}{2},{\rm As}\,\pi'\ot\mu)}{\Gamma(s+\frac{1}{2})}.
\end{align*}
Note that we have
\[
L(1,\itPi,{\rm Ad})
=
2^{-k'}\pi^{-k'-4}\Gamma(2s'+k'/2+1)\Gamma(-2s'+k'/2+1)\Gamma(s+1/2)\Gamma(-s+1/2).
\]
We first deal with the case where $\pi$ is tempered. 
In this case, both $\mu$ and $\nu$ are unitary characters so that $s$ is a purely 
imaginary number. By \eqref{E:conjugate for tempered L}, we find that
\[
\ol{\sR_\delta(\itPi(\bft)(W_{\pi'}\ot f_\pi))}
=
\mu'(-\sqrt{-1})2^{-2}\pi^{-s+\frac{1}{2}}
\frac{L(\frac{1}{2},{\rm As}\,\pi'\ot\nu)}{\Gamma(-s+\frac{1}{2})}.
\]
It then follows from  \eqref{E:norm for v_n,j},  \lmref{L:norm of Whittaker function for archimedean} and
\lmref{L:norms for principal series for R} that the RHS of 
\eqref{E:reduction of local period for CxR} is equal to
\[
2\pi
\cdot
\frac{L(1,\itPi,{\rm Ad})}{L(\frac{1}{2},\itPi,r)}
\cdot
2^{-1}\pi^{-1}
\frac{L(\frac{1}{2},\itPi,r)}{L(1,\itPi,{\rm Ad})}
\cdot
\frac{k'+1}{\left(v_{k',\frac{k'}{2}},v_{k',\frac{k'}{2}}\right)_{k'}}
=
2^{-k'}\frac{\Gamma(k'+2)}{\Gamma(\frac{k'}{2}+1)^2}.
\]
Suppose that $\pi$ is complementary. Then we have $\mu=\chi |\cdot|^{\lambda}$ and 
$\nu=\chi |\cdot|^{-\lambda}$ for some unitary character $\chi$ and $\lambda\in\R$ with
$0<|\lambda|<1/2$. Note that $s=\lambda$. By \eqref{E:conjugate for non-tempered L1}, 
we see that
\[
\ol{\sR_\delta(\itPi(\bft)(W_{\pi'}\ot f_\pi))}
=
\mu'(-1)\sR_\delta(\itPi(\bft)(W_{\pi'}\ot f_\pi)).
\]
On the other hand, we have
\begin{align*}
\omega_{\pi'}(\sqrt{-1})^{-1}\cdot\lambda_{\C/\R}(\psi)
\cdot
\mu(-1)\cdot\ol{\epsilon(1/2,{\rm As}\,\pi'\ot\mu,\psi)}
=
1.
\end{align*}
Applying \lmref{L:norm of Whittaker function for archimedean},
\lmref{L:norms for principal series for R} and \eqref{E:norm for v_n,j}, \eqref{E:conjugate for non-tempered L1}, 
\eqref{E:conjugate for non-tempered L2},  we find that the RHS of 
\eqref{E:reduction of local period for CxR} is equal to
\[
2\pi
\cdot
\frac{L(\frac{1}{2},{\rm As}\,\pi'\ot\nu)}{L(\frac{1}{2},{\rm As}\,\pi'\ot\mu)}
\cdot
\frac{L(1,\itPi,{\rm Ad})}{L(\frac{1}{2},\itPi,r)}
\cdot
2^{-1}\pi^{-1}
\frac{L(\frac{1}{2},{\rm As}\,\pi'\ot\mu)^2}{L(1,\itPi,{\rm Ad})}
\cdot
\frac{k'+1}{\left(v_{k',\frac{k'}{2}},v_{k',\frac{k'}{2}}\right)_{k'}}
=
2^{-k'}\frac{\Gamma(k'+2)}{\Gamma(\frac{k'}{2}+1)^2}.
\]
This finishes the computations of the case $(a)$.

$(c)$ In this case, we have
\[
\itPi(\bft)(W_{\pi'}\ot f_\pi)
=
W_{\pi'}\ot\rho(J_2)f_\pi
\]
which is ${\rm SO}(2)$-invariant on the right. 
By \lmref{L:Tate integral for CxR} and \lmref{L:combinatorial identity}, we find that
\begin{align*}
\sR_\delta&(\itPi(\bft)(W_{\pi'}\ot f_\pi))\\
&=
\mu(-1)\sum_{j=0}^{\frac{k'-1}{2}}
\begin{pmatrix}
\frac{k'-1}{2}\\j
\end{pmatrix}
\zeta(W_{2j+1},\mu)\\
&=
2^{-1}\mu(-1)\mu'\left(\sqrt{-1}\right)
(2\pi)^{-\left(s+\frac{k'}{2}+\frac{1}{2}\right)}
\sum_{j=0}^{\frac{k'-1}{2}}
\begin{pmatrix}
\frac{k'-1}{2}\\j
\end{pmatrix}
\Gamma\left(s'+\frac{s}{2}+\frac{3}{4}+j\right)
\Gamma\left(-s'+\frac{s}{2}+\frac{k'}{2}-\frac{1}{4}-j\right)\\
&=
2^{-1}\mu(-1)\mu'\left(\sqrt{-1}\right)
(2\pi)^{-\left(s+\frac{k'}{2}+\frac{1}{2}\right)}
\frac{
\Gamma\left(s'+\frac{s}{2}+\frac{3}{4}\right)
\Gamma\left(-s'+\frac{s}{2}+\frac{1}{4}\right)
\Gamma\left(s+\frac{k'}{2}+\frac{1}{2}\right)
}
{\Gamma\left(s+1\right)}\\
&=
\mu(-1)\mu'\left(\sqrt{-1}\right)2^{-2}\pi^{1+s}
\frac{L(\frac{1}{2},{\rm As}\,\pi'\ot\mu)}{\Gamma(s+1)}.
\end{align*}
The adjoint $L$-factor when evaluated at $s=1$ is given by
\[
L(1,\itPi,{\rm Ad})
=
2^{-k'}\pi^{k'-5}
\Gamma(2s'+k'/2+1)
\Gamma(-2s'+k'/2+1)
\Gamma(s+1)
\Gamma(-s+1).
\]
Observe that $\pi$ is tempered since $k=1$. We can proceed as in the case $(a)$ to show that 
\[
I^*(\itPi,\bft)=2^{-k'}\frac{\Gamma(k'+2)}{\Gamma(\frac{k'+1}{2})\Gamma(\frac{k'+3}{2})}.
\]
This shows the case $(c)$.

(d) We first explain why we need to separate this case from the second. 
Note that $\vec{W}_{\pi'}$ given by 
\propref{P:test vector in Whittaker model} $(3)$ is actually a scalar-valued function and ${\rm SO}(2)$-invariant
on the right. We show that $\sR_\delta(\vec{W}_{\pi'}\ot f)=0$
for every $f\in\cB(\mu,\nu)$. We may assume that $f(k(\theta))=e^{in\theta}$ for some even integer $n$.
By \propref{P:test vector in Whittaker model} $(3)$, we find that
\[
\vec{W}_{\pi'}\left(\pMX{-y_1}{0}{0}{1}k(\theta)\right)f\left(\pMX{-y_2}{0}{0}{1}k(\theta)\right)
=
-e^{in\theta}\vec{W}_{\pi'}\left(\pMX{y_1}{0}{0}{1}\right)f\left(\pMX{y_2}{0}{0}{1}\right).
\]
From this and \eqref{E:decomposition of rankin-selberg integral},
we see that $\sR_\delta(\vec{W}_{\pi'}\ot f)=0$ for every $f\in\cB(\mu,\nu)$.
Back to the computations for the case $(d)$. We have
\[
\itPi(\bft)(W_{\pi'}\ot f_\pi)
=
W_{\pi'}\ot f_\pi
\]
which is ${\rm SO}(2)$-invariant on the right. 
By \propref{P:test vector in Whittaker model} $(3)$ and \lmref{L:integration formula for Bessel function} $(1)$, we obtain
\begin{align*}
\sR_\delta(\itPi(\bft)(W_{\pi'}\ot f_\pi))
&=
\zeta\left(W^{(2)}_0,\mu\right)+\zeta\left(W^{(2)}_2,\mu\right)\\
&=
-\mu'\left(\sqrt{-1}\right)2^{-s+\frac{1}{2}}\pi^{-s-\frac{3}{2}}
\Gamma\left(s'+\frac{s}{2}+\frac{3}{4}\right)
\Gamma\left(-s'+\frac{s}{2}+\frac{3}{4}\right)\\
&=
-\mu'(\sqrt{-1})\pi^{s+\frac{1}{2}}
\frac{L(\frac{1}{2},{\rm As}\,\pi'\ot\mu)}{\Gamma(s+\frac{1}{2})}.
\end{align*}
We need to compute the $L^2$-norm $\|W_{\pi'}\|^2$, which follows easily from 
\lmref{L:norm of Whittaker function for archimedean} and 
\eqref{E:norm for v_n,j},
\eqref{E:Schur's orthogonal relation}
\[
\|W_{\pi'}\|^2
=
\frac{(v_{2,1},v_{2,1})_2}{3}\|\vec{W}^{(2)}_{\pi'}\|^2
=
2^{-5}3^{-1}\pi^{-4}\Gamma(2s'+2)\Gamma(-2s'+2).
\]
Note that we have
\[
L(1,\itPi,{\rm Ad})
=
\pi^{-4}
\Gamma(2s'+1)
\Gamma(-2s'+1)
\Gamma(s+1/2)
\Gamma(-s+1/2).
\]
Also if $\pi$ is complementary, then one checks that 
\begin{align*}
\omega_{\pi'}(\sqrt{-1})^{-1}\cdot\lambda_{\C/\R}(\psi)
\cdot
\mu(-1)\cdot\ol{\epsilon(1/2,{\rm As}\,\pi'\ot\mu,\psi)}
=
1.
\end{align*}
Similar arguments as in the proof of the case $(a)$ show that $I^*(\itPi,\bft)=48\pi^2(1/4-s'^2)^{-1}$. 
This completes the proof.
\end{proof}

\begin{Remark}
By \cite[Lemma 6.1 (i)]{JLbook}, the positive real number $(1/4-s'^2)$ is essentially the eigenvalue 
of a central element of $\frak{U}_\C$ when acting on $\pi'$.
\end{Remark}

\section{Proof of \propref{P:regularization}}\label{S:regularization}
The purpose of this section is to prove \propref{P:regularization} when $F$ is archimedean.
Therefore $F=\R$ or $\C$ in this section. We follow the notations in 
\S\ref{SSS:An identity between invariant forms}. It is not hard to see that if
\propref{P:regularization} holds for $\psi$ and $\delta$, then it also 
holds for $\psi_a$ and $b\delta$ for any $a,b\in F^{\x}$. Here $\psi_a$ is a character of $F$ 
defined by $\psi_a(x)=\psi(ax)$. As a consequence of this observation, we let $\psi$ be the 
character of $F$ as in \eqref{E:additive character of archimedean}, and let
$\delta$ be given by \eqref{E:delta for R} in the rest of this section. The Haar measures on
various groups are those stated in \S\ref{SSS:An identity between invariant forms}. 
To unify the notations, we extend $|\cdot|$ to $\C$ by $|x|:=\sqrt{x\b{x}}$. 
Then $|x|_F=|x|^d$ for $x\in F$, where $d:=[F:\R]$. 

Define the Gaussian function on $F$ by $\varphi_F(x)=e^{-d\pi |x|^2}$. More generally, if $t$ 
is a positive real number, then we set $\varphi_{F,t}(x)=\varphi_F(\frac{x}{\sqrt{t}})$. The 
Fourier transform $\wh{\varphi}_{F,t}$ of $\varphi_{F,t}$ is given by
\[
\wh{\varphi}_{F,t}(x)
=
\int_F
\varphi_{F,t}(y)\psi_F(xy)dy
=
t^{\frac{d}{2}}e^{-d\pi t|x|^2}.
\]
Observe that $\stt{\wh{\varphi}_{F,t}}_{t>0}$ is an approximate identity as $t\to\infty$. 

\begin{lm}\label{L:main estimate}
Let $0\leq\lambda<1$. Define for $t>0$, the function
\[
g_F(t)
=
\int_F |x-1|_F^{-\lambda}\wh{\varphi}_{F,t}(x)dx
=
\int_F|x-1|_F^{-\lambda}t^{\frac{d}{2}}e^{-d\pi t|x|^2}dx.
\]
Then we have $g_F\ll_{d,\lambda} 1$. 
\end{lm}

\begin{proof}
Fix $0<\epsilon<1$. We divide the integral into two parts
\[
g_F(t)
=
\int_{|x-1|_F>\epsilon}|x-1|_F^{-\lambda}\wh{\varphi}_{F,t}(x)dx
+
\int_{|x-1|_F\leq \epsilon}|x-1|_F^{-\lambda}\wh{\varphi}_{F,t}(x)dx.
\] 
On one hand, we have
\[
\int_{|x-1|_F>\epsilon}|x-1|_F^{-\lambda}\wh{\varphi}_{F,t}(x)dx
\leq 
\epsilon^{-\lambda}\int_F\wh{\varphi}_{F,t}(x)dx
=
\epsilon^{-\lambda}.
\]
On the other, 
\[
\int_{|x-1|_F\leq \epsilon}|x-1|_F^{-\lambda}\wh{\varphi}_{F,t}(x)dx
\leq
t^{\frac{d}{2}}e^{-d\pi(1-\epsilon)^2t}\int_{|x-1|_F\leq\epsilon}|x-1|_F^{-\lambda}dx
=
c_F\frac{\epsilon^{1-\lambda}}{1-\lambda}t^{\frac{d}{2}}e^{-d\pi(1-\epsilon)^2t}
\]
where $c_\R=2$ and $c_\C=2\pi$. Since the function $t\mapsto t^{\frac{d}{2}}e^{-d\pi(1-\epsilon)^2t}$ 
is bounded on $\R_{>0}$, the lemma follows.
\end{proof}
  
Let $n$ be a natural number. Define a function $\phi_{F,n}$ on $F^{\x}\backslash\G(F)$ by
\[
\phi_{F,n}(g)
=
e^{-\frac{d\pi}{n}\left(|r|+|r|^{-1}\right)}
\quad\text{if}\quad
g=zk\pMX{r}{0}{0}{1} h\in\G(F)
\]
where $k,h\in K$ and $z, r\in F^{\x}$. Next lemma shows that $\phi_{F,n}$ is well-defined.
 
\begin{lm}\label{L:function for regularization}
The function $\phi_{F,n}(g)$ is well-defined and bi-$K$-invariant. 
Moreover, we have $\lim_{n\to\infty}\phi_{F,n}(g)=1$ pointwisely on $F^{\x}\backslash\G(F)$ and
\[
\phi_{F,n}\left(\pMX{y}{x}{0}{1}\right)
=
e^{-\frac{d\pi}{n}(|y|+|y|^{-1})}
\varphi_{F,n|y|}(x),
\]
for $y\in F^{\x}$ and $x\in F$.
\end{lm}

\begin{proof}
For $g\in\G(F)$, let $g^*={}^t\b{g}$ be the conjugate transpose of $g$. Suppose that
\[
g=k_1\pMX{y_1}{0}{0}{y_2} k_2=k_3\pMX{y_3}{0}{0}{y_4} k_4
\]
where $y_j\in F^{\x}$ and $k_\ell\in K$ for $1\leq j,\ell\leq 4$. By considering the determinant
of $g$, we find that $y_1y_2=y_3y_4$. On the other hand, both $\stt{|y_1|^2, |y_2|^2}$ and 
$\stt{|y_3|^2, |y_4|^2}$ are the set of eigenvalues of $gg^*$, and hence they are equal. It follows
that
\[
e^{-\frac{d\pi}{n}(|y_1/y_2|+|y_1/y_2|^{-1})}
=
e^{-\frac{d\pi}{n}(|y_3/y_4|+|y_3/y_4|^{-1})}.
\] 
This shows that $\phi_{F,n}(g)$ is well-defined. It is clear from the definition that $\phi_{F,n}(g)$
is bi-$K$-invariant and that $\lim_{n\to\infty}\phi_{F,n}(g)=1$ pointwisely on 
$F^{\x}\backslash\G(F)$. To prove the last assertion, let $g=\pMX yx01=zk\pMX{r}{0}{0}{1} h$,
for some $z,r\in F^{\x}$ and $k,h\in K$. Then
\[
gg^*
=
\begin{pmatrix}
|x|^2+|y|^2&x\\\b{x}&1
\end{pmatrix}
=
|z|^2k
\begin{pmatrix}
|r|^2&0\\0&1
\end{pmatrix}
k^*.
\]
Taking trace and norm, we find that 
$|x|^2+|y|^2+1=|z|^2(|r|^2+1)$ and $|y|^2=|z|^4|r|^2$. From these equations we obtain 
\[
\frac{|x|^{2}}{|y|}+|y|+|y|^{-1}
=
|r|+|r|^{-1}.
\]
This proves the last assertion and hence concludes the proof.
\end{proof}

In the following, we define $a(y), n(x)\in\G(F)$ by
\[
a(y)=\pMX y001
\quad\text{and}\quad
n(x)=\pMX 1x01
\]
where $y\in F^{\x}$ and $x\in F$. 
Recall the local trilinear integral $\sI(W\ot f;\t{W}\ot\t{f})$ is defined by 
\eqref{E:trilinear local integral}. For each $n\in\mathbb{N}$, we put
\[
\sI_n(W\ot f;\t{W}\ot\t{f})
=
\int_{F^{\x}\backslash\G(F)}
\cB_{\pi_\cK}(\rho(g)W,\t{W})
\cB_{\pi}(\rho(g)f,\t{f})\phi_{F,n}(g)dg.
\]
By \lmref{L:function for regularization} and the dominated convergence theorem, one sees 
that 
\[
\lim_{n\to\infty}\sI_n(W\ot f;\t{W}\ot\t{f})=\sI(W\ot f;\t{W}\ot\t{f}).
\] 
To prove \propref{P:regularization}, it remains to show that
\begin{equation}\label{E:limit to RankinSelberg}
\lim_{n\to\infty}\sI_n(W\ot f;\t{W}\ot\t{f})
=
\frac{\zeta_\cK(1)}{\zeta_F(1)}
\cdot\sR_\delta(W\ot f)
\cdot\t{\sR}_\delta(\t{W}\ot\t{f}).
\end{equation}
The rest of this section is devoted to prove \eqref{E:limit to RankinSelberg}.
There are two cases depending on $\cK=F\x F$ or $\cK=\C$.

\subsection{Case $\cK=F\x F$}
In this case, we have $\pi_\cK=\pi_1\bt\pi_2$ and 
$\cW(\pi_\cK,\psi_\cK)=\cW(\pi_1,\psi)\ot\cW(\pi_2,\psi)$, where $\pi_1, \pi_2$ are unitary
irreducible admissible generic representations of $\G(F)$. Let $\cB_{\pi_j}$ be the 
$\G(F)$-invariant bilinear pairing between $\cW(\pi_j,\psi)$ and $\cW(\t{\pi}_j,\psi)$ defined
by \eqref{E:bilinear pairing for Whittaker} with $\cK^{\x}$ be replaced by $F^{\x}$ for $j=1,2$.
Then we have $\cB_{\pi_\cK}=\cB_{\pi_1}\ot\cB_{\pi_2}$.
Let $W_1\in\cW(\pi_1,\psi)$, $\t{W}_1\in\cW(\t{\pi}_1,\psi)$ and $g,h\in\G(F)$. 
Define
\[
\sA_n(\rho(g)W_1,\rho(h)\t{W}_1)
=
\int_F
\psi(x)\cB_{\pi_1}(\rho(n(x)g)W_1,\rho(h)\t{W}_1)\phi_{F,n}(h^{-1}n(x)g)dx.
\]
The following lemma is the core of the proof for the present case.  
\begin{lm}\label{L:main lemma for split}
We have 
\[
\lim_{n\to\infty}
\sA_n(\rho(a(y_1))W_1,\rho(a(y_2))\t{W}_1)
=
\zeta_F(1)W_1(a(-y_1))\t{W}_1(a(y_2)).
\]
\end{lm}

\begin{proof}
The proof is similar to that of the Fourier inversion formula. Put
\[
c_{F,n}(y_1,y_2)
=
\zeta_F(1)e^{-\frac{d\pi}{n}(|y_1/y_2|+|y_1/y_2|^{-1})}.
\]
Observe that ${\rm lim}_{n\to\infty}c_{F,n}(y_1,y_2)=\zeta_F(1)$.
By \lmref{L:function for regularization} and the choice of the measure on $F^{\x}$, 
we find that formally
\begin{align}\label{E:computation for A_n split}
\begin{split}
\sA_n(&\rho(a(y_1))W_1,\rho(a(y_2))\t{W}_1)\\
&=
c_{F,n}(y_1,y_2)
\int_F\int_{F^{\x}}\psi((t+1)x)W_1(a(y_1 t))\t{W}_1(a(-y_2 t))\varphi_{F,n|y_1y_2|}(x)|t|_F^{-1}dtdx\\
&=
c_{F,n}(y_1,y_2)
\int_{F}
W_1(a(y_1(-1+t)))\t{W}_1(a(-y_2(-1+t)))|-1+t|_F^{-1}\wh{\varphi}_{F,n|y_1y_2|}(t)dt.
\end{split}
\end{align}
To justify the computation, we use \eqref{E:estimate Whittaker function} to obtain
\begin{align*}\label{E:estimate of A_n split}
\begin{split}
\int_F\int_{F^{\x}}&\psi((t+1)x)W_1(a(y_1 t))\t{W}_1(a(-y_2 t))\varphi_{F,n|y_1y_2|}(x)|t|_F^{-1}dtdx\\
&\ll_{\lambda(\pi_1),W_1,\t{W}_1,\epsilon}
\int_F\int_{F^{\x}}
|y_1y_2|_F^{1/2-\lambda(\pi_1)-\epsilon}|t|_F^{-2\lambda(\pi_1)-2\epsilon}
\Phi(y_1t)\t{\Phi}(y_2t)\varphi_{F,n|y_1y_2|}(x)dtdx.
\end{split}
\end{align*}
Since we can choose $\epsilon$ so that $-2\lambda(\pi_1)-2\epsilon>-1$, the last integral converges absolutely. 
The lemma follows from \eqref{E:computation for A_n split} and the fact that 
$\stt{\wh{\varphi}_{F,n|y_1y_2|}}_{n\in\mathbb{N}}$ is an approximate identity as $n\to\infty$.
\end{proof}

We first compute the integral $\sI_n(W\ot f;\t{W}\ot\t{f})$ formally. 
Without loss of generality, we may assume that
$W=W_1\ot W_2$ and $\t{W}=\t{W}_1\ot\t{W}_2$, where $W_j\in\cW(\pi_j,\psi)$ and 
$\t{W}_j\in\cW(\t{\pi}_j,\psi)$ for $j=1,2$. To proceed with the computation,
we follow the trick in \cite{MV2010}. Put
\[
F(h)=W_2(h)f(h)
\quad\text{and}\quad
\t{F}(h)=\t{W}_2\left(J_2 h\right)\t{f}(h).
\] 
Define
\[
\cB(F,\t{F})
=
\int_{F^{\x}\N(F)\backslash\G(F)}F(h)\t{F}(h)dh.
\]
This integral converges absolutely and one checks easily that (\cite[Lemma 3.2.7]{MV2010})
\[
\cB(F,\t{F})
=
\cB_{\pi_2}(W_2,\t{W}_2)\cB_{\pi}(f,\t{f}).
\]
It follows that
\begin{align*}
\sI_n(W\ot f;\t{W}\ot\t{f})
&=
\int_{F^{\x}\backslash\G(F)}
\cB_{\pi_1}(\rho(g)W_1,\t{W}_1)\cB(\rho(g)F,\t{F})\phi_{F,n}(g)dg\\
&=
\int_{F^{\x}\backslash\G(F)}\int_{F^{\x}\N(F)\backslash\G(F)}
\cB_{\pi_1}(\rho(g)W_1,\t{W}_1)F(hg)\t{F}(h)\phi_n(g)dhdg\\
&=
\int_{F^{\x}\N(F)\backslash\G(F)}\int_{F^{\x}\backslash\G(F)}
\cB_{\pi_1}(\rho(g)W_1,\rho(h)\t{W}_1)F(g)\t{F}(h)\phi_{F,n}(h^{-1}g)dgdh\\
&=
\int_{F^{\x}\N(F)\backslash\G(F)}\int_{F^{\x}\N(F)\backslash\G(F)}
\sA_n(\rho(g)W_1,\rho(h)\t{W}_1)F(g)\t{F}(h)dgdh.
\end{align*}
To justify the computation, it remains to show that the integral
\[
\int_{F^{\x}\N(F)\backslash\G(F)}\int_{F^{\x}\N(F)\backslash\G(F)}
\sA_n(\rho(g)W_1,\rho(h)\t{W}_1)F(g)\t{F}(h)dgdh
\]
is absolutely convergent. By the decomposition \eqref{E:decomposition of rankin-selberg integral},
the integral becomes
\begin{align*}\label{E:expansion for split}
\begin{split}
\int_K\int_K f(k_1)\t{f}(k_2)\int_{F^{\x}}\int_{F^{\x}}
\sA_n(&\rho(a(y_1)k_1)W_1,\rho(a(y_2)k_2)\t{W}_1)\\
&\x W_2(a(y_1)k_1)\t{W}_2(a(-y_2)k_2)\mu(y_1y^{-1}_2)|y_1y_2|_F^{-\frac{1}{2}}
d^{\x}y_1d^{\x}y_2dk_1dk_2.
\end{split}
\end{align*}
By the right $K$-finiteness and the fact that $K$ is compact, it suffices to show that the integral
\begin{equation}\label{E:main integrand for split}
\int_{F^{\x}}\int_{F^{\x}}
\sA_n(\rho(a(y_1))W_1,\rho(a(y_2))\t{W}_1)
W_2(a(y_1))\t{W}_2(a(-y_2))\mu(y_1y^{-1}_2)|y_1y_2|_F^{-\frac{1}{2}}
d^{\x}y_1d^{\x}y_2
\end{equation}
is absolutely convergent. Let $\lambda_j=\lambda(\pi_j)$ for $j=1,2$, and let 
$|\mu(y)|=|y|_F^\lambda$ for some $\lambda\in\R$. Note that we have $\lambda(\pi)=|\lambda|$ and 
$\Lambda(\itPi)=\lambda_1+\lambda_2+|\lambda|<1/2$. 
By \eqref{E:computation for A_n split} and \lmref{L:main estimate}, 
\eqref{E:estimate Whittaker function}, we find that 
\begin{align*}
\sA_n(\rho(a(y_1))W_1,\rho(a(y_2))\t{W}_1))
&=
c_{F,n}(y_1,y_2)
\int_{F}
W_1(a(y_1(-1+t)))\t{W}_1(-y_2(-1+t))|-1+t|_F^{-1}\wh{\varphi}_{F,n|y_1y_2|}(t)dt\\
&\ll_{\lambda_1,W_1,\t{W}_1,\epsilon}
|y_1y_2|_F^{\frac{1}{2}-\lambda_1-\epsilon}
\int_F|t-1|^{-2\lambda_1-2\epsilon}_F\wh{\varphi}_{F,n|y_1y_2|}(t)dt\\
&\ll_{\lambda_1,W_1,\t{W}_1,d,\epsilon}
|y_1y_2|_F^{\frac{1}{2}-\lambda_1-\epsilon}.
\end{align*}
Using \eqref{E:estimate Whittaker function} again we see that 
\begin{align*}
\sA_n(\rho(a(y_1))W_1,\rho(a(y_2))\t{W}_1)
&W_2(a(y_1))\t{W}_2(a(-y_2))\mu(y_1y^{-1}_2)|y_1y_2|_F^{-\frac{1}{2}}\\
&\ll_{\lambda_1,\lambda_2,W,\t{W},d,\epsilon}
|y_1|_F^{\frac{1}{2}-\lambda_1-\lambda_2+\lambda-2\epsilon}
|y_2|_F^{\frac{1}{2}-\lambda_1-\lambda_2-\lambda-2\epsilon}
\Phi(y_1)\t{\Phi}(y_2)
\end{align*}
for some Schwartz functions $\Phi$ and $\t{\Phi}$. This shows that the integral 
\eqref{E:main integrand for split} converges absolutely. Moreover, the integrand of 
\eqref{E:main integrand for split} is bounded by an integrable function on $F^{\x}\x F^{\x}$, 
which is independent of $n$. By the right $K$-finiteness and the dominated convergence
theorem, we find that 
\begin{align*}
\begin{split}
\lim_{n\to\infty}\sI_n(W\ot f;\t{W}\ot\t{f})
=
\int_K&\int_K f(k_1)\t{f}(k_2)\int_{F^{\x}}\int_{F^{\x}}
\lim_{n\to\infty}\sA_n(\rho(a(y_1)k_1)W_1,\rho(a(y_2)k_2)\t{W}_1)\\
&\x W_2(a(y_1)k_1)\t{W}_2(a(-y_2)k_2)\mu(y_1y^{-1}_2)|y_1y_2|_F^{-\frac{1}{2}}
d^{\x}y_1d^{\x}y_2dk_1dk_2,
\end{split}
\end{align*}
which by \lmref{L:main lemma for split} is equal to
\begin{align*}
\zeta_F(1)\int_K\int_K f(k_1)\t{f}(k_2)\int_{F^{\x}}\int_{F^{\x}}
&W_1(a(-y_1)k_1)\t{W}_1(a(y_2)k_2)\\
&\x W_2(a(y_1)k_1)\t{W}_2(a(-y_2)k_2)\mu(y_1y^{-1}_2)|y_1y_2|_F^{-\frac{1}{2}}
d^{\x}y_1d^{\x}y_2dk_1dk_2.
\end{align*}
This is precisely the RHS of \eqref{E:limit to RankinSelberg}. This completes the proof for the 
case where $\cK=F\x F$.

\subsection{Case $\cK=\C$}
In this subsection, we prove \eqref{E:limit to RankinSelberg} for the case where $\cK=\C$.
Let $W\in\cW(\pi_\C,\psi_\C)$, $\t{W}\in\cW(\t{\pi}_\C,\psi_\C)$ and $g,h\in\G(\R)$.
Define
\[
\sA_n(\rho(g)W,\rho(h)\t{W})
=
\int_\R
\cB_{\pi_\C}(\rho(n(x)g)W,\rho(h)\t{W})\phi_{\R,n}(h^{-1}n(x)g) dx.
\]
As in the proof of the case where $\cK=F\x F$, we investigate the behaviour of 
$\sA_n(\rho(g)W,\rho(h)\t{W})$ when $n\to\infty$.

\begin{lm}\label{L:main lemma for nonsplit}
We have
\[
\lim_{n\to\infty}
\sA_n(\rho(a(y))W,\t{W})
=
\zeta_\C(1)\int_{\R^{\x}}
W(a(iyv))\t{W}(a(-iv))\frac{d^{\x}v}{|v|}.
\]
\end{lm}

\begin{proof}
The proof is similar to that of the partial Fourier inversion formula. 
Put
\[
c_n(y)
=
\zeta_\C(1)e^{-\frac{\pi}{n}(|y|+|y|^{-1})}.
\]
Note that ${\rm lim}_{n\to\infty}c_n(y)=\zeta_\C(1)$. 
We compute formally. By \lmref{L:function for regularization} and the choice of the measure on 
$\C^{\x}$, we have
\begin{align}\label{E:A_n for nonsplit}
\begin{split}
\sA_n(\rho(a(y))W,\t{W})
&=
c_n(y)
\int_{\R}\int_{\C^{\x}}
\psi_\C(xz)W(a(yz))\t{W}(a(-z))\varphi_{\R,n|y|}(x)|z|_\C^{-1}dzdx\\
&=
c_n(y)
\int_{\R}\int_{\R}\int_{\R}
\psi_\R(2ux)W(a(y(u+iv)))\t{W}(a(-u-iv))\varphi_{\R,n|y|}(x)\frac{2dudv}{(u^2+v^2)}dx\\
&=
c_n(y)
\int_\R\int_\R
W(a(y(u+iv)))\t{W}(a(-u-iv))\wh{\varphi}_{\R,4n|y|}(u)\frac{dudv}{(u^2+v^2)}.
\end{split}
\end{align} 
We write $z=u+iv$ for $u,v\in\R$ in the second line of \eqref{E:A_n for nonsplit}.
To justify the computation, we use \eqref{E:estimate Whittaker function} to find that
\begin{align*}
\begin{split}
\int_{\R}\int_{\C^{\x}}
&\psi_\C(xz)W(a(yz))\t{W}(a(-z))\varphi_{\R,n|y|}(x) |z|_\C^{-1} dzdx\\
&\ll_{\lambda',W,\t{W},\epsilon}
|y|^{1-2\lambda'-2\epsilon}
\int_\R\int_{\C^{\x}}
|z|_\C^{-2\lambda'-2\epsilon}\Phi(yz)\t{\Phi}(z)\varphi_{\R,n|y|}(x)dzdx. 
\end{split}
\end{align*}
Here $\lambda'=\lambda(\pi_\C)$ and $\Phi, \t{\Phi}$ are Schwartz functions on $\C$. 
Since we can choose $\epsilon$ so that $-2\lambda'-2\epsilon>-1$, the last integral is absolutely convergent.

To complete the proof, we show that
\begin{equation}\label{E:change limit and integral nonsplit}
\lim_{n\to\infty}
\sA_n(\rho(a(y)W),\t{W})
=
\int_\R\lim_{n\to\infty}\int_\R
W(a(y(u+iv)))\t{W}(a(-u-iv))\wh{\varphi}_{\R,4n|y|}(u)\frac{dudv}{(u^2+v^2)}.
\end{equation}
Then the assertion follows immediately from the fact that 
$\stt{\wh{\varphi}_{\R,n|y|}}_{4n\in\mathbb{N}}$ is an approximate identity as $n\to\infty$.
Put 
\begin{equation}\label{E:f_n}
f_n(y;v)
=
\int_\R
W(a(y(u+iv)))\t{W}(a(-u-iv))\wh{\varphi}_{\R,4n|y|}(u)\frac{du}{(u^2+v^2)}.
\end{equation}
We show that as a function in $v$, $f_n(y;v)$ is bounded by an integrable function on $\R$, 
which is independent of $n$. By \eqref{E:estimate Whittaker function}, there exist $Schwartz$
$functions$ $\varphi, \t{\varphi}$ $on$ $\R$ so that
\begin{equation}\label{E:estimate Whittaker for nonsplit}
W(a(y(u+iv)))\t{W}(a(-u-iv))
\ll_{\lambda',W,\t{W},\epsilon}
|y|^{1-2\lambda'-2\epsilon}(u^2+v^2)^{1-2\lambda'-2\epsilon}\varphi(yv)\t{\varphi}(v).
\end{equation}
As we will see, the function $\varphi$ is superfluous in the proof here; however, it will be 
used later. From \eqref{E:estimate Whittaker for nonsplit}, one sees immediately that
\begin{align}\label{E:estimate f_n}
\begin{split}
f_n(y;v)
&\ll_{\lambda',W,\t{W},\epsilon}
|y|^{1-2\lambda'-2\epsilon}\varphi(yv)\t{\varphi}(v)
\int_\R(u^2+v^2)^{-2\lambda'-2\epsilon}\wh{\varphi}_{\R,4n|y|}(u)du\\
&\ll_{\lambda',W,\t{W},\epsilon}
|y|^{1-2\lambda'-2\epsilon}|v^2|^{-2\lambda'-2\epsilon}\varphi(yv)\t{\varphi}(v).
\end{split}
\end{align}
Since $\Lambda(\itPi)=2\lambda'+\lambda(\pi)<1/2$, we find that $-4\lambda'-4\epsilon>-1$. 
It follows that $f_n(y;v)$ as a function of $v$ is bounded by an integrable function on $\R$, 
which is independent of $n$.
Therefore \eqref{E:change limit and integral nonsplit} follows from the dominated convergence 
theorem.
\end{proof}

As in the previous case, we first compute the integral $\sI_n(W\ot f;\t{W}\ot\t{f})$ formally.
We have
\begin{align*}
\sI_n(W\ot f;\t{W}\ot\t{f})
&=
\int_{\R^{\x}\backslash\G(\R)}\int_{{\rm SO}(2)}
\cB_{\pi_\C}(\rho(g)W,\t{W}) f(hg)\t{f}(h)\phi_{\R,n}(g)dhdg\\
&=
\int_{{\rm SO}(2)}\int_{\R^{\x}\backslash\G(\R)}
\cB_{\pi_\C}(\rho(g)W,\rho(h)\t{W}) f(g)\t{f}(h)\phi_{\R,n}(h^{-1}g)dgdh\\
&=
\int_{{\rm SO}(2)}\int_{{\rm SO}(2)}f(k)\t{f}(h)
\int_{\R^{\x}}\sA_n(\rho(a(y)k)W,\rho(h)\t{W})\mu(y)|y|^{-\frac{1}{2}}d^{\x}ydkdh.
\end{align*}
It remains to justify the computation. By the right ${\rm SO}(2)$-finiteness, we only need to show
that the integral
\begin{equation}\label{E:integrand for A_n nonsplit}
\int_{\R^{\x}}\sA_n(\rho(a(y))W,\t{W})\mu(y)|y|^{-\frac{1}{2}}d^{\x}y
\end{equation}
is absolutely convergent. Let $|\mu(y)|=|y|^\lambda$ for some $\lambda\in\R$. Note that 
$\lambda(\pi)=|\lambda|$. Recall that $\lambda'=\lambda(\pi_\C)$ and that 
$\Lambda(\itPi)=2\lambda'+|\lambda|<1/2$. Let $f_n(y;v)$ be the function given by \eqref{E:f_n}.
Then by \eqref{E:A_n for nonsplit} and \eqref{E:estimate f_n}, we find that 
\begin{align*}
\sA_n(\rho(a(y))W,\t{W})\mu(y)|y|^{-\frac{1}{2}}
&=
c_n(y)\mu(y)|y|^{-\frac{1}{2}}\int_\R f_n(y;v)dv\\
&\ll_{\lambda',\lambda,W,\t{W},\epsilon}
|y|^{\frac{1}{2}-2\lambda'+\lambda-2\epsilon}\int_\R
|v|^{-4\lambda'-4\epsilon}\varphi(yv)\t{\varphi}(v)dv.
\end{align*}
Note that the integrand of \eqref{E:integrand for A_n nonsplit} is bounded by an integrable 
function in $y\in\R^{\x}$, which is independent of $n$. Indeed, we have
\[
\int_{\R^{\x}}|y|^{\frac{1}{2}-2\lambda'+\lambda-2\epsilon}\int_\R
|v^2|^{-2\lambda'-2\epsilon}\varphi(yv)\t{\varphi}(v)
dvd^{\x}y
=
\int_{\R^{\x}}\int_{\R^{\x}}
|y|^{\frac{1}{2}-2\lambda'+\lambda-2\epsilon}
|v|^{\frac{1}{2}-2\lambda'-\lambda-2\epsilon}\varphi(y)\t{\varphi}(v)
d^{\x}vd^{\x}y
\] 
after changing the variables. Applying the dominated convergent theorem and the right 
${\rm SO}(2)$-finiteness, we can exchange the limit and the integrals and we obtain
\[
\lim_{n\to\infty}
\sI_n(W\ot f;\t{W}\ot\t{f})
=
\int_{{\rm SO}(2)}\int_{{\rm SO}(2)}f(k)\t{f}(h)\int_{\R^{\x}}\lim_{n\to\infty}
\sA_n(\rho(a(y)k)W,\rho(h)\t{W})\mu(y)|y|^{-\frac{1}{2}}d^{\x}ydkdh,
\]
which by \lmref{L:main lemma for nonsplit} is equal to
\[
\zeta_\C(1)\int_{{\rm SO}(2)}\int_{{\rm SO}(2)}f(k)\t{f}(h)\int_{\R^{\x}}\int_{\R^\x}
W(a(iy)k)\t{W}(a(-iv)h)\mu(yv^{-1})|yv^{-1}|^{-\frac{1}{2}}d^{\x}yd^{\x}vdkdh.
\]
This is precisely the RHS of \eqref{E:limit to RankinSelberg} and hence the case where $\cK=\C$ is proved. 

This finishes the proof of \propref{P:regularization} for the archimedean case.

\appendix
\section{Whittaker functions of $\G$ over archimedean local fields}
\label{S:appendix}
The purpose of this appendix is to prove \propref{P:test vector in Whittaker model} and the equation
\eqref{E:estimate Whittaker function} for the archimedean case. 
So we let $F=\R$ or $\C$ in this appendix. Let $\psi$ be a non-trivial additive character of $F$ and 
$\pi$ be an irreducible admissible $generic$ representation of $\G(F)$ with the minimal weight $k\geq 0$. 
We may assume $\pi$ to be a constituent of $\rho(\mu,\nu)$ for some characters $\mu, \nu$.

 \subsection{Whittaker functional}
Let $(r_\psi,\cS(F^2))$ be the Weil representation of $\G(F)$ defined in 
\cite[Section 1]{JLbook} and let $\cS(F^2,\psi)$ be the subspace defined by \eqref{E:algebraic Schwartz space}.
For $\Phi\in\cS(F^2)$, we define its partial Fourier transform by 
\begin{equation}\label{E:partial Fourier transform}
\Phi^\sim(x,y)
=
\int_F\Phi(x,u)\psi(yu)du.
\end{equation}
Here $du$ is self-dual with respect to $\psi$. 
Note that $\cS(F^2,\psi)$ is invariant under the
partial Fourier transform and we have $\left(r_\psi(g)\Phi\right)^\sim=\rho(g)\Phi^\sim$
(\cite[Proposition 1.6]{JLbook}).
For $\Phi\in\cS(F^2)$ and $g\in\G(F)$, we define 
\begin{equation}\label{E:integral representation for Whittaker functions}
W_\Phi(g)
=
\mu({\rm det}(g))|{\rm det}(g)|_F^{\frac{1}{2}}
\int_{F^{\x}}
r_\psi(g)\Phi(t,t^{-1})\mu\nu^{-1}(t)d^{\x}t.
\end{equation}
Then we have $W_\Phi\in\cW(\psi)$ in the notation of \S\ref{SS:Whittaker model}.
Following \cite{JLbook}, we set  $\cW(\mu,\nu ; \psi)=\stt{W_\Phi\mid \Phi\in\cS(F^2,\psi)}$.
Then $\cW(\mu,\nu ; \psi)=\cW(\nu,\mu ; \psi)$ by \cite[Proposition 3.4]{JLbook}. Moreover,  
by \cite[Lemmas 5.13.1, 6.3.1]{JLbook}, if $|\mu\nu^{-1}(y)|=|y|_F^r$ for some $r>-1$, 
then $L(f_{\Phi^\sim}):=W_{\Phi}$
\footnote{Our $L$ is the $inverse$ of $A$ in \cite[Lemmas 5.13.1, 6.3.1]{JLbook}.}
defines an $\G(F)$-isomorphism 
from  $\rho(\mu,\nu)$  onto $\cW(\mu, \nu ; \psi)$, where $f_\Phi=f^{(s)}_\Phi|_{s=1/2}$ is the Godement section 
defined in \S\ref{SSS:Godement section and intertwining operator}. When $r\leq -1$, 
the map $L$ still gives rise to an $\G(F)$-isomorphism when $\rho(\mu,\nu)$ is irreducible, 
but now one needs to use the analytic continuation of the Godement sections. 
In particular, $\cW(\mu,\nu ; \psi)=\cW(\pi,\psi)$ if $\pi\cong\pi(\mu,\nu)$ is a principal series representation. 
On the other hand, if $\pi$ is a discrete series representation of $\G(\R)$ (so that $k\geq 2$), then
$\cW(\pi,\psi)\subset\cW(\mu,\nu ; \psi)$ and $W_\Phi\in\cW(\pi,\psi)$ if and only if 
\[
\int_\R x^n\frac{\partial^m}{\partial y}\Phi(x,0)dx=0
\]
for every non-negative integers $n,m$ with $m+n=k-2$ (\cite[Corollary 5.14]{JLbook}).

In the following, we assume that $\psi$ is given by \eqref{E:additive character of archimedean}.
We also fix the choices of Haar measures $dt$ and $d^{\x}t$ as follows.
Let $dt$ be the usual Lebesgue measure on $\R$ and be $twice$ of the 
usual Lebesgue measure on $\C$. On $F^{\x}$, we take $d^{\x}t=\zeta_F(1)|t|_F^{-1}dt$.
Recall that 
\[
\cS(\R^2,\psi)
=
\bigoplus_{a,b\geq 0}\C\,x_1^a x_2^b e^{-\pi(x_1^2+x_2^2)}
\quad\text{and}\quad
\cS(\C^2,\psi)
=
\bigoplus_{a,b,c,d\geq 0}\C\,x_1^a\b{x}_1^bx_2^c\b{x}_2^d e^{-2\pi(x_1\b{x}_1+x_2\b{x}_2)}.
\]
The following lemma is the key of this appendix.

\begin{lm}\label{L:kirillov model for archimedean}
Let notations be as above.
\begin{itemize}
\item[(1)]
Suppose that $F=\R$ and $\mu\nu^{-1}=|\cdot|^{2s}{\rm sgn}^k$ for some $s\in\C$ and $k\in\Z$. 
Let $\Phi(x_1,x_2)=x_1^ax_2^be^{-\pi(x_1^2+x_2^2)}$ for some integers $a,b\geq 0$. Then
\[
W_\Phi\left(\pMX y001\right)
=
\begin{cases}
2\mu(y)|y|^{\frac{1}{2}-s-\frac{a-b}{2}}y^a K_{s+\frac{a-b}{2}}(2\pi|y|)
\quad&\text{if $a-b\equiv k\pmod 2$},\\
0
\quad&\text{otherwise}.
\end{cases}
\]
\item[(2)]
Suppose that $F=\C$ and $\mu\nu^{-1}(y)=|y|_\C^{2s}\left(y/|y|_\C^{\frac{1}{2}}\right)^k$ for some $s\in\C$ and $k\in\Z$. 
Let $\Phi(x_1,x_2)=x_1^a\b{x}_1^bx_2^c\b{x}_2^de^{-2\pi(x_1\b{x}_1+x_2\b{x}_2)}$ for some 
integers $a,b,c,d\geq 0$. Then
\[
W_\Phi\left(\pMX y001\right)
=
\begin{cases}
4\mu(y)|y|_\C^{\frac{1}{2}-s-\frac{a+b-c-d}{4}}y^a\b{y}^b K_{2s+\frac{a+b-c-d}{2}}\left(4\pi|y|_\C^{\frac{1}{2}}\right)
\quad&\text{if $a-b-c+d+k=0$},\\
0\quad&\text{otherwise}.
\end{cases}
\]
\end{itemize}
\end{lm}

\begin{proof}
Let $\alpha,\beta$ be positive real numbers and $s\in\C$. We first show that 
\begin{equation}\label{E:integral to K-bessel function}
\int_0^\infty e^{-\alpha(\beta r^2+r^{-2})}r^{s-1} dr
=
\beta^{-\frac{s}{4}}K_{\frac{s}{2}}\left(2\alpha\beta^{\frac{1}{2}}\right).
\end{equation}
Here $dr$ is the usual Lebesgue measure on $\R_{>0}$ and 
$\beta^{1/2}$ and $\beta^{1/4}$ are positive real numbers.
By changing the variable $r^2\mapsto r$, we find that
\[
\int_0^\infty e^{-\alpha(\beta r^2+r^{-2})}r^{s-1} dr
=
\frac{1}{2}\int_0^\infty e^{-\alpha(\beta r+r^{-1})}r^{\frac{s}{2}-1}dr
=
\frac{1}{2}\int_0^\infty 
e^{-\alpha\beta^{\frac{1}{2}}\left(\beta^{\frac{1}{2}} r+\beta^{-\frac{1}{2}} r^{-1}\right)}r^{\frac{s}{2}-1}dr.
\]
\eqref{E:integral to K-bessel function} now follows from changing the variable 
$\beta^{\frac{1}{2}}r\mapsto r$ again and  \eqref{E:K-bessel function}.

Recall that for $y\in F^{\x}$, we have 
\[
r_\psi\left(\pMX y001\right)\Phi(x_1,x_2)=\Phi(yx_1,x_2).
\]
It follows that $\Phi(yt,t^{-1})$ is equal to
\[
y^a t^{a-b}e^{-\pi(y^2 t^2+ t^{-2})}
\quad\text{or}\quad
y^a\b{y}^b t^{a-c}\b{t}^{b-d}e^{-2\pi(|yt|_\C+|t|_\C^{-1})}
\]
according to $F=\R$ or $\C$. We only compute the case $F=\C$ as the other case is similar (and easier).
Note that $d^{\x}t=2\pi^{-1}d^{\x}rd\theta$ if we use the polar coordinate $t=re^{i\theta}$. 
Here $d^{\x}r=\frac{dr}{r}$ and $dr$ is the usual Lebesgue measure on $\R_{>0}$. 
Let $\beta=|y|_\C$. By definition, we have
\begin{align*}
W_\Phi\left(\pMX y001\right)
&=
\mu(y)|y|^{\frac{1}{2}}\int_{\C^{\x}} \Phi(yt,t^{-1})\mu\nu^{-1}(t)d^{\x}t\\
&=
\mu(y)|y|_\C^{\frac{1}{2}}y^a\b{y}^b
\int_0^{2\pi} e^{i(a-b-c+d+k)\theta}d\theta
\int_0^\infty e^{-2\pi(\beta r^2+r^{-2})}r^{4s+a+b-c-d-1}2\pi^{-1}dr.
\end{align*}
Now the assertion follows from \eqref{E:integral to K-bessel function}. This completes the proof.
\end{proof}

\subsection{Explicit formulae}\label{A:explicit formulae}
Now we prove \propref{P:test vector in Whittaker model}. We have two cases depending on $F=\R$ or $\C$.

\subsubsection{$F=\R$}
Suppose that $F=\R$. As we mentioned, the proof of the first case is given by \cite[Lemma 3.3]{CCI2019}  
so that we assume in the following that $\pi\cong\pi(\mu,\nu)$ is a principal series representation. 
Let $n$ be a non-negative integer and $\Phi_n\in\cS(\R^2,\psi)$ be the element such that
\[
\Phi_n(x_1,x_2)^\sim
=
\left(x_1+ix_2\right)^ne^{-\pi(x_1^2+x_2^2)}.
\] 
It follows from the relation $\left(r_\psi(g)\Phi\right)^\sim=\rho(g)\Phi^\sim$ that
$r_\psi(k(\theta))\Phi_n(x_1,x_2)=e^{in\theta}\Phi_n(x_1,x_2)$. 
Suppose that $n\equiv k\pmod 2$.  We first show that $W_{\Phi_n}$ defines a non-zero weight $n$ element in
$\cW(\pi,\psi)$, i.e. $W_{\Phi_n}$ is a non-zero generator of $\cV_\psi(\pi,n)$. Indeed, we have 
\[
\rho(k(\theta))W_{\Phi_n}=W_{r_\psi(k(\theta))\Phi_n}=e^{i k\theta}W_{\Phi_n}
\]
by the above observation. It remains to show that $W_{\Phi_n}$ is non-zero. 
Let $\mu\nu^{-1}=|\cdot|^{2s}{\rm sgn}^k$ for some $s\in\C$. 
Since $\pi$ is a principal series representation, the map $L$ is an $\G(F)$-isomorphism and we have
$L(f_{\Phi_n^\sim})=W_{\Phi_n}$. To show that $W_{\Phi_n}\neq 0$, it suffices to show that $f_{\Phi_n^\sim}\neq 0$.
However, one checks easily that $f_{\Phi_n^\sim}(I_2)=(\sqrt{-1})^n\zeta_\R(2s+n+1)\neq 0$.
We now compute $\Phi_k$ and $\Phi_2$. By easy computations, we find that
\[
\Phi_k(x_1,x_2)
=
(x_1+x_2)^k e^{-\pi(x_1^2+x_2^2)}
\quad
\text{and}
\quad
\Phi_2(x_1,x_2)
=
(x_1^2+2x_1x_2+x_2^2-1/2\pi) e^{-\pi(x_1^2+x_2^2)}.
\]
It then follows from \lmref{L:kirillov model for archimedean} $(1)$ that
$W_\pi$ in \propref{P:test vector in Whittaker model} $(2)$ is equal to $2^{-1} W_{\Phi_k}$.
On the other hand, if $k=0$, then we have
\[
\rho(V_+)W_\pi=2^{-1}\rho(V_+)W_{\Phi_0}=2^{-1} W_{r_\psi(V_+)\Phi_0}
\]
and by a direct calculation that $r_\psi(V_+)\Phi_0=-2\pi\Phi_2$. 
This shows \propref{P:test vector in Whittaker model} $(2)$.

\subsubsection{$F=\C$}
Suppose that $F=\C$ and $\pi\cong\pi(\mu,\nu)$ is a principal series representation.
Since $\pi(\mu,\nu)\cong\pi(\nu,\mu)$, we may assume that 
$\mu\nu^{-1}(e^{i\theta})=e^{ik\theta}$ by the uniqueness of the Whittaker model. 
Let $n\geq k$ be an integer with the same parity. 
We describe how to construct a non-zero element in $\cV_\psi(\pi,n)$ following \cite[Section 18]{JLbook2}.
The construction is similar to the case when $F=\R$.
Write $n=k+2m$ for some non-negative integer $m$ and put
\[
\vec{\Psi}_{k,n}(x_1,x_2)
=
(x_2X-x_1Y)^m(\b{x}_1X+\b{x}_2Y)^{k+m} e^{-2\pi(x_1\b{x}_1+x_2\b{x}_2)}\in\cS(\C^2,\psi)\ot\cL_n(\C).
\]
One checks easily that 
\begin{equation}\label{E:SU(2)-equivalent for Schwartz function}
\vec{\Psi}_{k,n}((x_1,x_2)u)=\rho_n(u^{-1})\vec{\Psi}_{k,n}((x_1,x_2))
\end{equation}
for every $u\in{\rm SU}(2)$. 
The partial Fourier transform \eqref{E:partial Fourier transform} can be extended to 
$\cS(\C^2,\psi)\ot\cL_n(\C)$ by performing the transform on the coefficients. 
Let $\vec{\Phi}_{k,n}\in\cS(\C^2,\psi)\ot\cL_n(\C)$ so that $\vec{\Phi}_{k,n}^\sim=\vec{\Psi}_{k,n}$.
The Weil representation $(r_\psi,\cS(\C^2))$ of $\G(\C)$ can be extended to the space $\cS(\C^2)\ot\cL_n(\C)$ in a 
similar way as well as the integration \eqref{E:integral representation for Whittaker functions} and the Godement section
defined in \S\ref{SSS:Godement section and intertwining operator}. By \eqref{E:SU(2)-equivalent for Schwartz function},
we see that 
\[
f_{\vec{\Psi}_{k,n}}:=f^{(s)}_{\vec{\Psi}_{k,n}}|_{s=1/2}\in\left(\cB(\mu,\nu)\ot\cL_n(\C)\right)^{{\rm SU}(2)}.
\]
Moreover, $f_{\vec{\Psi}_{k,n}}(I_2)\neq 0$ by \cite[Lemma 18.4]{JLbook2} and the assumption that 
$\pi\cong\pi(\mu,\nu)$ is a principal series representation. 
It follows that $W_{\vec{\Phi}_{k,n}}:=L(f_{\vec{\Psi}_{k,n}})$ defines a non-zero
element in $\cV_\psi(\pi,n)$. More precisely, $W_{\vec{\Phi}_{k,n}}$ is given by the integral 
\eqref{E:integral representation for Whittaker functions} with $\Phi$ replaced by $\vec{\Phi}_{k,n}$.
To prove \propref{P:test vector in Whittaker model} $(3)$, it remains to compute $\vec{\Phi}_{k,k}$ and $\vec{\Phi}_{0,2}$.
By routine calculations, we find that
\begin{equation}\label{E:vector value Schwartz function}
\vec{\Phi}_{k,k}(x_1,x_2)
=
\sum_{j=0}^k
\begin{pmatrix}
k\\j
\end{pmatrix}
\left(-\sqrt{-1}\right)^{k-j}
\b{x}_1^j x_2^{k-j}e^{-2\pi(x_1\b{x}_1+x_2\b{x}_2)}X^j Y^{k-j}
\end{equation}
and
\begin{align}\label{E:vector value Schwartz function 2}
\begin{split}
\vec{\Phi}_{0,2}(x_1,x_2)
=
&\left(-\sqrt{-1}\right)\b{x}_1\b{x}_2e^{-2\pi(x_1\b{x}_1+x_2\b{x}_2)}X^2
+
\left(\frac{1}{2\pi}-x_1\b{x}_1-x_2\b{x}_2\right)e^{-2\pi(x_1\b{x}_1+x_2\b{x}_2)}XY\\
&+\left(\sqrt{-1}\right)x_1x_2 e^{-2\pi(x_1\b{x}_1+x_2\b{x}_2)}Y^2.
\end{split}
\end{align}
It then follows from \lmref{L:kirillov model for archimedean} $(2)$ that 
$\vec{W}_\pi=2^{-2}\mu\nu(-1)\left(\sqrt{-1}\right)^k W_{\vec{\Phi}_{k,k}}$.
This shows \propref{P:test vector in Whittaker model} $(3)$.
Similarly when $k=0$, we have $\vec{W}_\pi^{(2)}=2^{-2}W_{\vec{\Phi}_{0,2}}$. This 
completes the proof of \propref{P:test vector in Whittaker model}.

\subsection{Majorization of Whittaker functions}\label{A:majorization of Whittaker functions}
We prove \eqref{E:estimate Whittaker function} when $F=\R$ or $\C$. We begin with a lemma.

\begin{lm}\label{L:estimate modified bessel function}
Let $s\in\C$ with $|{\rm Re}(s)|=r$. Let $r_1>0$ such that $r_1\geq r$. 
Let $\varphi$ be a function on $\R^{\x}$ defined by $\varphi(y)=y^{r_1}K_s(|y|)$. Then we have
\[
\varphi(y)\ll_{r,r_1} {\rm exp}\left(-\frac{1}{2}|y|\right).
\]
\end{lm}
\begin{proof}
Since $K_s(z)=K_{-s}(z)$ and $K_s(|y|)\ll K_{{\rm Re}(s)}(|y|)$, we may assume that $s=r$. 
Then the lemma follows immediately from the asymptotic form of $K_r(|y|)$. Indeed, when 
$|y|\to\infty$, we have
\[
K_r(|y|)\sim\sqrt{\frac{\pi}{2|y|}}e^{-|y|},
\]
while when $0<|y|\ll 1$, we have
\[
K_r(|y|)\sim
\begin{cases}
-{\rm ln}\left(\frac{|y|}{2}\right)-\gamma
\quad&\text{if $r=0$},\\
\frac{\Gamma(r)}{2}\left(\frac{2}{|y|}\right)^r
\quad&\text{if $r>0$}.
\end{cases}
\]
Here $\gamma$ stands for the Euler's constant.
This finishes the proof.
\end{proof} 

\begin{prop}\label{P:estimate whittaker function}
Suppose that $\pi$ is unitary. Let $\epsilon>0$ and $W\in\cW(\pi,\psi)$. Then
\[
W\left(\pMX y001 h\right)
\ll_{\pi,W,\epsilon}
|y|_F^{\frac{1}{2}-\lambda(\pi)-\epsilon}
{\rm exp}\left(-d|y|_F^\frac{1}{d}\right),
\]
for every $h\in K$. Here $d=[F:\R]$.
\end{prop}

\begin{proof}
Suppose that $F=\C$ and $\pi\cong\pi(\mu,\nu)$ with $\mu\nu^{-1}(y)=|y|_\C^{2s}(y/|y|_\C^{\frac{1}{2}})^k$ for some 
$s\in\C$ and $k\in\Z$. By the right $K$-finiteness, it suffices to prove the assertion when $h$ is the identity element.
By \lmref{L:kirillov model for archimedean}, we may assume
\[
W\left(\pMX y001\right)
=
\mu(y)|y|_{\C}^{\frac{1}{2}-s-\frac{a+b-c-d}{4}}y^{a}\cdot \b{y}^{b}
\cdot
K_{2s+\frac{a+b-c-d}{2}}\left(4\pi |y|_{\C}^{\frac{1}{2}}\right).
\]
Here $a,b,c,d$ are non-negative integers. Note that $|\mu(y)|=|y|_\C^{{\rm Re}(s)}$ by \eqref{E:unitary character}. 
Let $|y|_\C=r^2$ for some $r>0$.
We have 
\[
W\left(\pMX y001\right)
\ll 
r^{1+\frac{a+b+c+d}{2}}K_{2{\rm Re}(s)+\frac{a+b-c-d}{2}}(4\pi r).
\]
Since $\pi$ is unitary, we have $\lambda(\pi)=|{\rm Re}(s)|$. In particular, $\lambda(\pi)-{\rm Re}(s)\geq 0$.
Let $\epsilon>0$. We have 
\begin{align*}
r^{1+\frac{a+b+c+d}{2}}K_{2{\rm Re}(s)+\frac{a+b-c-d}{2}}(4\pi r)
&=
r^{1-2\lambda(\pi)-2\epsilon}\cdot r^{2\lambda(\pi)+\frac{a+b+c+d}{2}+2\epsilon}
K_{2{\rm Re}(s)+\frac{a+b-c-d}{2}}(4\pi r)\\
&\ll_{a,b,c,d,\lambda(\pi),\epsilon}
r^{1-2\lambda(\pi)-2\epsilon}e^{-2r}
\end{align*}
by \lmref{L:estimate modified bessel function}. This proves the case where $F=\C$.

Suppose that $F=\R$. If $\pi$ is a discrete series representation, then the assertion follows from 
\propref{P:test vector in Whittaker model} $(1)$. 
If $\pi$ is a principal series representation, then it can be proved in a similar way.
\end{proof}

\bibliographystyle{alpha}
\bibliography{ref}
\end{document}